%% file: sing.tex
\documentclass[graybox,envcountsect,envcountsame]{svmult}
\usepackage{adjustbox}
\usepackage{multirow}
\usepackage{type1cm}         
\usepackage{makeidx}         
\usepackage{graphicx}        
\usepackage{multicol}        
\usepackage[bottom]{footmisc}
\usepackage{newtxtext}       
\usepackage{newtxmath}       
\usepackage{url}
\usepackage[all]{xy}
\usepackage{graphicx}
\usepackage[pagebackref,bookmarksopen]{hyperref}
\hypersetup{
	colorlinks=true,
	linkcolor=blue,
	citecolor=magenta,
	urlcolor=cyan,
}
\usepackage{bookmark} 

\makeatletter
\providecommand*{\toclevel@titlech}{0} 
\edef\toclevel@authorch{\the\numexpr\toclevel@titlech+1} 
\makeatother

 \DeclareMathOperator*{\cod}{cod}

\makeatletter                                       
\newenvironment{chapquote}[3][35pt]
  {\setlength{\@tempdima}{#1}
    \ifx\relax#2\relax\setlength{\@tempdimb}{0pt}\else\setlength{\@tempdimb}{#2}\fi
    \def\chapquote@author{#3}
   \parshape 1 \@tempdima \dimexpr\textwidth-\@tempdima-\@tempdimb\relax
   \itshape}
  {\newline\par\normalfont\hfill--\ \chapquote@author\hspace*{\@tempdimb}\par\bigskip}
\makeatother

 \newcommand{\widesim}[2][1.5]{
   \mathrel{\underset{#2}{\scalebox{#1}[1]{$\sim$}}}}
 
\begin{document}

\title*{Old and new results on density of stable mappings}
\titlerunning{Results on density of stable mappings}
\author{Maria Aparecida Soares Ruas}
\authorrunning{M. A. S. Ruas}
\institute{Maria Aparecida Soares Ruas \at Instituto de Ci\^encias
Matem\'aticas e de Computa\c c\~ao - USP, Av. Trabalhador s\~ao-carlense,
400, 13566-590 - S\~ao Carlos - SP, Brazil, \email{maasruas@icmc.usp.br}}
\maketitle

\begin{chapquote}{}{C.T.C.Wall\cite{Wal81}}
  The analysis of the conditions for a map-germ to be finitely
  determined and of the degree of determinacy involves the most important
  of the local aspects of singularity theory. 
\end{chapquote}

\abstract{Density of stable maps is the common thread of this
  paper. We review Whitney's contribution to singularities of
  differentiable mappings and Thom-Mather theories on $C^{\infty}$
  and $C^{0}$-stability. Infinitesimal and algebraic methods are
  presented in order to prove Theorem A and Theorem B on density of
  proper stable and topologically stable mappings $f:N^{n}\to P^{p}.$
  Theorem A states that the set of proper stable maps is dense in the
  set of all proper maps from $N$ to $P$, if and only if the pair
  $(n,p)$ is in \emph{nice dimensions,} while Theorem B shows that
  density of topologically stable maps holds for any pair $(n,p).$  A
  short review of results by du Plessis and Wall on the range in
  which proper smooth mappings are $C^{1}$- stable is given.
  A Thom-Mather map is a topologically stable map $f:N \to P$ whose \index{Map!Thom-Mather}
  associated $k$-jet map $j^{k}f:N \to P$ is transverse to
  the Thom-Mather stratification in $J^{k}(N,P).$ We give a detailed
  description of Thom-Mather maps for pairs $(n,p)$ in the boundary
  of the nice dimensions.
  The main open question on density of stable mappings is to determine
  the pairs $(n,p)$ for which Lipschitz stable mappings are
  dense.  We discuss recent results by Nguyen, Ruas and Trivedi  on
  this subject, formulating conjectures for the density
  of Lipschitz stable mappings in the boundary of the nice
  dimensions. At the final section, Damon's results relating
  $\mathcal{A}$-classification  of map-germs and $\mathcal{K}_{V}$
  classification of sections of the discriminant $V=\Delta(F)$ of a
  stable unfolding of $f$ are reviewed and open problems are discussed.
}

\section*{Contents}
\setcounter{minitocdepth}{1}
\dominitoc

\section{Introduction}
\label{sec:introduction}

Although Riemann, Klein, Poincar\'e and other great mathematicians of
the nineteenth century  already used deep topological concepts in their work,
the birth of algebraic and differential topology as formal sub-areas
of Mathematics occurred in the first half  of the twentieth century. 

After previous works of Whitehead, Veblen and others, the American
mathematician Hassler Whitney introduced fundamental concepts and
proved strong results in differential topology such as the well known
\emph{strong Whitney
  embedding theorem} and \emph{weak Whitney embedding
  theorem}. The first one states that any smooth real $m$-dimensional
manifold can be smoothly embedded in $\mathbb R^{2m}$, while the
latter says that any continuous mapping of an $n$-dimensional
manifold to an $m$-dimensional manifold may be approximated by a
smooth embedding provided that $m > 2n.$ Furthermore, replacing
embedding by immersion in this last  statement the result holds for
all $m\geq 2n.$  His survey paper \emph{Topological
  properties of differentiable manifolds} published in 1937
\cite{Whi37} contains many contributions he made in those early years
of differential topology. 

In 1944, Whitney \cite{Whi44} studied the first pair of dimensions not
covered by his immersion theorem. For mappings $f$ from $\mathbb
R^{n}$ to $\mathbb R^{2n-1}$ Whitney proved that singularities cannot
be avoided in general. He introduced the \emph{semi regular mappings}
as proper mappings $f:\mathbb R^{n}\to \mathbb R^{2n-1}$ whose only
singularities are the generalized cross-caps (Whitney umbrellas)
points.  Away from singular
points, $f$ is an immersion with transverse
double points, and when $n=2,$ a finite number of triple points may also
appear in the image of  $f.$ These are the only stable singularities in these
dimensions. However, only later, Whitney introduced the notion of
stable mappings.

Abstract spaces and their topological properties were known by then,
so that the notion of \emph{stability} of systems and mappings appeared
naturally. It appeared first in dynamical systems, introduced by
A. Andronov and L. Pontryagin \cite{AndPon} for a class of autonomous
differential systems on the plane, under the name of ``syst\`emes
grossiers''. {The term ``structural stability'' appears in the english
language edition of the book by Andronov and Chaikin, edited under the
direction of Solomon Lefschetz in 1949 \cite{AndCha} (see also
\cite{Sot}). It also appears in other pioneering papers on the
subject, among them the paper \emph{On structural stability} by
Mauricio Peixoto \cite{Pei}, published in 1959.}

The notion of stable mappings was formulated by Whitney in
\cite{Whi58} around the middle of last century. He characterized stable
mappings  from $\mathbb R^{n}$ to $\mathbb R^{p}$  with $p \geq
2n-1$ in \cite{Whi43} and stable mappings from the plane into the
plane in \cite{Whi55}, showing in these cases  that stable mappings
form  a dense set in the space of smooth proper mappings.

The article Whitney \cite{Whi55} published in 1955 is a landmark, considered
by many to be the cornerstone of the theory of singularities. The
stable singularities of mappings of the plane into the plane are folds
and cusps and any proper smooth mapping $f:\mathbb R^{2}\to \mathbb
R^{2}$ can be approximated by a stable mapping. Whitney
conjectured that density of stable mappings would hold for any
pair $(n,p).$  However 
Ren\'e Thom showed, in his 1959 lecture at Bonn, that this is not the
case by given an example of a map $f:\mathbb R^{9}\to \mathbb R^{9}$
that appears generically in a $1$-parameter family of maps.

Thom conjectured that the topologically stable maps are always dense
and gave an outline of the proof. The complete proof was given  by John
Mather,  who from 1965 to 1975, solved almost completely the program drawn by Ren\'e Thom for
the problem of stability.

Mather found several characterizations of stability and proved that
the set $\mathcal S^{\infty}(N,P)$ of stable mappings is dense in the
set $C_{pr}^{\infty}(N,P)$ of smooth proper mappings, from the $n$-dimensional manifold
$N$ to the $p$-dimensional manifold $P,$ if and only if $(n,p)$ is in the \emph{nice
  dimensions,} which he completely characterized in
\cite{Mat71}. Based on Thom's ideas, he  also proved in \cite{Mat71-3,
  Mat76} that the set of topologically  stable 
 mappings $\mathcal S^{0}(N,P)$ in $C_{pr}^{\infty}(N,P)$ is residual for all pairs $(n,p).$

 The 70's was blooming period for singularity theory. Along with
 Mather's work, Ren\'e Thom's book on catastrophe theory \cite{Tho72}
 and Arnold's  seminal classification of simple singularities of
 functions \cite{Arn} also had a great impact.
 These works paved the intense development of the theory
 of the following decades.  
The deep understanding of stable
mappings, versal unfoldings and finite determinacy transformed
singularity theory into an organizing center for several areas of
mathematics and sciences.

The common thread of these notes is the question of density of stable
mappings in  $C^{\infty}_{pr}(N,P).$  We outline the solutions of the various formulations of this
problem: $C^{\infty}, C^{0}$ and $C^{l}, 1\leq l < \infty$
stability. The remaining  open problem in this setting is density
of Lipschitz stable mappings. Recent
progress in the solution of this problem  appear in \cite{RuaTri, NguRuaTri}.

We give an account of tools for the proofs of the main theorems
including  the notion of infinitesimal stability, the generalized
Malgrange's theorem, Thom's transversality theorem, mappings of finite
singularity type and finite determinacy of Mather's groups. Whitney
and Thom's results on stratified sets and maps are fundamental
pieces of the theory.
For an account of these topics we refer to David Trotman's article in
Volume 1 of this Handbook .

In these notes we concentrate on the discussion of real
singularities. The infinitesimal methods discussed here also hold true
for holomorphic mappings. For an account on Mather's theory of
${\mathcal A}$-equivalence  and the description of the topology of
stable perturbations of ${\cal A}$-finitely determined holomorphic
germs the reader may consult the notes by David Mond and Juan Jos\'e
Nu\~no-Ballesteros in this Handbook \cite{Han}.

Related topics to those discussed in these notes, as well as new
developments of the theory, are given in the subsections \emph{Notes} at the end of
each section. The final section includes a discussion of open problems
in the theory of singularities of smooth mappings.

\section{Setting the problem}

\label{sec:dens-stable-mapp}

Let $C^{\infty}(N,P)=\{f:N\to P, f\in C^{\infty}\}$ be the set of
  smooth mappings from $N$ to $P,$ where $N$ and $P$ are smooth
  manifolds of dimension $n$ and $p$ respectively. The topology on
  $C^{ \infty}(N,P)$  is the $C^{\infty}$-Whitney topology.

 We  review here the contributions of
  singularity theory to solve the following problem.

  \begin{problem}\label{pro}
      Find an open and dense set $\mathcal S$ in $C^{\infty}(N,P)$ and
  describe all singularities of mappings $f\ \in \mathcal S.$
\end{problem}

The relevant equivalence is $\mathcal{A}$-equivalence.

\begin{definition}\label{aequi} \index{$\mathcal A$-equivalence}
    Two smooth maps $f,g:N \to P$ are $\mathcal
    A$-\emph{equivalent} if there
  exist $C^{\infty}$ diffeomorphisms $h:N \to N$ and $k:P \to P$ such
  that the following diagram commutes
$$
  \xymatrix{\ar @{} [dr] |{\circlearrowleft}
N \ar[d]_h \ar[r]^f &P  \ar[d]^k \\
N \ar[r]_g   &      P }
$$
\end{definition}
\begin{definition}\label{asta}\index{Map!stable}
  The map $f:N \to P$ is \emph{stable} ($\mathcal A$-\emph{stable}) if there exists a neighborhood $W$ of $f$ in
  $C^{\infty}(N,P),$ such that  $g \widesim{\mathcal{A}}
  f$ for every $g\in W.$
\end{definition}

  Replacing $C^{\infty}$-diffeomorphisms by homeomorphisms,  $C^{l}$-diffeomorphisms, $l>0$ or bi-Lipschitz
  homeomorphisms in definitions \ref{aequi} and  \ref{asta} we get respectively the definitions of $C^{0}$-$\mathcal
 A$, $C^{l}$-$\mathcal A$ ($l> 0$), \emph{bi-Lipschitz}-$\mathcal A$ \emph{equivalences}
 and of \emph{topological
 stability}, $C^{l}$-\emph{stability, or Lipschitz stability of maps} in
 $C^{\infty}(N,P).$
\index{$C^{0}$-$\mathcal A$-equivalence}
\index{$C^{l}$-$\mathcal A$-equivalence}
\index{Bi-Lipschitz! $\mathcal A$-equivalence}
\index{Map!topologically stable}
\index{Map!Lipschitz stable}
\index{Map!$C^{l}$- stable}

Before starting the discussion of Problem \ref{pro}, we review some notation
and definitions.

The Whitney $C^{\infty}-$ topology in $C^{\infty}(N,P)$ was defined by
John Mather in \cite{Mat68-1}. We review it here (more details can be
found in the book of Golubitsky and Guillemin \cite{GolGui}).

For $x\in N,$ $y\in P$ and for a non-negative integer $k,$ we denote by
$J^{k}(N,P)_{x,y}$ the set of $k$-jets of map-germs $(N,x)\to
(P,y).$ When $N=\mathbb R^{n},$ $P=\mathbb R^p,$ we denote
$J^{k}(n,p)$ the set of polynomial mappings $f:\mathbb R^n \to
\mathbb R^p$ of degree $\leq k,$ such that $f(0)=0.$

The set $J^{k}(N,P)=\bigcup_{x\in N,y\in P} J^{k}(N,P)_{x,y}$ is the $k$-jet
space of mappings from $N$ to $P.$ \index{Jet space of mappings}
The set
$J^{k}(N,P)$ is a smooth manifold (theorem 2.7 in
\cite{GolGui}). Moreover, it has the structure of a fibre bundle
with  basis  $N\times P.$


Let $U$ be an open set in $J^{k}(N,P)$ and 
$$
M(U)=\{ f\in C^{\infty}(N,P) |\; j^{k}f(N)\subset U \}.
$$
The family of sets $\{M(U)\}$ where $U$ is an open set of $J^{k}(N,P)$
is a basis for a topology in $C^{\infty}(N,P)$ (note that $M(U)\cap
M(V)=M(U\cap V)$).
This topology is called the Whitney $C^{k}$-topology.\index{Whitney! $C^{k}$-topology}

Denote by $W_{k}$ the set of open subsets of $C^{\infty}(N,P)$ in the
Whitney $C^{k}$-topology. The Whitney $C^{\infty}$-topology is the
topology whose basis is $W=\cup_{k=0}^{\infty} W_{k}.$
\index{Whitney! $C^{\infty}$-topology}

Given a metric $d$ on $J^{k}(N,P),$ compatible with
its topology and a nonnegative continuous function $\delta: N\to
\mathbb R$ we can define a basic neighborhood of $f\in C^{\infty}(N,P)$ as
follows 
$$
B_{\delta}(f)=\{g \in C^{\infty}(N,P) |\;
d(j^{k}f(x),j^{k}g(x))<\delta(x), \forall x \in N  \}.
$$ 


When $N$ is compact, $f_{n}$ converges to $f$ in the Whitney
$C^{k}$-topology if and only if $j^{k}f_{n}$ converges uniformly to
$j^{k}f.$  On noncompact manifolds  $f_{n}$ converges to $f$ in the Whitney
$C^{k}$-topology if and only there exists a compact $K\subset N,$ such
 $j^{k}f_{n}$ converges  to $j^{k}f$ uniformly in $K,$ and there
 exists $n_{0}$ such that $f_{n}\equiv f$ in $N\setminus K$ for any $n\geq
 n_{0}$ (for details see the book by Golubitsky and Guillemin \cite{GolGui}). 

Thus we can see that there is a great difference in the Whitney
topology depending on whether or not the domain $N$ is a compact
manifold.

When $N$ is not compact, the Whitney $C^{k}$-topology is a very fine
topology, with many open  sets. As a consequence, dense sets in
$C^{\infty}(N,P)$ are very large sets, and theorems characterizing
these sets in $C^{\infty}(N,P)$ are strong results.

\subsection[The work of Hassler Whitney]{The work of Hassler Whitney: from 1944 to 1958}
\label{sec:from-1944-1958:the}

The foundations of the theory were Whitney's work, in which he formulated
the problem of classifying singularities that can not be eliminated by
small perturbations, and completely succeeded  in solving it for maps
from $\mathbb R^{n}$ to $\mathbb R^{p}$ with $p\geq 2n-1$ in Whitney \cite
{Whi43} and  from $\mathbb R^{2}$ to $\mathbb R^{2}$  in Whitney \cite{Whi55}. 

The article \cite{Whi55} published in 1955 is a magnificent work
dedicated to maps from the plane into the plane. In the introduction to the article, Whitney presents a
complete review of the existing results and future perspectives of the
theory. We reproduce it here:
\itshape
``Let $f_{0}$ be a
mapping of an open set $R$ in $n$-space $E^{n}$ into $m$-space
$E^{m}.$ Let us consider, along with $f_{0},$ all the mappings $f$
which are sufficiently good approximations to $f_{0}.$ By the
Weierstrass Approximation Theorem, there are such mappings $f$ which
are analytic; in fact, (see \cite[Lemma1]{Whi34}) we may make $f$
approximate to $f_{s}$ throughout $R$ arbitrarily well, and if
$f_{0}$ is $r$-smooth, (i.e., has continuous partial  derivatives of
order $\leq r$),  we may make corresponding derivatives of $f$
approximate those of $f_{0}.$

Supposing $f$ is smooth, (i.e., $1$-smooth), the Jacobian matrix $J(f)$
of $f$ is  defined (using fixed coordinate systems); we say the point
$p\in R$ is a \emph{regular point} or \emph{singular point} of $f,$
according as $J(f)$ is of maximal rank (i.e., of rank $\min(n,m)$) or
lesser rank. In general we cannot expect $f$ to be free of singular
points. A fundamental problem is to determine what sort of
singularities any good approximations $f$ to $f_{0}$ must have; what
sort of sets they occupy, what $f$ is like near such points, what
topological properties hold with references to them, etc.

Some special cases of this problem have been studied as follows:

\noindent a) For $m=1,$ we have a real valued function in $R.$ It was
shown by M. Morse in Theorem 1.6 of \cite{Mor31}, that $f$ may be chosen
so that the singular points (called critical points here) are
isolated, the ``Hessian'' being non-zero at each.'' ``Moreover, each
critical point may be assigned a ``type number''; topological relations
among these were given by Morse \cite{Mor25}.

\noindent b) If $m\geq 2n,$ we may find an $f$ with no singular
points; see (a) and (b) of Theorem 2 in \cite{Whi34}.

\noindent c) If $m=2n-1,$ we may obtain an $f$ with singular
points: see \cite{Whi43}. For each such point $p\in R,$ coordinate
systems $(x_{1}, x_{2}, \dots, x_{n})$ in $E^{n}$ and $(u_{1}, u_{2},
\dots, u_{m})$ in $E^{m}$ may be chosen, in which $f,$ near $p,$ has
the form

{$$
\ u_{1}= x_{1}^{2},\ \ \ \, u_{i}= x_{i},\,\,\,\,\,
u_{n+i-1}=x_{1}x_{i},\ \ \  (i=2,
\dots, n).
$$}

The singularities are studied from a topological point of view in
\cite{Whi44}.

\noindent d) Some beginnings have been made for the other pairs of
values $(n.m)$ by N. Wolfsohn, \cite{Wolf}, but no complete
classification of the singularities exist in these cases. Thus the
smallest pair of values for which the problem is open is the pair
$(2,2),$ i.e for mappings of the plane into the plane; it is this case
that we treat here. In this case, there can be ``folds'' lying along
curves and isolated ``cusps'' on the folds ''\upshape(Figure \ref{fig:foldsandcusps}).

We review Whitney's results in this section.

Let $f:U\to \mathbb R^{2}$ be a smooth mapping defined on the open set
$U\subset\mathbb R^{2}.$ With coordinates systems $(x,y)$ in $U$ and
$(u,v)$ in the target, the Jacobian of $f$ is given by
$$
J(f)=u_{x}v_{y}-u_{y}v_{x}.
$$
\index{Point!regular}\index{Point!singular}
{A point $p\in U$ is \emph{regular} or \emph{singular} according as
$J(f)(p)\neq 0$ or $J(f)(p)= 0.$ A singular point $(x_{0},y_{0})$ is
\emph{good} if the derivatives $\frac{\partial J(f)}{\partial
  x}(x_{0},y_{0})$ and $\frac{\partial J(f)}{\partial
  y}(x_{0},y_{0})$ do not vanish simultaneously. We say that $f$ is
\emph{good} if every singular point of $f$ is good. This condition
implies that the set $S(f)$ of singular points of $f$ is a regular
curve.}
\index{Map!good}
If $f$ is good and $p$ is a singular point, let $\phi:(-\epsilon,
\epsilon) \to \mathbb R^{2}$ be a parametrization of the singular set
$\mathcal{S}(f)$ in a neighborhood of $p\in \mathcal{S}(f)$ such that
$\phi(0)=p.$ Then, we define
\index{Point!fold}
\noindent (i) If $(f\circ\phi)'(0) \neq 0,$ we say  $p$ is
\emph{fold point } of $f.$

\noindent (ii) $(f\circ\phi)'(0)= 0$ and $(f\circ\phi)''(0) \neq 0,$
we say  $p$ is a \emph{cusp point} of $f.$
\index{Point!cusp}
These definitions are independent of the parametrization chosen for
$S(f)$ in a neighborhood of $p.$ 

One can easily see that at a fold point, the restriction of $f$ to its
singular set is non singular, while a cusp point is a singular point  of this restriction.

It follows from the definition that cusp points are isolated.

\begin{definition}[Whitney \cite{Whi55}, p. 379]
Let $f$ be a good map.\index{Map!excellent}
  We say that $p$ is an \emph{excellent point} of $f$ if it is a
  regular, fold or cusp point of $f.$ If  each point $p\in U$ is
  excellent we say $f$ is \emph{excellent.}
\end{definition}


Any smooth map can be approximated in
the $C^{r}$-Whitney topology, $r\geq 3,$ by an excellent map.

\begin{theorem}[Whitney \cite{Whi55}, Theorem 13A ]
  Let $f_{0}$ be a mapping from $U\subset \mathbb R^{2}$ to $\mathbb
  R^{2},$ where $U$ is an open set in $\mathbb R^{2}.$ Then
  arbitrarily near $f_{0}$ there is an excellent mapping $f.$ If
  $f_{0}$ is $r$-smooth and $\epsilon$ is a positive continuous
  function in $U,$ we make $f$ an $(r,\epsilon)$-approximation of $f.$ 
\end{theorem}

Prior to Thom's transversality theorem (\cite{Tho56}), Whitney
introduced the method of characterizing in the jet space the set of
jets with degenerate singularities, the so called ``bad set''.

In addition,  methods of producing generic $C^{r}$-perturbations of any
given map were also introduced by him. The goal was to find
sufficiently close perturbations that would avoid the bad set.

For polynomial maps from the plane into plane, the bad set are the
polynomial maps admitting singularities more degenerate than folds and
cusps.

Folds and cusps have simple normal forms.

\begin{figure}
  \centering
  \scalebox{0.60}{\includegraphics {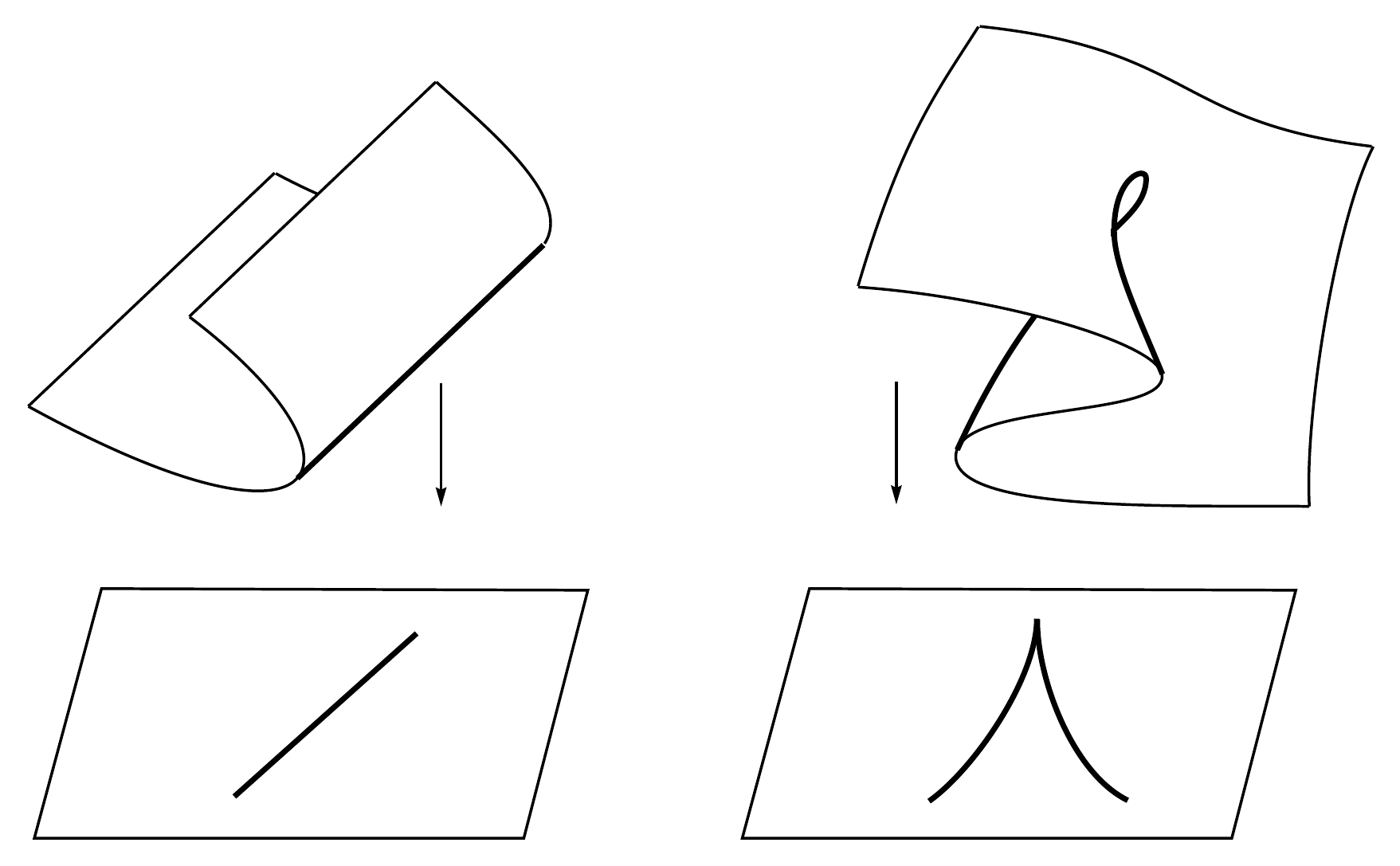}}  
  \caption{Folds and Cusps}\label{fig:foldsandcusps}
\end{figure}

\begin{theorem}[Whitney \cite{Whi55}, Theorems 15A and 15B ]
  \begin{enumerate}
  \item Let $p$ be a fold point of the $r$-smooth mapping $f$
  of  $\mathbb R^{2}$ into $\mathbb R^{2},$ with $r\geq 3.$ Then
  $(r-3)$-smooth coordinate systems $(x,y),\, (u,v)$ may be introduced
  about $p$ and $f(p)$ respectively, in terms of which $f$ takes the
  form
  \begin{equation}
    \label{eq:1.1}
      u=x^2\, ,v=y
    \end{equation}
    
  \item Let $p$ be a cusp point of the $r$-smooth mapping $f$
  of  $\mathbb R^{2}$ into $\mathbb R^{2},$ with $r\geq 12.$ Then
  $(\frac r2-5)$-smooth coordinate systems $(x,y),\, (u,v)$ may be introduced
  about $p$ and $f(p)$ respectively, in terms of which $f$ takes the
  form
  \begin{equation}
    \label{eq:1.2}
      u=xy -x^3\, ,v=y
    \end{equation}
  \end{enumerate}
\end{theorem}

  While the proof of (1) is not hard, Whitney's proof of the normal
  form in a neighborhood of a cusp point $p$ follows by an ingenious
  sequence of changes of coordinates in the source and
  target. The tool is essentially the implicit function theorem.

Today, there are simpler proofs of this result, based on current tools
of singularity theory: see for instance, Theorem 2.4, Chapter VI in 
Golubitsky and Guillemin's book \cite{GolGui} or Example 3.6 in Mond
and Ballesteros \cite{MonNun}.

  The notion of stable mappings is due to Whitney. In order to
  characterize them, in addition to the local behavior of stable
  singularities, it is necessary to explain the behavior of multiple
  points. For maps from the plane into the plane the following holds.

  \begin{theorem}
    Let $f:N^{2}\to P^{2}$ be a smooth map, $N$ and $P$
    $2$-dimensional manifolds, $N$ compact. Then $f$ is $C^{\infty}$-
    stable if and only if the following conditions hold.
    \begin{enumerate}
    \item $f$ is excellent and hence $\mathcal{S}(f)$ is a
      regular curve, with at most a finite number of cusp points.
    \item  If $p_{1}$ and $p_{2}$ are singular points of $f,$
      $f(p_{1})=f(p_{2}),$ then $p_{1}$ and $p_{2}$ are not cusp
      points. Moreover the fold lines intersect transversaly at
      $f(p_{1})=f(p_{2}).$ 
    \item  The restriction of $f$ to $\mathcal{S}(f)$ has no
      triple points.
    \end{enumerate}
    \end{theorem}

    Whitney  formulated in \cite{Whi58} a general approach to defining a
stratification in jet space and to define locally generic mappings as
those  whose $r$-jets were transversal to the strata of the
stratification, for every $r \in \mathbb N^{*}.$ The article contains
an explicit description of generic singularities for pairs $(n,p)$
such that $n,p\leq 5.$


He asked the question whether for any pair of dimensions
$(n,p),$ the stable maps could be characterized by transversality to a
finite collection of submanifolds in jet space, so that one could
apply Thom's transversality theorem to prove that a smooth map could
be always approximated by stable maps.

However, in a course taught at the University of Bonn in 1959, Ren\'e
Thom showed with  an example that
it is not always possible to approximate a given map by $C^{\infty}$
stable mappings (See section \ref{sec:thoms-example}, on Thom's
example). In fact, in the notes \emph{ Singularity of differentiable  
mappings I,} written by Harold Levine \cite{Tho68}, Thom sketched
the proof that  $C^{2}$-stable mappings do not form an open set in
$C^{\infty}(N,P),$ when $n=p=9$ and  he formulated conjectures that
promoted a great development in the
theory in the following decades. In particular, Thom conjectured the density of topologically
stable mappings, proved by John Mather in 1971. We discuss Ren\'e Thom and John
Mather's contributions in the next section.

\subsection[From 1958 to 1970]{Ren\'e Thom and John Mather: from 1958 to 1970}
\label{sec:from-1958-1970}

We start by reviewing the subjects covered by R. Thom in his course at
the University of Bonn. H. Levine's notes are divided into three
chapters.

Chapter I, named ``Jets'' introduces the notion of jet spaces, the
action of the group ${\mathcal A}$ in jet space and ${\mathcal
  A}^{r}$-invariant manifolds, denominated, in the notes,
\emph{critical varieties} in $J^{r}(n,p).$ The set ${\cal S}_{k}$
of $1-$jets of corank $k$ and its topological closure $\overline {\cal S}_{k}$
in $J^{1}(n,p)$ were defined.

In Chapter II, entitled ``Singularities of mappings'', Thom's
transversality theorem was stated and proved. We remark however that
the topology in the space of mappings in Thom's proof was the weakest topology
making the mapping
\begin{align*}
  j^{r}:C^{\infty}(N,P) &\to C^{\infty}(N,J^{r}(N,P))\\
 f &\to j^{r}f
\end{align*}
continuous. The topology in the second space was the compact open
  topology. The transversality theorem in
\cite{Tho68} was stated as follows:
For $s>r\geq 0,$ let $W$ be a codimension $q,$ $C^{s-r}$ submanifold
of $J^{r}(N,P),$ $s-r> \dim N -q.$ Then the set of mappings $f\in
C^{\infty}(N,P),$ such that $j^{r}f \pitchfork W$ is dense in
$C^{\infty}(N,P).$
\index{Thom!transversality theorem}
 The notion of second order
 singularities  $S_{h,k}$ in $J^{2}(n,p)$ was introduced. These sets are connected to
the singular points $S_{h}\subset J^{1}(n,p)$ by the relation: if
$j^{1}f \pitchfork S_{k},$ then
$(j^{2}f)^{-1}(S_{k,h})=S_{h}(S_{k}(f)).$ The general definition of
the singular varieties $S_{k_{1},\dots,k_{r}}\subset J^{r}(N,P),$
introduced in \cite{Tho68} was
better formulated by J.M. Boardman, in 1967, in \cite{Boa}. Mather's
account in \cite{Mat71-2} is the clearest.

\begin{remark}
  In the following sections the sets $S_{k}$ and $S_{k,h}$ will be
  denoted by $\Sigma^{k}$ and $\Sigma^{k,h},$ respectively.
\end{remark}

In Chapter III, ``Equivalence and stability'', Thom formulated the
problem of characterizing singularities determined by their
jet of some order.
The name \emph{finitely determined germs,} was later given by John
Mather \cite{Mat68-2}, who also gave 
necessary and sufficient conditions for finite determinacy. The notion
of $C^{s}$-stable mappings and the example illustrating that $C^{2}$
stable mappings are not dense when $n=p=9$ were discussed in that
chapter.

The notion of homotopic stability was also introduced. A mapping $f:N\to P$
is \emph{homotopically stable} if for every homotopy $F:N\times I \to
P$ of $f,$ there exist $t_{0}$ and homotopies of diffeomorphisms
$\phi_{t}:N\to N,$ $0\leq t \leq t_{0},$  $\psi_{t}:P \to P,$ of
$1_{N}$ and $1_{P}$ such that $F_{t}=\psi_{t}\circ f \circ \phi_{t},\
t<t_{0}.$

The program for the theory of stable mappings originated from the
contributions of Whitney and Thom consisted of finding pairs of
dimensions $(n,p),$ for which there exists a set of mappings
$\mathcal{S}\subset C^{\infty}(N^{n},P^{p}),$ with the following
properties:
\begin{enumerate}
\item $\mathcal{S}$ is a residual set in $C^{\infty}(N^{n},P^{p}),$
\item The maps $f\in \mathcal{S}$ are $C^{\infty}$-stable,
\item There exists a finite number of polynomial normal forms such that
  every singular point of $f\in\mathcal{S}$ is equivalent to a
  normal form in this list.
\end{enumerate}

In a memorable series of six articles from 1968 to 1971,
John Mather found several characterizations of stability and provided
theorems answering almost completely the question of density of stable
maps.

The main results on density of stable mappings are stated below. The
proofs are based on ideas of Ren\'e Thom  developed by Mather in the
sequence of papers, on Stability of $C^{\infty}$-mappings, I to VI,
\cite{Mat68-1,Mat69-1, Mat68-2, Mat69-2, Mat70, Mat71} and
\cite{Mat12, Mat71-3, Mat76}. In 
these notes we review the main steps leading to the proofs of Theorems
A and B.

Let $C^{\infty}_{pr}(N,P)$ be the set of proper smooth mappings
$f:N\to P.$

\vspace{.8em}
\noindent\textbf{Theorem A}\ \textbf{(Density of stable mappings in the
  nice dimensions, Mather \cite{Mat71,Mat69-2})}\label{th:4}
\emph{  The set $\mathcal{S}^{\infty}(N,P)$ of proper stable mappings $f:N\to
  P$ is dense in $C_{pr}^{\infty}(N,P)$ if and only if $(n,p)$ is in the nice dimensions.}
\vspace{.8em}

See section \ref{sec:nice-dimensions} for the definition of the nice dimensions.

\vspace{.8em}
\noindent\textbf{Theorem B}\ \textbf{(Density of topologically stable
  mappings, Mather \cite {Mat71-3, Mat76})}\label{th:4.2}
\emph{The set $\mathcal{S}^{0}(N,P)$  of proper topologically stable mappings is
  dense in $C_{pr}^{\infty}(N,P).$}
\vspace{.8em}

The main tools in the proofs of theorems \textbf{A} and \textbf{B} are the notion of infinitesimal
stability, Thom's transversality theorem, the generalized
Malgrange theorem, 
the notions of mappings of finite singularity type and contact
equivalence, finite determinacy and unfoldings of Mather's groups, properties of Whitney stratified sets and Thom's 
isotopy theorems. Such notions and results form the framework of the
theory of singularities of differentiable mappings.

We organize the contents of the next sections as follows.

In section \ref{sec:equiv-noti-stab} we introduce infinitesimally
stable and transverse stable mappings. The main goal of the section is
to discuss theorem \ref{th:3.10}  which establishes the equivalence
between these notions and stable mappings.

Section \ref{sec:finite-determ-math}  gives a short presentation of
the infinitesimal machinery of singularity theory. We introduce the
contact group $\mathcal{K}$ defined by Mather as a tool to classify
stable singularities. For Mather's groups $\mathcal{G}=\mathcal{R},
\mathcal{L}, \mathcal{A}, \mathcal{C}$ and $\mathcal{K}$ we define 
$\mathcal{G}$-finitely determined germs and prove the Infinitesimal
Criterion for $\mathcal{G}$-determinacy. We finish the section with a
discussion of maps of finite singularity type (FST), a global
version of $\mathcal{K}$-finitely determined germs, which plays a
central role in the proof of theorem B.

In section \ref{sec:nice-dimensions} we define the nice dimensions
and give an outline of the proof of theorem A.

Section \ref{sec:thoms-example} gives a detailed presentation of
Thom's example, illustrating  that the set of stable maps in
$C^{\infty}_{pr}(\mathbb R^{9},\mathbb R^{9})$ is not dense.

Section \ref{sec:dens-topol-stable-1} is dedicated to the proof of
density of topologically stable mappings $f:N \to P,$ when $N$ is
compact manifold. The general lines of the proof are discussed, although
the details are omitted. 

Section \ref{sec:bound-nice-dimens} gives a systematic
presentation of the topologically stable singularities in the boundary
of the nice dimensions. Much of the section is well known to experts,
however the organized presentation of the  Thom-Mather stratification in
jet space and the discussion of properties of topologically stable
mappings in these dimensions do not appear in the literature.

The question of the density of Lipschitz stable mappings is still
open.
We report on section \ref{sec:non-trivial-example} some recent results
of Ruas and Trivedi \cite{RuaTri} and Nguyen, Ruas and Trivedi \cite{NguRuaTri} on this subject. 

In section \ref{sec:open-problems}, Damon's results relating
  $\mathcal{A}$-classification  of map-germs and $\mathcal{K}_{V}$
  classification of sections of the discriminant $V=\Delta(F)$ of a
  stable unfolding of $f$ are reviewed and open problems are
  discussed.

\section{Equivalent notions of stability}
\label{sec:equiv-noti-stab}

Mather defined infinitesimally stable mappings
in \cite{Mat68-1}, in order to introduce infinitesimal deformations
of a map as a tool to study stability. The main goal in this
section is to review Mather's result that, for proper mappings, stability and infinitesimal
stability are equivalent notions.

First, we introduce some notation.
Let $C^{\infty}(N)=\{\lambda:N\to \mathbb R\}$ be the ring of smooth
functions defined on the smooth manifold $N.$

We denote by $\Theta_{f}$ the $C^{\infty}(N)$-module of vector
fields along $f,$ defined as follows

$$
\Theta_{f}=\{\sigma:N\to TP|\;\, \pi_{2}\circ \sigma=f\}
$$
where $\pi_{2}:TP\to P$ is the projection of the tangent bundle $TP$
into $P.$

Let  $f^{*}(TP)$ denote the pull-back bundle over $N $ via $f.$ Then
the module $\Theta_{f}$ is the set of sections of this bundle.

Similarly, 
$$
\Theta_{N}=\{\xi:N\to TN|\;\, \pi_{1}\circ \xi=I_{N}\}
$$
is the set of sections of the
tangent bundle of $N,$ and
$$
\Theta_{P}=\{\eta:P\to TP|\;\, \pi_{2}\circ \eta=I_{P}\},
$$
the set of sections of the tangent bundle of $P,$ where $I_{N}$ and
$I_{P}$ are the identities.,

The set $\Theta_{N}$ is a $C^{\infty}(N)$-module, while $\Theta_{P}$ is a
module over the ring $C^{\infty}(P).$

We have the following diagram and homomorphisms

$$
  \xymatrix{TN \ar[d]_{\pi_{1}} \ar[r]^{df} &TP  \ar[d]^{\pi_{2}} \\
N \ar[r]_f\ar[ru]^{\sigma}   &      P }.
$$

\begin{align*}
  tf:\Theta_{N}&\to \Theta_{f}\\
      \xi &\mapsto tf(\xi)
\end{align*}
where $tf(\xi)(x)= df_{x}(\xi(x)),$ 
\begin{align*}
  {\omega}f:\Theta_{P}&\to \Theta_{f}\\
      \eta &\mapsto {\omega}f(\eta)=\eta\circ f
\end{align*}

The map $tf$ is a homomorphism of $C^{\infty}(N)$-modules. The map
$f:N\to P$ induces a ring homomorphism
\begin{align*}
  f^{*}:C^{\infty}(P) &\to C^{\infty}(N) \\
  \phi &\mapsto f^{*}(\phi) = \phi\circ f.
\end{align*}         

We say that the map ${\omega}f$ is a homomorphism over
$f^{*}(C^{\infty}(P))$ (or alternatively a $C^{\infty}(P)$-module
homomorphism via $f$).

Notice that  ${\omega}f(\eta_{1}+ \eta_{2}) = (\eta_{1}+\eta_{2})\circ f={\omega}f(\eta_{1}) + {\omega}f(\eta_{2})$
and  $ {\omega}f(\alpha \eta) = (\alpha\circ f)(\eta\circ f)=(\alpha\circ f) {\omega}f(\eta),$
for any  $\alpha \in C^{\infty}(P)$ and any $\eta_{1}, \eta_{2} \in \Theta_{p}.$

\begin{definition}\index{Map!infinitesimally stable}
The map $f:N \to P$ is \emph{infinitesimally stable} if for any $\sigma \in
\Theta_{f},$ there are sections $\xi \in \Theta_{N}$ and $\eta \in
\Theta_{P}$ such that $\sigma=tf(\xi)+\eta\circ f.$
 Equivalently, we can say that $\Theta_{f}=tf(\Theta_{N})+ {\omega}f(\Theta_{P}).$
\end{definition}

\begin{example}
  If $N$ is compact, $1-1$ immersions and submersions $f:N\to P$ are
  infinitesimally stable.
\end{example}

Infinitesimal stability has a local counterpart that we define now.
 Recall that two maps $f,g:N^{n}\to P^{p}$ define the same \emph{germ} at $x=a$ if they agree in some
neighborhood of $a.$ The point $x=a$ is the \emph{source} of the germ
and $b=f(a)$ is its \emph{target}. The analogues of the above notations
for a germ $f:(N,a) \to (P,b)$ can be obtained replacing  $N$ by
$(N,a)$ and $P$ by $(P,b)$ in the previous notation.  However to
simplify notation, we take local coordinates such that $a=0\in
 \mathbb R^{n}$ and $f(a)=0\in \mathbb R^{p},$ denoting the
 germ  $f:(\mathbb R^{n},0) \to (\mathbb R^{p},0).$ In this case, we
 use the usual notation:

$\mathcal{E}_n=\{\lambda:(\mathbb R^{n},0)\to \mathbb R\}$ is the
local ring of $C^{\infty}$ function germs at the origin. Its unique
maximal ideal is $\mathcal{M}_{n}=\{\lambda \in
\mathcal{E}_{n}|\;\,\lambda(0)=0\}.$ 

$\mathcal{E}_{n}^{p}=\{f:(\mathbb R^{n},0) \to \mathbb R^{p}\}$ is a
free $\mathcal{E}_{n}$-module of rank $p,$ also denoted by $\mathcal{E}_{n,p}.$

The local version of  the previous diagram is

$$
  \xymatrix{{(T\mathbb R^{n},0)} \ar[d]_{\pi_{1}} \ar[r]^{df}
    &{(T\mathbb R^{p},0)}  \ar[d]^{\pi_{2}} \\
{(\mathbb R^{n},0)} \ar[r]_f\ar[ru]^{\sigma}   &      {(\mathbb R^{p},0)} }.
$$

The set

$$
\Theta_{f}=\{\sigma:(\mathbb R^n,0)\to (T\mathbb R^{p},0)|\;\, \pi_{2}\circ \sigma=f\}
$$
is the  $\mathcal{E}_{n}$-module of rank $p$ consisting of germs of vector
fields along $f.$ When $f$ is the identity in $\mathbb R^{n},$
respectively in $\mathbb R^{p},$ we obtain

$$
\Theta_{n}=\{\xi:(\mathbb R^n,0)\to (T\mathbb R^n,0)|\;\, \pi_{1}\circ \xi=\textrm{id}_{\mathbb R^n}\}
$$
and
$$
\Theta_{p}=\{\eta:(\mathbb R^p,0)\to (T\mathbb R^p,0)|\;\, \pi_{2}\circ \eta=\mathrm{id}_{\mathbb R^p}\}
$$

We now define the groups acting on $\mathcal{E}_{n}^P.$

\begin{definition} \index{Mather's groups}
  Let
  \begin{align*}
    &\mathcal{R}=\{h:(\mathbb R^n,0)\to (\mathbb
  R^n,0),\,\mathrm{ germs\;  of\; } C^\infty-\mathrm{diffeomorphisms\
    in}\, (\mathbb R^n,0)\},\\
   &\mathcal{L}=\{k:(\mathbb R^p,0)\to (\mathbb
  R^p,0),\,\mathrm{ germs\;  of\; } C^\infty-\mathrm{diffeomorphisms\
     in}\, (\mathbb R^p,0)\},
  \end{align*}
  and $\mathcal{A}=\mathcal{R}\times \mathcal{L}.$
\end{definition}

The actions of the groups $\mathcal{R}, \mathcal{L}$ and $\mathcal{A}$
are as follows

\begin{align*}
  \mathcal{R}\times\mathcal{E}_{n}^{p}&\to \mathcal{E}_{n}^{p}
  &\mathcal{L}\times\mathcal{E}_{n}^{p}&\to \mathcal{E}_{n}^{p}   &\mathcal{A}\times\mathcal{E}_{n}^{p}&\to \mathcal{E}_{n}^{p}\\
  (h,f) &\mapsto f\circ h^{-1},  &(k,f) &\mapsto k\circ f,  &((h,k),f) &\mapsto k\circ f\circ h^{-1}.
\end{align*}

These notions extend to multigerms. Let $ S=\{x_{1},x_{2},\dots
x_{s}\}$ be a finite subset of $\mathbb R^{n}.$

\begin{definition} \index{Multigerm}
  A \emph{multigerm} at $\mathcal{S}=\{x_{1}, \dots,x_{s}\}$ is the germ of a smooth map
  $$f=\{f_{1},f_{2},\dots f_{s}\}:(\mathbb R^{n},S)\to(\mathbb
  R^{p},y),\; f_{i}(x_{i})=y,\ i=1,\dots, s.$$
By a local change of coordinates at each $x_{i}\in \mathcal{S},$ we
can take $f_{i}:(\mathbb R^{n},0)\to(\mathbb R^{p},0)$ and we let
 $\mathcal{M}_{S}\mathcal{E}^{p}_{n,S}$ be the vector space of these map-germs, and call $f_{i}, i=1,\dots, s$ a \emph{branch} of $f.$
\end{definition}

The previous notations for monogerms extend naturally to multigerms. As before
$\Theta_{f}$ and $\Theta_{n,S}$ are $\mathcal{E}_{n,S}$-modules. The
map $ tf:\Theta_{n,S}\to \Theta_{f}$ is an $\mathcal{E}_{n,S}$-module
homomorphism defined by $tf(\xi)(x)=df_{x}(\xi(x)).$

The map-germ $f:(\mathbb R^{n},S)\to (\mathbb
R^{p},0)$ induces the ring homomorphism

\begin{align*}
  f^{*}:\mathcal{E}_{p}&\to \mathcal{E}_{n,S} \\
  \gamma &\mapsto f^{*}(\gamma)= \gamma\circ f,
\end{align*}           
and  we say that the map
\begin{align*}
  {\omega}f:\Theta_{p} &\to \Theta_{f}.\\
  \eta &\mapsto {\omega}f(\eta)= \eta \circ f
\end{align*}         
is a \emph{homomorphism over $f^{*}(\mathcal{E}_{p})$} (or
alternatively, an $\mathcal{E}_{p}$-module homomorphism via $f$). 

\begin{definition}\index{Germ!$\mathcal{A}$-equivalence}
  Two germs $f,g:(\mathbb{R}^{n},S)\to (\mathbb{R}^{p},0)$ are
  $\mathcal{A}$-equivalent ($f\widesim{\mathcal{A}}g$) if there exist
  $h:(\mathbb{R}^{n},S)\to (\mathbb{R}^{n},S)$ and $k:(\mathbb{R}^{p},0)\to (\mathbb{R}^{p},0)$
  such that $g=k \circ f\circ h^{-1}.$
\end{definition}

\begin{definition} \index{Germ!infinitesimally stable}
The germ   $f:(\mathbb{R}^{n},S)\to (\mathbb{R}^{p},0)$ is \emph{infinitesimally stable} if
$$
  tf(\Theta_{n,S})+{\omega}f(\Theta_{p})=\Theta_{f}
$$
\end{definition}

\begin{remark}
  When we refer to an infinitesimally stable multigerm $f:(N,S)
  \to (P,y),$ we use the notation
$$
  tf(\Theta_{(N,S)})+{\omega}f(\Theta_{(P,y)})=\Theta_{f}.
$$
\end{remark}

\begin{definition} \index{Tangent space!$\mathcal A$-equivalence}\index{Extended tangent space!$\mathcal A$-equivalence}
  For the groups $\mathcal{G}=\mathcal{R},\mathcal{L},\mathcal{A},$
  and any multigerm $f:(\mathbb{R}^{n},S)\to (\mathbb{R}^{p},0),$ we define the
  \emph{tangent space} $T\mathcal{G}_{f}$ and the \emph{extended
    tangent space} $T\mathcal G_{e}f$ as  follows:
\begin{align*}
  T\mathcal{R}f&= tf(\mathcal{M}_{n}\Theta_{n,\mathcal{S}}) 
  &T\mathcal{R}_{e}f&= tf(\Theta_{n,\mathcal{S}})\\
  T\mathcal{L}f&= {\omega}f(\mathcal{M}_{p}\Theta_{p}) &T\mathcal{L}_{e}f&=
  {\omega}f(\Theta_{p}) \\
  T\mathcal{A}f&= tf(\mathcal{M}_{S}\Theta_{n,\mathcal{S}})+ {\omega}f(\mathcal{M}_{p}\Theta_{p}) &T\mathcal{A}_{e}f&= tf(\Theta_{n,\mathcal{S}})+{\omega}f(\Theta_{p})
\end{align*}
\end{definition}

{One can give a heuristic justification for the definition of the
tangent space for the groups $\mathcal G$ in the above
definition. They  can be seen as the set of
``tangent vectors'' at the origin,
to ``paths'' $f_{t}$, such  that $f_{0}=f,$ and
$f_{t}$ is contained in the $\mathcal{G}-$orbit of $f.$ A careful
calculation in the case $\mathcal{G}=\mathcal{A},$ beginning with
$f_{t}=\psi_{t}\circ f_{t}\circ \phi_{t}$ and differentiating with
respect to $t,$ is done on pages
60-61 of the book of Mond and Nu\~no-Ballesteros \cite{MonNun}.}

For any group $\mathcal{G}$ acting on $\mathcal{E}_{n, S}$ the
$\mathcal{G}$-codimension and  the $\mathcal{G}_{e}$-codimension  to
the $\mathcal{G}$-orbit of $f,$ are given by
$$
\mathcal{G}\textrm{-}{\cod f}=\dim_{\mathbb R}\frac{\mathcal{M}_{S}\Theta_{f}}{T\mathcal{G}f}\
\textrm{\ and\ }\ 
 \mathcal{G}_e-{\cod f}=\dim_{\mathbb R}\frac{\Theta_{f}}{T{\cal G}_{e}f}.
$$

Note that a map-germ $f\in \mathcal{E}_{n, S}$ is infinitesimally
stable if and only if $\mathcal{A}_e$-${\cod f}=0.$ 
\begin{definition}\index{Map!locally infinitesimally stable}
A mapping $f:N\to P$ is \emph{locally infinitesimally stable at} $\mathcal{S}=\{x_{1},
\dots, x_{s}\}\subset N$ if the germ of $f$ at $\mathcal{S}$ is
infinitesimally stable.
 \end{definition}

 The next theorem shows that for proper mappings infinitesimal stability is locally a
 condition of finite order. That is, if the equations can be solved
 locally to order $p=\dim P,$ then they can be solved globally. 

 \begin{theorem}[Theorem 1.6, Chapter 5, \cite{GolGui}]\label{delta}
   Let $f:N\to P$ be a smooth and  proper $C^{\infty}$ mapping. Then
   $f$ is infinitesimally stable if and only if for every $y\in P$ and every
   finite set $\mathcal{S} \subset f^{-1}(y),$  with no
   more than $(p+1)$ points, we have 
$$
\Theta_{f}= tf(\Theta_{(N,\mathcal S)})+{\omega}f(\Theta_{(P,y)})+\mathcal 
M^{p+1}_{\mathcal S}\Theta_{f}.
$$
\end{theorem}

The proof of the necessity in theorem \ref{delta} is obvious. To prove
the sufficiency, the main tool is the generalized Malgrange
Preparation Theorem proved by Mather in \cite{Mat68-1}. See
Proposition \ref{prop:69-2} and Corollary \ref{cor:inf-sta}.
A complete proof of this theorem is given in Chapter 5, section 1 of \cite{GolGui}.

Our main goal in this section is to discuss the following theorem.

 \begin{theorem}[Mather \cite{Mat70}, Theorem 4.1]\label{th:3.10}
   The following conditions are equivalent in $C^{\infty}_{pr}(N,P)$
   for a proper mapping $f:N\to P.$
   \begin{enumerate}
   \item $f$ is stable,
   \item $f$ is infinitesimally stable.
   \item $f$ is transverse  stable.
   \end{enumerate}
 \end{theorem}

 We present the main steps of the proof of Theorem \ref{th:3.10}. Initially we 
discuss the notion of transverse stability.

\subsection{Transverse stability and the proof of $2. \Leftrightarrow 3.$}

{The idea of transverse stability consists in defining a stratification in jet space, such 
that the strata of this stratification are invariant by the action of the group $\mathcal{A}$ in jet 
space.} A map is \emph{transverse stable} if its $k$-jet is transversal to this stratification. To make 
this notion more precise, we introduce the $r$-fold $k$-jet bundle, following Mather \cite{Mat70}.

Let $N$ and $P$ be manifolds. Let $N^{(r)}=\{(x_1,x_2,\dots, x_r)\in
N^{r} |\; x_i\neq x_j\ 
\textrm{if}\  i\neq j\}$. Let $\pi_{N}:J^{k}(N,P) \to N$ denote the
projection where $J^{k}(N,P)$ is the bundle of $k$-jets. We define
$_{r}J^{k}(N,P)=(\pi^{r}_{N})^{-1}(N^{(r)})$ where
$\pi^{r}_{N}:J^{k}(N,P)^{r}\to N^{r}$ is the projection.

It follows that
$$_{r}J^{k}(N,P)=\{(z_{1}, \dots, z_r)\in
J^{k}(N,P)^{r}, \mathrm{such\  that}\ \pi_{N}(z_{i})\neq \pi_{N}(z_{j}),
\ \mathrm{if}\  i\neq j\}.$$

The set $_{r}J^{k}(N,P)$ is a fibre bundle over $N^{(r)}\times P^{r},$
and we call it the \emph{r fold k-jet bundle} of mappings of $N$ into
$P.$

If $f:N \to P$ is a $C^{\infty}$ mapping, we define

\begin{align*}
_{r}j^{k}f:N^{(r)} &\to\  _{r}J^{k}(N,P)
\intertext{by}
_{r}j^{k}f(x_{1}, \dots, x_{r}) &=(j^{k}f(x_{1}), \dots, j^{k}f(x_{r}))
\end{align*}

The action of the group $\mathcal{A}$ in $_{r}J^{k}(N,P)$ is defined
 as follows. If $(h,h') \in \mathcal{A},$
$z=(z_{1},\dots, z_{r}) \in\
_{r}J^{k}(N,P).$ $x_{i}=\pi_{N}z_{i},$ and
$j^{k}f_{i}(x_{i})=z_{i},$ then $(h,h')z=(z_{1}', \dots, z_{r}')$
where $z_{i}'=j^{k}(h'\circ f_{i}\circ h^{-1})h(x_{i}).$ We denote by
$\mathcal{A}^{k}$ the group of $k$-jets of elements in $\mathcal{A}.$

\begin{proposition}[Mather \cite{Mat70}, Proposition 1.4]
  An $\mathcal{A}^{k}$ orbit $W$ in $_{r}J^{k}(N,P)$ is a submanifold.
\end{proposition}

\begin{definition}
   $f:N\to P$ is \emph{transverse} stable if $_{r}j^{k}f:N^{(r)} \to   {_{r}J}^{k}(N,P)$ is transverse to every $\mathcal{A}^{k}$ orbit $W$
   in $_{r}J^{k}(N,P).$
 \end{definition}

An important remark is that in order to understand the local
structure of the orbits in $_{r}J^{k}(N,P)$ it is sufficient to
understand the structure of the orbits in $\pi^{r}_{P}(\Delta_{r}),$
where $\Delta_{r} \subset P^{r}$ is the diagonal (see Mather
\cite{Mat70} for details). In other words, it suffices to take jets
with sources $\mathcal S =\{x_{1}, \dots, x_{r}\}$ for which
$f(x_{1})= \dots =f(x_{r}).$

 The next proposition gives a characterization of transversality of
 $_{r}j^{k}f$ to $W;$ it is an important step in the proof of theorem \ref{th:3.10}.

 \begin{proposition}[Mather \cite{Mat70}, Proposition 2.6]\label{seila}
   $_{r}j^{k}f$ is transverse to $W$ at $x$ if and only if,
 $$
 tf(\Theta_{(N,\mathcal S)})+{\omega}f(\Theta_{(P,y)})+\mathcal M^{k+1}_{S}\Theta_{f}=\Theta_{f},
$$
where $y=f(x),$ $\mathcal{S}= f^{-1}(y)=\{x_{1}, \dots, x_{r}\}.$
 \end{proposition}

 From proposition \ref{seila} and theorem \ref{delta} we obtain the
 proof of $2.\Longleftrightarrow 3.$ in theorem \ref{th:3.10}. 

That $1.$ implies $3.$ in Theorem \ref{th:3.10} follows from a
general fact, and it is not hard to show.

In fact, let $f:N\to P$ be a stable mapping. It follows from the
transversality theorem that $f$ can be
well approximated by  a mapping $g:N\to P,$ such that $g$ is transverse
stable and $g\widesim{\mathcal A} f.$ That is, there is $(h,k)\in \mathcal{A}$
such that $g=k\circ f \circ h^{-1}.$ Now, transversality is preserved
by $\mathcal{A}$-equivalence,  hence $f$ is
transverse stable as well, as we wanted to show.

We have proved $1. \Rightarrow 2.\Leftrightarrow 3..$ 

Mather proved in \cite{Mat69-1}, Theorem 1 that if $f$ is proper and
infinitesimally stable then it is stable, that is
$2.\Rightarrow 1..$


His proof follows from the following result. 

\begin{theorem}[Mather \cite{Mat69-1}, Theorem 2]\label{star}
  If $f$ is proper and infinitesimally stable, then there exists a
  neighborhood $U$ of $f$ in $C^{\infty}(N,P)$ and continuous mappings
  $H_{1}:U\to \textrm{Diff}^{\infty}(N)$ and $H_{2}:U\to
  \textrm{Diff}^{\infty}(P)$ such that $H_{1}(f)=1_{N},$
  $H_{2}(f)=1_{P}$ and $g = H_{2}(g)\circ f \circ H_{1}(g),$ for $g\in U.$
\end{theorem}

Du Plessis and Wall \cite{PleWal} introduced the notion of 
\emph{W-strongly} stable mappings as stable mappings $F:N\to P$
admitting a neighborhood $U$ in $C^{\infty}(N,P)$ satisfying the
conditions stated in theorem \ref{star}. 

The main difficult  to prove that stable mappings are $W$-strongly stable is that in the Whitney
$C^{\infty}$ topology, the composition of mappings is not
continuous. However continuity holds when one restricts to proper
mappings. The strong stability of non proper functions was recently
discussed by Kenta Hayano in \cite{Hay}.

It follows that the result $2. \Rightarrow 1.$ is an easy consequence of theorem \ref{star}.

The hypothesis that $f$ is proper cannot be omitted, as we see in the
following example.

\begin{example}[\emph{\cite{Mat69-1}, pp.\, 267}]
  Let $N=(-1,1) \cup (1,2),$ $P=(-1,1),$ and
  \begin{align*}
    f|_{(-1,1)}:(-1,1) &\to (-1,1) &f|_{(1,2)}:(1,2) &\to (-1,1) \\ 
            x &\mapsto x^{2}  &x\mapsto 2-x
  \end{align*}

{ We can verify that $f$ is infinitesimally stable, as
 the restrictions to $(-1,1)$ and $(1,2)$ are.}

 However, $f$ is not stable since it has the following non-stable
 property: for any $a \in P,$ $f^{-1}(a)$ contains either $0,$ $1$ or
 $3$ points.
\end{example}

The reader can find in \cite{Mat70} the discussion of which
implications in theorem \ref{th:3.10} depend on the hypothesis that
$f$ is proper.

In the next example we illustrate the role of the Whitney
$C^{\infty}$-topology in the characterization of stable mappings.

\begin{example}
  The cusp map
  \begin{align*}
    F:\mathbb R^{2} &\to \mathbb R^{2}\\
    (x,y) &\mapsto F(x,y)=(x, y^{3}+xy)
  \end{align*}
is a stable mapping when the topology  in $C^{\infty}_{pr}(\mathbb
R^{2},\mathbb R^{2})$ is the Whitney topology.
This follows from Whitney's theorem as we discussed in section
\ref{sec:from-1944-1958:the}. 
We can also apply Mather's result: the map $F$ is proper and
infinitesimally stable , hence it is stable

Let $F_{n}(x,y)=(x, y^{3}+xy+\frac{x^{2}}{n}y)$. The singular set of
$F_{n}$ is the set $\Sigma_{n}$ defined by $3y^{2}+x+\frac{x^{2}}{n}=0.$ For each $n$,
$F_{n}$ has two cusp points: $(0,0)$ and $(-n,0).$

We can easily see that $F_{n}\to F$ in $C^{\infty}_{pr}(\mathbb
R^{2},\mathbb R^{2})$ with the topology of uniform convergence on
compact sets. Hence $F$ is not stable when one considers this topology in
$C^{\infty}_{pr}(\mathbb R^{2},\mathbb R^{2}).$ 
\end{example}

\subsection{Notes}
\label{sec:notes}

The definitions and properties of infinitesimally stable mappings also
hold for real and complex analytic germs. However, care is necessary to
characterize stable maps $f:N \to P,$ when $f$ is a holomorphic map
between complex manifolds $N$ and $P.$ In fact,Thom's transversality
theorem does not hold in general in this case. See discussion by
F. Forstneri\v{c}, \cite{Fos} and examples given by S. Kaliman and
M. Za\u{\i}denberg in \cite{KalZai}. In a recent paper, S. Trivedi \cite{Tri13}
proves that the set of maps between Stein manifolds and Oka manifolds,
transverse to a countable collection of submanifolds in the target is
dense in the space of holomorphic maps with the weak topology. The results
hold, in particular, for holomorphic maps $f:\mathbb C^{n} \to
\mathbb C^{p},$ as the complex spaces satisfy the hypothesis of the
theorem.

A related problem is the characterization of topologically stable
polynomial mappings $f:\mathbb C^{n} \to \mathbb C^{p}.$ M. Farnick,
Z. Jeloneck and M.A. S. Ruas \cite{FarJelRua}, characterize topologically stable
polynomial mappings $F:\mathbb C^{2}\to\mathbb C^{2}$ in the space
$\Omega_{\mathbb C^{2}}(d_{1}, d_{2})$ of polynomial mappings of degree
bounded by $(d_{1},d_{2}).$
Locally stable singularities are folds and cusps, but the behavior of
generic polynomial mappings at infinity imposes new restrictions. The
number of cusps of a topological stable $F\in \Omega_{\mathbb
  C^{2}}(d_{1}, d_{2})$ is given by $c(F)= d_{1}^{2}+d_{2}^{2}+3
d_{1}d_{2}-6d_{1}-6d_{2}+7.$ In particular, when $d_{1}=1$ and
$d_{2}=3,$  $c(F)=2.$
\section{Finite determinacy of Mather's groups}
\label{sec:finite-determ-math}

Mather's groups are the groups $\mathcal{G}=\mathcal{R},\mathcal{L},\mathcal{A},
\mathcal{K}$ and $\mathcal{C}.$

The contact group $\mathcal{K},$ defined by Mather in
\cite{Mat68-2} plays a fundamental role in the classification of
stable singularities. In subsections \ref{subsec:contact-group} and \ref{sec:class-stable-sing}
we define the group $\mathcal{K},$ discuss properties of $\mathcal{K}$-equivalence
and their role in the study of stable mappings.

The problem of classification of stable singularities motivated the
introduction of the notion of $\mathcal{G}$-finitely determined
germs \cite{Mat68-2}.
For the groups $\mathcal{G}=\mathcal{R}$ or $\mathcal{K},$ finite
determinacy was studied by J. Tougeron in \cite{Tou1} and chapter II
of \cite{Tou2}. When $\mathcal{G}=\mathcal{A}$ or $\mathcal{L},$ the
first results are due to Mather's in \cite{Mat68-2}.
Infinitesimal criteria of finite determinacy for
$\mathcal{G}=\mathcal{A}$ and $\mathcal{L}$ depend on the
Preparation Theorem. We discuss the infinitesimal criterion for
Mather's group in section \ref{subsec:fi-nite-determined}. In section
\ref{sec:maps-finite-sing} we introduce the basic properties of maps
of finite singularity type.

\subsection{The contact group}
\label{subsec:contact-group}

\begin{definition}\label{def:contact} \index{Contact!group}
  The \emph{contact group} $\mathcal{K}$ is the set of pairs of germs of
  diffeomorphisms $(h,H),$ where $h:(\mathbb R^{n},0)\to (\mathbb
  R^{n},0),$ $H:(\mathbb R^{n}\times \mathbb R^{p},0)\to (\mathbb
  R^{n}\times\mathbb R^{p},0)$ such that $\pi_{1}\circ H=h,$
  $(\pi_{2}\circ H)(x,0)=0$ where $\pi_{1}$ and $\pi_{2}$ are the
  projections into $\mathbb R^{n}$ and $\mathbb R^{p},$ respectively. 
\end{definition}

Notice that
$H(x,y)=(h(x),H_{2}(x,y)),\, H_{2}(x,0)=0.$

The set of pairs $(h,H)\in \mathcal{K},$ such that $h$ is the
identity $I_{\mathbb{R}^{n}}$ form a subgroup of $\mathcal{K},$ usually
denoted by $\mathcal{C}.$

\begin{definition}\index{Contact!equivalence}
  Let $f,g\in \mathcal{E}_{n}^{p}.$ We say
  that $f$ and $g$ are contact equivalent, $f\widesim{\mathcal{K}}g, $ if there is a pair $(h,H) \in
  \mathcal{K}$ such that $H(x,f(x))=(h(x),g(h(x)).$
\end{definition}

\begin{remark}

  Notice that if $f \widesim{\mathcal{K}} g, $ then the diffeomorphism
$H:(\mathbb R^{n}\times \mathbb R^{p},0) \to (\mathbb
R^{n}\times\mathbb R^{p},0)$ sends $\mathrm{graph}(f)$ into
$\mathrm{graph}(g),$ leaving $\mathbb R^{n}\times\{0\}$ invariant (see
Figure \ref{fig:contact}). This geometric viewpoint of contact equivalence was
extended by Montaldi \cite{Mon} as follows: two pairs of germs of submanifolds
of $\mathbb R^{m}$ have the same contact type if there is a
germ of diffeomorphism of $\mathbb R^{m}$ taking one  pair to the
other. Moreover,  he proved in \cite{Mon}, that the
contact type of a pair of germs of manifolds is completely characterized by
the $\mathcal{K}$-equivalence class of a convenient map. This result
is one the fundamental pieces of the applications of singularity theory
to differential geometry (see Bruce and Giblin \cite{BruGib} and Izumiya,
Romero-Fuster, Ruas and Tari,  \cite{IzRoRuTa}).

\end{remark}\index{$\mathcal{K}$-equivalence}
\index{Tangent  space!$\mathcal{K}$-equivalence}\index{Extended tangent space!$\mathcal{K}$-equivalence}
The \emph{tangent space} and the \emph{extended tangent space} of
$\mathcal{K}$-equivalence are, respectively

\begin{align*}
  T\mathcal{K}f &=tf(\mathcal{M}_{n}\Theta_{n})
+f^{*}(\mathcal{M}_{p})\Theta_{f}\\
T\mathcal{K}_{e}f &=tf(\Theta_{n})
+f^{*}(\mathcal{M}_{p})\Theta_{f}
\end{align*}

We also define $\mathcal{K}$-${\cod f}=\dim_{\mathbb
  R}\frac{\mathcal{M}_{n}\Theta_{f}}{T\mathcal{K}f}$ and $\mathcal{K}_e$-${\cod
  f}=\dim_{\mathbb R}\frac{\Theta_{f}}{T\mathcal{K}_{e}f}.$

\begin{figure}[h]
  \centering
 \scalebox{0.60}{\includegraphics{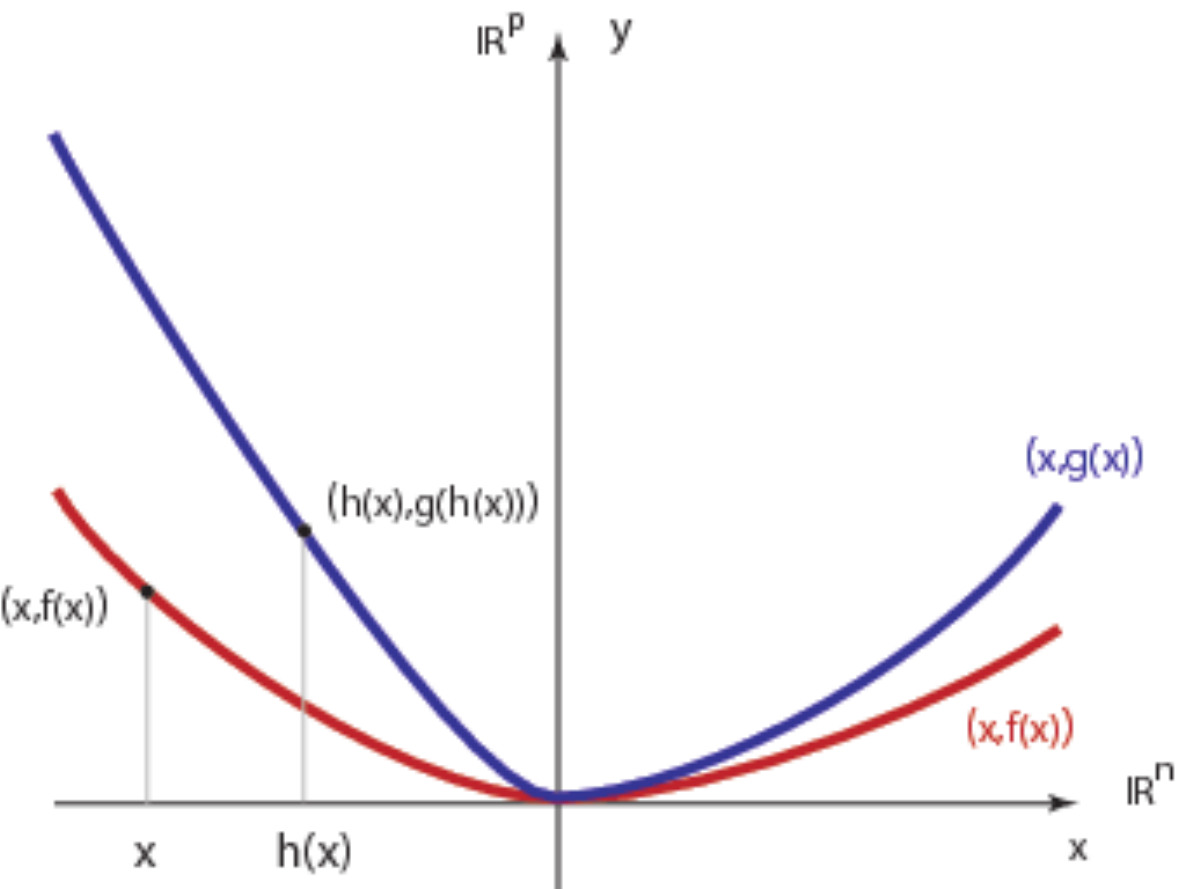}}
  \caption{Contact equivalence}
  \label{fig:contact}
\end{figure}

The following result was first proved by Mather in \cite{Mat69-2}.

\begin{proposition}[Gibson \cite{Gib}, Proposition 2.2, Mond and
  Nu\~no-Ballesteros \cite{MonNun}, Section 4.4]

The following statements are equivalent.
\begin{itemize}
\item [(1)]\  Two map-germs $f,g \in \mathcal{E}_{n}^{p}$ are $\mathcal{K}$-equivalent.
\item [(2)]\  There exists a germ of diffeomorphism $h:(\mathbb R^{n},0) \to (\mathbb R^{n},0)$ such that
$$h^{*}f(\mathcal{M}_{p})\mathcal{E}_{n}=g^{*}(\mathcal{M}_{p})\mathcal{E}_{n}.$$
\end{itemize}
\end{proposition}

The local algebra we introduce now is an useful invariant of $\mathcal{K}$-equivalence.
For a given map-germ $f:(\mathbb R^{n},0) \to (\mathbb R^{p},0)$ we
define the \emph{local algebra} of $f$ as \index{Local algebra}
$$
Q(f)=
\frac{\mathcal{E}_{n}}{f^{*}(\mathcal{M}_{p})\mathcal{E}_{n}}.
$$

It follows from the previous proposition that the isomorphism class of
$Q(f)$ is a $\mathcal{K}$-
invariant.
Furthermore, it is a complete invariant of
$\mathcal{K}$-equivalence for germs $f$ with finite
$\mathcal{K}$-codimension. More precisely, we have

\begin{theorem}\label{pr:9}
If $f$ and $g$ are map-germs with finite $\mathcal{K}$-codimension it
follows that
$$  f\widesim{\mathcal{K}}g \mathrm{\ if\ and\ only\ if\ the\ local\
  algebras\ } Q(f)\ \mathrm{\ and\ } Q(g) \mathrm{\ are\ isomorphic.\ } $$
\end{theorem}

\begin{remark}
  For complex analytic germs the hypothesis of $\mathcal{K}-$determinacy
in Theorem \ref{pr:9} is not needed.
\end{remark}

\begin{example}\label{ex:6.5}
  Let $F:(\mathbb R^{n},0) \to (\mathbb R^{p},0)$ be a germ of rank
  $r.$ Then, up to $\mathcal{A}$-equivalence, we can take $F$ in the
  normal form $F(x,y)=(x,f(x,y)), x \in \mathbb R^{r}, y \in \mathbb
  R^{n-r},$ with $f:(\mathbb R^{n},0) \to (\mathbb R^{p-r},0)$ and
  $j^{1}f(0,0)\equiv 0.$ Let $f_{0}:(\mathbb R^{n-r},0) \to (\mathbb
  R^{p-r},0)$ be the rank zero germ $f_{0}(y)=f(0,y).$ Then
  $Q(F)=Q(f_{0}).$

  If $\mathcal{K}$-$\cod f_{0} <\infty$ and $Q(F) \cong Q(f_{0})$ it follows
  that $F$ is $\mathcal{K}$-equivalent to the suspension $F_{0}(x,y)=
  (x,f_{0}(y))$ of $f_{0}.$
\end{example}

As we shall see in the next section, germs $f\in \mathcal{E}_{n}^{p}$ of
finite $\mathcal{K}$-codimension are finitely
$\mathcal{K}$-determined, and in this case
$\mathcal{K}(f)=\mathcal{K}(z),$ where $z=j^{k}f(0)$ for some $k.$

Now, for each positive integer $k,$ we set
$$Q_{k}(f)=\frac{\mathcal{E}_{n}}{f^{*}(\mathcal{M}_{p})\mathcal{E}_{n}+\mathcal{M}_{n}^{k+1}}.$$

$Q_{k}(f)$ is the local algebra of $z=j^{k}f(0).$ We can also write $Q_{k}(f)=Q(z).$ 

It is not hard to show that $z\widesim{\mathcal {K}^{k}}z'$  if and
only if $Q_{k}(z)$ and $Q_{k}(z')$ are isomorphic.

This definition can be extended to  $k$-jets of a multigerm $f:(\mathbb
R^{n},S)\to (\mathbb R^{p},0)$ $S=\{x_{1}, x_{2}, \dots, x_{s}\}.$
\index{Contact!class}
By a \emph{contact class} in $J^{k}(N,P)$ we mean an equivalence class of
$ _{s}J^{k}(N,P)$ under the relation of $\mathcal{K}^{k}$-equivalence.

\subsection{Finitely determined germs}
\label{subsec:fi-nite-determined}

Let $\mathcal{G}$ be a group acting in the space of germs $f:(\mathbb
R^{n},0) \to (\mathbb R^{p},0).$  We say that $f$ is
finitely $\mathcal{G}$-\emph{determined}\index{$\mathcal{G}$-determined} if there exists a positive integer
$k$ such that for all $g:(\mathbb R^{n},0) \to (\mathbb R^{p},0)$ with
$j^{k}g(0)=j^{k}f(0),$  it follows that $f\widesim{\mathcal{G}} g.$ We
  say that $f$ is $\mathcal{G}$-\emph{finitely determined} if $f$ is
  $k$-determined for some $k.$ The denomination ${\cal
    G}$-\emph{finite} germs is also widely used.

  Finite determinacy has been an important subject in singularity theory
  for many decades and the bibliography in this topic is extensive.

  With regard to results on necessary and sufficient conditions of
  finite determinacy and estimates of the order of determinacy we
  refer  to Mather \cite{Mat68-2}, Gaffney \cite{Gaf76, Gaf79}, du
  Plessis \cite{Ple80}, Damon \cite{Dam84} and Du Plessis, Bruce and
  Wall \cite{BruPleWal}. The survey article by Terry Wall
  \cite{Wal81} is a complete account of the theory of finite
  determinacy for Mather's groups $\mathcal{G}=\mathcal{A},
  \mathcal{R}, \mathcal{L}, \mathcal{K}$ and $\mathcal{C}$ until
  1981. See also the clear presentation (with examples) in Chapter 6
  of the book of Mond and Nu\~no-Ballesteros \cite{MonNun}.

An important advance appeared in \cite{Dam84} in which J. Damon  defined
the \emph{geometric subgroups of }$\mathcal{K},$  a large class of
  subgroups for which the theory of finite determinacy can be
  formulated as for Mather's group.

The following theorem, known as \emph{infinitesimal
  criterion}\index{Finite determinacy!infinitesimal criterion of} gives
necessary and sufficient conditions for finite determinacy.
The original result is due to Mather \cite{Mat68-2}. We give here an
improved version due to Gaffney \cite{Gaf79} and du Plessis
\cite{Ple80}. The statement and proof of Theorem \ref{gaf-wal} are 
slight modifications  of T. Wall \cite[Theorem
1.2]{Wal81}. The reader can also compare the statement for the group
$\mathcal{A}$ in section 1.2.3 (Theorem 1.2.12) of the article of Mond
and Nu\~no-Ballesteros in this Handbook \cite{Han}.


{\begin{theorem}
  \label{gaf-wal}
  For each $f\in \mathcal{E}^{p}_{n}$, $\mathcal{G}=\mathcal{R},
  \mathcal{L}, \mathcal{A}, \mathcal{C}, \mathcal{K}$ the following conditions are
  equivalent
  \begin{enumerate}
  \item [(1)] $f$ is finitely $\mathcal{G}$-determined, 
  \item [(2)] for some $r,$ $T\mathcal{G}f \supset \mathcal{M}^r_n
    \Theta_{f},$
  \item [(3)] $\mathcal{G}$-$\cod f <\infty,$
  \item [(4)] $\mathcal{G}_{e}$-$\cod f <\infty.$
  \end{enumerate}
   More precisely, if we set $\epsilon =1$ for $\mathcal{G}=\mathcal{R},
  \mathcal{C}$ or $ \mathcal{K}$ and $\epsilon=2$ for $\mathcal{G}= \mathcal{L}, \mathcal{A},$
  \begin{enumerate}
  \item [(i)] If $f$ is $k$-$\mathcal{G}$-determined then $T\mathcal{G}f \supset \mathcal{M}^{k+1}_n
    \Theta_{f},$
  \item [(ii)] If $T\mathcal{G}f \supset \mathcal{M}^{k+1}_n\Theta_{f},$ then
    $f$ is  $(\epsilon k+1)$-$\mathcal{G}$-determined.
	 \item [(iii)]  If $T\mathcal{G}f + \mathcal{M}^{\epsilon
             k+2}_n\Theta_{f} \supset \mathcal{M}^{k+1}_n\Theta_{f},$
           then $T\mathcal{G}f \supset \mathcal{M}_n^{k+1}\Theta_{f}.$ 
  \end{enumerate}
\end{theorem} }

  This section is mainly devoted to describe this result. Although the
  theory applies to multigerms,  for   simplicity we restrict our
  discussion to monogerms $f:(\mathbb R^{n},0) \to (\mathbb R^{p},0).$

The successful approach to finite determinacy was inspired by the
action of a Lie group on finite dimensional manifolds.
The following lemma is due to Mather.

\begin{lemma}[Mather \cite{Mat69-2}, Lemma 3.1]
  Let $G$ be a Lie group, $M$ a $C^{\infty}$ manifold and
  $\alpha:G\times M \to M$ a $C^{\infty}$ action. Let $V$ be a
  connected $C^{\infty}$-submanifold of $M.$ Then $V$ is contained in
  an orbit of $\alpha$ if and only if
  \begin{enumerate}
  \item [(a)] For all $v\in V,$ $T_{v} G\cdot v \supseteq T_{v}V,$ and
  \item [(b)] $\dim T_{v}(G\cdot v)$ is the same for all $v\in V.$
  \end{enumerate}
\end{lemma}

Our groups are not Lie groups, and our function spaces are not Banach
manifolds. But, the solution to the problem of finding sufficient
conditions for a germ $f\in \mathcal{E}^{p}_{n}$ to be finitely
determined, consists in reducing our infinitesimal approach to jet
spaces.

Suppose $f$ is $k$-$G$-determined. Then, given $g\in
\mathcal{E}^{p}_{n},$ $j^{k}g(0)=j^{k}f(0),$ the one-parameter family
\begin{align*}
  \bar f:(\mathbb R^{n}\times \mathbb R, 0 \times \mathbb R) &\to (\mathbb R^{p}\times \mathbb
                                       R,0)\\
  (x,t) &\mapsto \bar f(x,t)=(1-t)f(x) + t g(x)
\end{align*}
has a constant $k$-jet $j^{k}\bar f_{t}(0)=j^{k}f(0) + t j^{k}(g-f)(0)=
j^{k}f(0).$

We will identify $\bar f$ with a ``line'' $L_{t}$ in
$\mathcal{E}^{p}_{n}.$ Our problem is to show that $L_{t}$ is
contained in a unique orbit.

A sufficient condition is to find a $1$-parameter family $h_t$ of
elements in $\mathcal{G}$ such that $h_{0}=1 \in \mathcal{G},$
$h_{t}(0)=0,$ $h_{t}\cdot f_{t}=f,$ for any $t\in \mathbb R.$
\index{Trivial family!$\mathcal{G}-$}
These conditions say that the family $\bar f$ is
$\mathcal{G}$-\emph{trivial}. As in the case of stable singularities,
the next step is to search for an infinitesimal condition, giving an
equivalent characterization of triviality in terms of vector fields.

This step, in principle, is not hard: the equation $h_{t}\cdot
f_{t}= f$ implies that $\frac{\partial}{\partial t}(h_{t}\cdot
f_{t})=0$ leading to the desired infinitesimal condition. The converse
follows from integration of vector fields.

For any group $\mathcal{G}$ acting on $\mathcal{E}_{n}^{p},$ we call
this result  ``the Thom-Levine lemma.''  We now specialize to $\mathcal{G}=\mathcal{A},$ as this case
includes all difficulties  of the proof of the infinitesimal criterion.

\begin{definition}\index{Trivial family!$\mathcal{A}-$}
  A $1$-parameter family $\bar f:(\mathbb R^{n}\times \mathbb R, 0) \to
  (\mathbb R^{p},0)$, $\bar f(x,0)=f(x)$ is $\mathcal{A}$-trivial if there is a pair $(h,k)$ of
  $1$-parameter families of germs of diffeomorphisms

  \begin{align*}
    h:(\mathbb R^{n}\times \mathbb R, 0) &\to (\mathbb R^{n},0) 
                                &k:(\mathbb R^{p}\times \mathbb R, 0) &\to (\mathbb R^{p},0)\\  
    (x,t) &\mapsto h(x,t)    &(y,t) &\mapsto k(y,t)
  \end{align*}
  such that $h(x,0)=x$, $k(y,0)=y,$ $ h_{t}(0)=0$, $k_{t}(0)=0$ and
  $$
  k_{t}\circ f_{t}\circ h_{t} =f.
  $$
\end{definition}

\begin{remark}
  We also use the notation $F(x,t)=(\bar f(x,t),t),$
  $H(x,t)=(h(x,t),t)$ and   $K(y,t)=(k(y,t),t)$  for the corresponding
  $1$-parameter unfoldings. In this  notation $F$ is
  $\mathcal{A}$-trivial if $K\circ F \circ H = f\times
  \mathrm{Id}_{\mathbb R}.$
{We denote by $\partial\cdot F$ the vector field in $(\mathbb
R^{n}\times\mathbb R, 0)$ with zero component in the
$\frac{\partial}{\partial t}$ direction, that is $dF(\frac{\partial}{\partial
  t})=(\partial\cdot F, 1).$}
\end{remark}

The next result is known as the Thom-Levine lemma (see \cite{Mat68-2,
  Ple80, MonNun}).
\index{Lemma!Thom-Levine}\index{Thom-Levine|see{ Lemma}}
\begin{proposition}
  \label{thom-levine}
  Let $f\in \mathcal{E}_{n}^{p}$ and $F:(\mathbb R^{n}\times \mathbb R, 0) \to
  (\mathbb R^{p}\times \mathbb R,0),$ $F(x,t)=(\bar f(x,t),t),$ $\bar
  f(0,t)=0, \,\bar f(x,0)=f(x),$ the germ at $0$ of a $1$-parameter unfolding of $F.$ Then $F$ is
  $\mathcal{A}$-trivial if and only there exist vector fields
  $V:(\mathbb R^{n}\times \mathbb R, 0) \to (\mathbb
  R^{n}\times \mathbb R,0)$ with $V(x,t)=v(x,t)
  +\frac{\partial}{\partial t},$ $v(x,t)=
  \sum_{i=1}^{n}v_{i}(x,t)\frac{\partial}{\partial x_{i}},$
  $v_{i}(0,t)=0$ for  $i=1,\dots,n$ 
and   $ W:(\mathbb R^{p}\times \mathbb R, 0) \to (\mathbb
  R^{p}\times \mathbb R,0)$ with $W(y,t)=w(y,t)
  +\frac{\partial}{\partial t},$ $w(y,t)=
  \sum_{j=1}^{p}w_{j}(y,t)\frac{\partial}{\partial y_{j}},$
  $w_{j}(0,t)=0$ for  $j=1,\dots,p.$ 
  such that
  \begin{equation}
    \label{eq:2}
\partial\cdot F(x,t)=\sum_{i=1}^{n}\frac{\partial \bar f}{\partial
  x_{i}}(x,t)\cdot v_{i}(x,t) + w \circ F(x,t).    
  \end{equation}
\end{proposition}

\begin{proof}
  We give here an idea of the proof.  The reader may consult, for
instance, Mather \cite[p.~144]{Mat68-2}, du Plessis \cite[p.~174]{Ple80}, or
Mond and Nu\~no-Ballesteros \cite[p.~37]{MonNun} for a 
complete proof.

If $F$ is a trivial unfolding of $f,$ $K\circ F \circ H = f\times 1_{\mathbb R}$
and then $\partial \cdot(K\circ F \circ H)=0$ and we apply the chain
rule to get \eqref{eq:2}.

Conversely, if condition \eqref{eq:2} holds, we consider the systems of
differential equations in $(\mathbb R^{n}\times \mathbb R,0)$ and  $(\mathbb
R^{p}\times \mathbb R,0),$ respectively:

\begin{align}\label{eq:4}
 &\begin{cases}\dot x = v(x,t)\\ 
   v(0,t)=0 
\end{cases} \qquad \qquad \qquad \qquad
\begin{cases}
        \dot y= w(y,t)  \\ 
    w(0,t)=0 
  \end{cases}\\  
& \qquad \nonumber
\end{align}

We can integrate these vector fields to obtain $1-$parameter families
$h_{t}$ and $k_{t}$ of diffeomorphisms of $(\mathbb R^{n}\times \mathbb R,0)$ and  $(\mathbb
R^{p}\times \mathbb R,0),$ respectively, such that $h_{0}(x)=x,$
$h_{t}(0)=0;$ $k_{0}(y)=y,$ $k_{t}(0)=0$ and  $k_{t}\circ \bar
f_{t}\circ h_{t}=f.$

\end{proof}

Condition \eqref{eq:2} in Proposition \ref{thom-levine} admits an useful
algebraic formulation. First, we introduce some notation.

Given the $1$-parameter unfolding $F:(\mathbb R^{n}\times \mathbb R,0)
\to (\mathbb R^{p}\times \mathbb R,0),$ $F(x,t)=(\bar f(x,t),t)$ with
$\bar f(x,0)=f(x),$ as before, $\Theta_{F}$ denotes the
$\mathcal{E}_{n+1}$ module of vector fields along $F.$ However, here
it will be more convenient to consider the submodule of $\Theta_{F}$
defined as:
$$
\Psi_{F}=\{\sigma \in \Theta_{F}|\, \mathrm{the}\ \mathbb R\mathrm{-component\ of}
\ \sigma\ \mathrm{is\  zero}\}.
$$

Similarly, $\Psi_{n+1}$ and $\Psi_{p+1}$ denote vector fields in
$(\mathbb R^{n}\times \mathbb R,0)$ and $(\mathbb R^{p}\times \mathbb
R,0)$ respectively, with zero $\mathbb R$-components.

The restrictions of the homomorphisms $tF$ and ${\omega}F$ give respectively
the $\mathcal{E}_{n+1}$-homomorphism $tF:\Psi_{n+1}\to \Psi_{F}$ and
the $\mathcal{E}_{p+1}$-homomorphism via $F^{*},$ ${\omega}F:\Psi_{p+1} \to \Psi_{F}.$

With this notation, we can see that \eqref{eq:2} holds if and
only if
\begin{equation}
  \label{eq:6}
  \partial \cdot F \in tF(\mathcal{M}_{n}\Psi_{n+1}) +
  {\omega}F(\mathcal{M}_{p}\Psi_{p+1})
\end{equation}
holds.

  We call $T\mathcal{A}_{un}(F)=tF(\mathcal{M}_{n}\Psi_{n+1}) +
  {\omega}F(\mathcal{M}_{p}\Psi_{p+1}),$ the $\mathcal{A}$-\emph{tangent space} of the
  unfolding $F.$ Similarly $T\mathcal{K}_{un}(F)=tF(\mathcal{M}_{n}\Psi_{n+1}) +
  F^{*}(\mathcal{M}_{p+1})\Psi_{p+1}$ is the
  $\mathcal{K}-$\emph{tangent space} of $F.$

We now turn to the algebraic tools we need in the proof of theorem
\ref{gaf-wal}.

In the cases $\mathcal{G}=\mathcal{R},\,\mathcal{C}$ or $\mathcal{K}$
the proof of the infinitesimal criterion of $\mathcal{G}$-determinacy
will follow from the following elementary result.
\index{Lemma!Nakayama}\index{Nakayama|see{ Lemma}}

\begin{lemma}[Nakayama's Lemma]\label{naka}
  Let $R$ be a commutative ring, $M$ an ideal such that for $x\in M,$
  $(1+x)$ is invertible. Let $C$ be a finitely generated $R$-module,
  $A$ a submodule, then
  \begin{itemize}
  \item [(i)] if $A+ M\cdot C=C,$ then $A=C,$
  \item [(ii)]\  if $R$ is a $k$-algebra, and
    $\dim_{k}(\frac{C}{A+M^{d+1}C})\leq d $ then $M^{d}\cdot C
    \subseteq A.$
  \end{itemize}
\end{lemma}

An equivalent formulation of condition (i) in Lemma \ref{naka} is the
following
\begin{itemize}
\item [(i')]\  If $M C = C,$ then $C=0.$
\end{itemize}

  When $\mathcal{G}=\mathcal{L}$ or $\mathcal{A},$ we need
  a fairly deep result, the generalized Malgrange preparation theorem
  (see Golubitsky and Guillemin \cite{GolGui}, Martinet \cite{Mar76, Mar77},
  Wall \cite{Wal81}).

  \index{Malgrange preparation theorem}
  \begin{theorem}[Preparation Theorem]\label{prepa}
    Let $f:(\mathbb R^{n},0) \to (\mathbb R^{p},0)$ be a $C^{\infty}$
    map-germ, $E$ a finitely generated $\mathcal{E}_{n}$-module. If
    $\dim_{\mathbb R}(\frac{E}{f^{*}(\mathcal{M}_{p})\cdot E})<\infty,$ then $E$ is
    finitely generated as $\mathcal{E}_{p}$-module (via $f$).
 \end{theorem}

 The next proposition is a consequence of the Preparation theorem. It
 is an useful tool to study $\mathcal{A}$-finite determinacy.

{ 
 \begin{proposition}[Bruce, du Plessis and Wall \cite{BruPleWal}, Lemma 2.6]\label{bpw}
  Let $C$ be a finitely generated $\mathcal{E}_{n}$-module,
  ${B} \subset {C}$ a finitely generated
  ${\mathcal E}_{n}$-submodule, ${A} \subset f^{*}(\mathcal{M}_{p}){C}$ a finitely generated
  $\mathcal{E}_{p}$-submodule   (via $f$), and ${M}$ a proper,
  finitely generated ideal in $\mathcal{E}_{n}.$ If 
$$
{M}{C} \subset {A} + {B} +
{M}(f^{*}(\mathcal{M}_{p}) + {M}){C}
$$
then ${M}{C} \subset {A} + {B}.$
\end{proposition}
}
 
We are now ready to prove Theorem \ref{gaf-wal}.

\begin{proof}[of Theorem \ref{gaf-wal}]
First we notice that (i) and (ii)  give respectively the implications
$(1)\Rightarrow (2)$ and  $(2)\Rightarrow (1).$ 
The implication  $(2)\Rightarrow (3)$ is trivial since
$\mathcal{M}^{k}_{n}\Theta_{f}$ has finite codimension.
 
It is easy to prove the equivalence between $(3)$ and $(4).$ {The
implication $(3)\Rightarrow (2)$ will follow from (iii), as we now
explain.}

For any
$\mathcal{G}=\mathcal{R},\mathcal{C},\mathcal{K},\mathcal{L},\mathcal{A}$
let
$$ c_{k}=\dim_{\mathbb K}\frac{\mathcal{M}_{n}\Theta_{f}}{T\mathcal{G}f
  +\mathcal{M}^{k}_{n}\Theta_{f}}, \ k \geq 1.
$$

Since $\mathcal{G}$-$\textrm{cod}\, f< \infty,$ the sequence
$$
0=c_{1}\leq c_{2}\leq \dots \leq \mathcal{G}\textrm{-cod}\, f$$
is finite.

Then, there exists $s$ such that $c_{k}=c_{s}$ for all $k\geq s+1.$ It
follows that
$T\mathcal{G}f+\mathcal{M}^{s}_{n}\Theta_{f} =
T\mathcal{G}f+\mathcal{M}^{k}_{n}\Theta_{f}$ for all $k\geq s+1.$ In
particular $\mathcal{M}^{s}_{n}\Theta_{f}\subseteq
T\mathcal{G}f+\mathcal{M}^{k}_{n}\Theta_{f}$ for all $k\geq s+1.$
Taking $k=s+1,$ when
$\mathcal{G}=\mathcal{R},\mathcal{C},\mathcal{K}$ and $k=2s,$ when
$\mathcal{G}=\mathcal{A},\mathcal{L},$ we obtain the statement in
(iii) from which the result follows.

{It suffices to prove (i), (ii) and (iii). For a clearer presentation,
we first prove (iii).}

If $\mathcal{G}=\mathcal{R},\mathcal{C},\mathcal{K},$ the result
follows easily by Nakayama's Lemma. If $\mathcal{G}=\mathcal{A}$ (the
argument for $\mathcal{G}=\mathcal{L}$ is similar) we apply  Proposition \ref{bpw} taking $C=\Theta_{f},$
$M=\mathcal{M}^{k+1}_{n},$  $B=tf(\mathcal{M}_{n}\Theta_{n})$ and
$A=\omega f(\mathcal{M}_{p}\Theta_{p}).$

We leave the details as an exercise to the reader.
\index{Finite determinacy!necessary condition for}
  \noindent \textbf{(i) Necessary condition for finite determinacy.}

This is not hard. A map-germ $f:(\mathbb
  R^{n},0) \to (\mathbb R^{p},0)$ is $k$-$\mathcal{G}$-determined if
  $\mathcal{G}  f$ contains all germs $g \in \mathcal{E}_{n}^{p},$ such
  that $j^{k}g(0)=j^{k}f(0).$ Let us denote this set by $\mathtt{W}.$

  Let
  \begin{align*}
    \pi^{l}:\mathcal{E}_{n}^{p} &\to J^{l}(n,p) \\
    g &\to j^{l}g(0).
  \end{align*}
As $\mathcal{G} f\supset \mathtt{W},$  then
$\pi^{l}(\mathcal{G} f)\supset \pi^{l}(\mathtt{W}).$ Thus we also get that
  \begin{equation}
    \label{eq:3}
      \textrm{the tangent space of } \pi^{l}(\mathcal{G}f)
  \supset \textrm{the tangent space of } \pi^{l}(\mathtt{W}).
\end{equation}

Notice that for all $l>k,$ the set $\pi^{l}(\mathtt{W})$ is the
  affine subspace of $J^{l}(n,p)$ consisting of all $l$-jets whose
  $k$-jet is $j^{k}f(0).$ Hence we can rewrite  \eqref{eq:3} as
$$ T\mathcal{G}f + \mathcal{M}_{n}^{l+1}\Theta_{f}\supset
\mathcal{M}_{n}^{k+1}\Theta_{f},\, l>k.$$

The result now follows from (iii) taking $l=k+1$ for
$\mathcal{G}=\mathcal{R},\,\mathcal{C}$ or $\mathcal{K}$ and  $l=2k+1$
 when $\mathcal{G}=\mathcal{A}$ or $\mathcal{L}.$ 

\index{Finite determinacy!sufficient condition for} 
\noindent \textbf{(ii) Sufficient condition for finite determinacy.}

  Let $f,g \in \mathcal{E}_{n}^{p},$ $j^{\epsilon k +1}f(0)=
  j^{\epsilon k +1}g(0)$, $\epsilon= 1$ or $2,$ $F(x,t)=(\bar
  f(x,t),t),$ where $\bar f(x,t)= (1-t)f(x) + tg (x), t\in [0,1].$

\noindent (\textbf{I}) $\mathcal{G}=\mathcal{R},\,\mathcal{C}$ or $\mathcal{K}.$

  In these cases the hypothesis

 \begin{equation}
   \label{eq:7}
   T\mathcal{G}f\supset \mathcal{M}_{n}^{k+1}\Theta_{f}
 \end{equation}
implies
\begin{equation}
  \label{eq:8}
  T\mathcal{G}_{un}(F)+\mathcal{M}^{k+2}_{n+1}\Psi_{F}\supseteq \mathcal{M}^{k+1}_{n}\Psi_{F}.
\end{equation}
 
 The proof that \eqref{eq:7} implies \eqref{eq:8} is not hard, but we omit it (the reader may consult Wall
 \cite{Wal81} or du Plessis \cite{Ple80}).

 The tangent spaces $T\mathcal{G}_{un}(F),$ $\mathcal{G}=\mathcal{R},
 \mathcal{C}$ or $\mathcal{K},$ are finitely generated
 $\mathcal{E}_{n+1}$-modules, so we can apply Nakayama's lemma to
 \eqref{eq:8} with $C= T\mathcal{G}_{un}(F)+\mathcal{M}^{k+1}_{n}\Psi_{n+1},$
 $A=T\mathcal{G}_{un}(F)$ and  $M= \mathcal{M}_{n+1}$ to get
 $T\mathcal{G}_{un}(F)\supseteq \mathcal{M}_{n}^{k+1}\Psi_{F}. $

Now, $\partial\cdot F = g-f \in \mathcal{M}_{n}^{k+2}\Psi_{F},$ and we
can apply the Thom-Levine lemma to prove that $F$ is
$\mathcal{G}$-trivial in some neighborhood of $t=0.$ For a proof of
the Thom-Levine lemma for $\mathcal{G}=\mathcal{K}$ see 
du Plessis et all \cite{GibWirPleLoo}. Notice that
$j^{k+1}\bar f_{t}(0)= j^{k+1}f(0), $ and the hypothesis (ii) holds
for $\bar f_{a},$ for any $a\in [0,1]$ , so
that the arguments of the proof also hold to prove that $F$ is
$\mathcal{G}$-trivial in a small neighborhood of $t=a$ for any $a \in
[0,1].$ Hence f is $(k+1)$-$\mathcal{G}$-determined, $\mathcal{G}=\mathcal{R},\,\mathcal{C}$ or $\mathcal{K}.$

\noindent (\textbf{II}) $\mathcal{G}=\mathcal{L}$ or $\mathcal{A}.$

In these cases, $T\mathcal{G}_{un}(F)$ is not an $\mathcal{E}_{n+1}$-module
 in general. Let $\mathcal{G}=\mathcal{A}$ ( the case
 $\mathcal{G}=\mathcal{L}$ follows as a particular case).

 $$TA_{un}(F)=tF(\mathcal{M}_{n}\Psi_{n+1})+{\omega}F(\mathcal{M}_{p}\Psi_{p+1}),$$
 $$F(x,t)=(\bar f(x,t), t),\ \bar f(x,t)= (1-t)f(x)+tg(x),$$
 and
 $j^{2k+1}f(0)=j^{2k+1}g(0)$
 
 First notice that if $F_{0}(x,t)=(f(x),t)$ is the suspension of $f,$
 the hypothesis $\mathcal{M}^{k+1}_{n}\Theta_{f}\subseteq
 tf(\mathcal{M}_{n}\Theta_{n}) +{\omega}f(\mathcal{M}_{p}\Theta_{p})$ implies
 that $$
 \mathcal{M}^{k+1}_{n}\Theta_{F_{0}}\subseteq
 tF_{0}(\mathcal{M}_{n}\Psi_{n+1})+{\omega}F_{0}(\mathcal{M}_{p}\Psi_{p+1})+(t\mathcal{M}_{n}^{k+1} +
 \mathcal{M}_{n}^{2k+2})\Psi_{F_{0}}
.$$
 Notice that $\mathcal{M}_{n}^{k+1}\Psi_{F_{0}}\subset
 \mathcal{M}_{n}^{k+1}\Theta_{f}+t\mathcal{M}_{n}^{k+1}\Psi_{F_{0}}.$

 The next step is to verify that similar  inclusion holds replacing
 $F_{0}$ by $F,$ $j^{2k+1}\bar f_{t}(0)=j^{2k+1}f(0), $ that is

 \begin{equation}
   \label{eq:9}
   \mathcal{M}_{n}^{k+1}\Psi_{F} \subset  tF(\mathcal{M}_{n}\Psi_{n+1})
   + {\omega}F(\mathcal{M}_{p}\Psi_{p+1})+(t\mathcal{M}_{n}^{k+1}+ \mathcal{M}_{n}^{2k+2})\Psi_{F}
 \end{equation}
 (see sublemma 2.2 in du Plessis \cite{Ple80}).

 If we can show that the term $(t\mathcal{M}_{n}^{k+1}+
 \mathcal{M}_{n}^{2k+2})\Psi_{F}$ can be eliminated in \eqref{eq:9}
 then the Thom-Levine lemma can be applied to prove that $F$ is
 $\mathcal{A}$-trivial.

 To achieve this goal Malgrange's  preparation theorem will be the
 fundamental tool.

 Multiplying \eqref{eq:9} by $\mathcal{M}_{n}^{k+1}$ and since $\mathcal{M}_{n}^{k+1}{\omega}F(\mathcal{M}_{p}\Psi_{p+1})\subset
 F^{*}(\mathcal{M}_{p})\mathcal{M}_{n}^{k+1} \Psi_{F},$
we get
\begin{equation}
  \label{eq:10}
  \mathcal{M}_{n}^{2k+2}\Psi_{F}\subset
  tF(\mathcal{M}_{n}^{k+2}\Psi_{n+1}) +
  F^{*}(\mathcal{M}_{p})\mathcal{M}_{n}^{k+1}\Psi_{F} +(t +\mathcal{M}_{n}^{k+1})\mathcal{M}_{n}^{2k+2}\Psi_{F}.
\end{equation}

The $\mathcal{E}_{n+1}$-module

$$
E=\frac{tF(\mathcal{M}_{n}^{k+2}\Psi_{n+1}) +
  F^{*}(\mathcal{M}_{p})\mathcal{M}_{n}^{k+1}\Psi_{F} +
  \mathcal{M}_{n}^{2k+2}\Psi_{F}.}
{tF(\mathcal{M}_{n}^{k+2}\Psi_{n+1}) +  F^{*}(\mathcal{M}_{p})\mathcal{M}_{n}^{k+1}\Psi_{F}}
$$
is finitely generated, since it is a quotient of finitely generated modules.
Moreover, from \eqref{eq:10} we get that $E=(t +\mathcal{M}_{n}^{k+1})E,$ and by
Nakayama's lemma it follows that $E=0.$ Then, we get
\begin{equation}
  \label{eq:11}
 \mathcal{M}_{n}^{2k+2}\Psi_{F}\subset  tF(\mathcal{M}_{n}^{k+2}\Psi_{n+1}) +
  F^{*}(\mathcal{M}_{p})\mathcal{M}_{n}^{k+1}\Psi_{F}.
\end{equation}

Using \eqref{eq:11} to replace part of the remainder term in
\eqref{eq:9}, we get
\begin{equation}
  \label{eq:12}
 \mathcal{M}_{n}^{k+1}\Psi_{F} \subset tF(\mathcal{M}_{n}\Psi_{n+1})
 + {\omega}F(\mathcal{M}_{p}\Psi_{p+1})+ (t+F^{*}(\mathcal{M}_{p}))\mathcal{M}_{n}^{k+1}\Psi_{F}.
\end{equation}

Let $E'$ be the $F^{*}(\mathcal{E}_{p+1})$-module
$$
E'=\frac{tF(\mathcal{M}_{n}\Psi_{n+1}) +{\omega}F(\mathcal{M}_{p}\Psi_{p+1})+
  \mathcal{M}_{n}^{k+1}\Psi_{F}}{tF(\mathcal{M}_{n}\Psi_{n+1}) +{\omega}F(\mathcal{M}_{p}\Psi_{p+1})}.
$$

Using \eqref{eq:12}, it follows that
$E'=(t+F^{*}(\mathcal{M}_{p}))E'.$ Notice that the ideal $\langle
t\rangle + F^{*}(\mathcal{M}_{p})$ is contained in
$F^{*}(\mathcal{M}_{p+1}),$ so it follows that $E'=F^{*}(\mathcal{M}_{p+1})E'.$

To apply Nakayama's lemma, one has to show that $E'$ is a
$F^{*}(\mathcal{E}_{p+1})$-module finitely generated. For this, let
the finitely generated $\mathcal{E}_{n+1}$-module 
$$
E''=\frac{tF(\mathcal{M}_{n}\Psi_{n+1}) +\mathcal{M}_{n}^{k+1}\Psi_{F}}{tF(\mathcal{M}_{n}\Psi_{n+1})}.
$$

Notice that the inclusion
$$tF(\mathcal{M}_{n}\Psi_{n+1})
+\mathcal{M}_{n}^{k+1}\Psi_{F}\subset tF(\mathcal{M}_{n}\Psi_{n+1}) +
{\omega}F(\mathcal{M}_{p}\Psi_{p+1}) +\mathcal{M}_{n}^{k+1}\Psi_{F}$$
induces an epimorphism of $F^{*}(\mathcal{E}_{n+1})$-modules $E'' \to E'$ so
that if $E''$ is a finitely generated $F^{*}
(\mathcal{E}_{p+1})$-module, then $E'$ also is.

From Malgrange preparation theorem, $E''$ is a finitely generated  
$F^{*}(\mathcal{E}_{p+1})$-module if and only if
\begin{equation}
  \label{eq:13}
\dim_{\mathbb R}\frac{E''}{F^{*}
(\mathcal{M}_{p+1})E''}< \infty.
\end{equation}


Now    
    $$
    \frac{E''}{F^{*}(\mathcal{M}_{p+1})E''} \simeq
    \frac{tF(\mathcal{M}_{n}\Psi_{n+1})+\mathcal{M}_{n}^{k+1}\Psi_{F}}
 {tF(\mathcal{M}_{n}\Psi_{n+1}) +F^{*}(\mathcal{M}_{p+1})\mathcal{M}_{n}^{k+1}\Psi_{F}}
 $$

 It follows from \eqref{eq:11} that 
$$
tF(\mathcal{M}_{n}^{k+2}\Psi_{n+1})+
F^{*}(\mathcal{M}_{p+1})\mathcal{M}_{n}^{k+1}\Psi_{F} \supset \mathcal{M}_{n}^{2k+2}\Psi_{F}.
$$

As $t\in F^{*}(\mathcal{M}_{p+1}),$ we also get that 

$$ 
tF(\mathcal{M}_{n}^{k+2}\Psi_{n+1})+
F^{*}(\mathcal{M}_{p+1})\mathcal{M}_{n}^{k+1}\Psi_{F} \supset
\mathcal{M}_{n+1}^{k+1}\mathcal{M}_{n}^{k+1}\Psi_{F}, 
$$

so that,

$$
\dim_{\mathbb R} \frac{E''}{F^{*}(\mathcal{M}_{p+1})E''} \leq
\dim_{\mathbb R} \frac{\mathcal{M}_{n}^{k+1}\Psi_{F}}{\mathcal{M}_{n+1}^{k+1}\mathcal{M}_{n}^{k+1}\Psi_{F}} <\infty 
$$
Then we can
apply Nakayama's lemma to \eqref{eq:12} to get that $E'=0,$ so
that $\mathcal{M}_{n}^{k+1}\Psi_{F} \subset
tF(\mathcal{M}_{n}\Psi_{n+1}) + {\omega}F(\mathcal{M}_{p}\Psi_{p+1}).$

To conclude we proceed as in part (I).
\end{proof}


The following result follows from Theorem \ref{gaf-wal} and Mather's
lemma.

\begin{proposition}\label{th:eps}
  Let $f\in \mathcal{E}_{n}^{p},$  $\epsilon =1$ when 
  $\mathcal{G}=\mathcal{R},\mathcal{C}$ or $\mathcal{K}$ and
  $\epsilon=2$ when   $\mathcal{G}=\mathcal{L},\mathcal{A}.$
  Then $f$ is $k$-$\mathcal{G}$-determined if and only if
  $\mathcal{M}_{n}^{k+1}\Theta_{g}\subset T\mathcal{G}g+
  \mathcal{M}_{n}^{\epsilon(k+1)}\Theta_{g}$ for all $g \in
  \mathcal{E}_{n}^{p}$ such that $j^{k}g(0)=j^{k}f(0).$
\end{proposition}

We see in the next example that the converse of condition $(i)$ in theorem
\ref{gaf-wal} does not hold, that is, the condition $T\mathcal{G}f\supseteq
\mathcal{M}_{n}^{k+1}\Theta_{f}$ does not imply in general that $f$ is $k$-$\mathcal{G}$-determined.

\begin{example}
  Let $f:(\mathbb R^{2},0) \to (\mathbb R, 0)$, $f(x,y)= x^{3}+y^{3},$ and
  $\mathcal G= \mathcal R.$ Then $$T\mathcal{R} f= \langle \frac {\partial f}{\partial x}, \frac
  {\partial f}{\partial y}\rangle \mathcal{M}_{2}= \mathcal{M}_{2}^{3}$$
  but $f$ is not $2$-$\mathcal{R}$-determined as $j^{2}f(0)\equiv 0.$
\end{example}

  A successful approach to a necessary and sufficient condition for
  finite determinacy appears in \cite{BruPleWal} where J. Bruce, A. du
  Plessis and C.T.C. Wall prove this condition for \emph{unipotent subgroups} of
  $\mathcal{G}=\mathcal{R},\,\mathcal{C},\,
  \mathcal{K},\,\mathcal{L}\,$ or $\mathcal{A}.$ 


  Let $\mathcal{G}_{s}=\{h\in \mathcal{G}|\, j^{s}h(0)=j^{s}1_{\mathcal{G}}\}$
 where $1_{\mathcal{G}}$ is the identity of $\mathcal{G},$ and
 $J^{s}\mathcal{G}$ the Lie group of $s$-jets
 of elements of $\mathcal{G}.$ The sets $\mathcal{G}_{s},$ $s\geq 1$ are
 unipotent subgroups of $\mathcal{G}.$ A special case of the main
 result in \cite{BruPleWal} is the following:

 \begin{theorem}[Bruce, du Plessis, Wall \cite{BruPleWal}]
A $C^{\infty}$ map-germ $f:(\mathbb R^{n},0) \to (\mathbb R^{p},0)$ is
$r$-$\mathcal{G}_{s}$-determined $(s\geq 1)$ if and only if
$\mathcal{M}_{n}^{r+1}\Theta_{f}\subset T\mathcal{G}_{s} (f).$ 
 \end{theorem}

\subsection{Classification of stable singularities}
\label{sec:class-stable-sing}

We consider here the problem of classification of stable germs with
respect to $\mathcal{A}$-equivalence. The main result is the following

\begin{theorem}[Mather \cite{Mat69-2}]\label{th:10}
  If $f, g$ are  stable germs then
  $f\widesim{\mathcal{A}}g$ if and only if the algebras $Q(f)$ and
  $Q(g)$ are isomorphic.
\end{theorem}

The proof of this theorem follows from the following property holding
for infinitesimally stable germs: $\mathcal{A}^{p+1}z=
\mathcal{K}^{p+1}z \cap St^{p+1},$ where $z=j^{p+1}f(0),$ and
$St^{p+1}$ is the set of all stable jets in $J^{p+1}(n,p).$ We omit
the complete proof, however the main steps leading to the proof are given.

\begin{example}
  The hypothesis that $f$ and $g$ are stable is essential. For
  instance, let $f(x,y)=(x,y^{3}+xy)$ and $g(x,y)= (x,y^{3}).$ Both
  algebras $Q(f)$ and $Q(g)$ are isomorphic to
  $\frac{\mathcal{E}_{1}}{(y^{3})},$ but $f$ and $g$ are not
  $\mathcal{A}$-equivalent. In fact, $f$ is stable and
  $g$ is not.
\end{example}

The condition that $f\in \mathcal{E}_{n}^{p}$ is infinitesimally
stable is determined by its $p+1$-jet. In fact
the following holds:

\begin{proposition}[Mather \cite{Mat69-2}, Proposition I.I] \label{prop:69-2}
  The map-germ  $f:(\mathbb R^{n},S)\to (\mathbb R^{p},0)$ is
   stable if and only if
  \begin{equation}
    \label{eq:5}
    tf(\Theta_{(n,S)})+{\omega}f(\Theta_{p}) + (f^{*}(\mathcal{M}_{p}) + \mathcal{M}_{S}^{p+1})\Theta_{f}=\Theta_{f}.
  \end{equation}
\end{proposition}

\begin{proof}
  We need to show that \eqref{eq:5} implies
  $$
  tf(\Theta_{(n,S)})+{\omega}f(\Theta_{p})= \Theta_{f}.
  $$

  The proof is similar to the proof of Proposition  \ref{bpw} but simpler.

  Let $D=tf(\Theta_{(n,S)})+ f^{*}(\mathcal{M}_{p})\Theta_{f}.$ Note that
  $$
  {\omega}f(\mathcal{M}_{p}\Theta_{p})\subset
  f^{*}(\mathcal{M}_{p})\Theta_{f}\subset D.$$
  Then
  $$
  \dim_{\mathbb R}\frac{\Theta_{f}}{\mathcal{M}_{S}^{p+1}\Theta_{f}+D}\leq
  \dim_{\mathbb R}\frac{{\omega}f(\Theta_{p})}{{\omega}f(\mathcal{M}_{p}\Theta_{p})}
  \leq p.
  $$

  The result then follows by Lemma \ref{naka} (ii).
  \end{proof}

  \begin{remark}\label{rem:4.21}
    Mather gives  in \cite{Mat69-2}, Proposition (I.6), a simple
    geometric characterization of a stable multigerm $f:(\mathbb R^{n},S)\to (\mathbb R^{p},0),$ $S=\{x_{1},
    x_{2}, \dots, x_{r}\}.$
    Recall that if $V$ is a vector space and
    $H_{1}, \dots, H_{r}$ are subspaces of $V,$ then  $H_{1},\dots,
    H_{r}$ are in \emph{general position} if for every sequence of
    integers $i_{1}, \dots, i_{l}$ with $1\leq i_{1}\leq \dots \leq
    i_{l}\leq r$, we have $\cod (H_{i_{1}}\cap \dots \cap H_{i_{l}})
    =\cod (H_{i_{1}})+ \dots + \cod (H_{i_{l}}).$

    Let $f_{i}:U_{i}\to \mathbb R^{p},$ $i=1, \dots,
    r$ be a representative of the germ $f_{i}:(\mathbb R^{n}, x_{i})
    \to (\mathbb R^{p},0).$ Denote by $X_{i}=\{x \in U_{i}|\,
    (f_{i},x)\widesim{\mathcal{A}}(f_{i},x_{i})\}$ where $(f_{i},x) $
    denotes the germ $f_{i}:(\mathbb R^{n}, x)  \to (\mathbb
    R^{p},0),$ $i=1,\dots, r.$ Since $f$ is infinitesimally stable,
    the sets $X_{i}$ are submanifolds.
    Mather's result states that the multigerm $f$ is stable if and only if
     each branch $f_{i}:(\mathbb R^{n},x_{i})\to (\mathbb R^{p},0),$
     $i=1\dots r$ is infinitesimally stable and the images 
    $f_{i}(X_{i}), i=1, \dots, r$ are in general position.
     \end{remark}
  \begin{corollary}\label{cor:inf-sta}
     An infinitesimally stable germ $f:(\mathbb R^{n},0)\to (\mathbb R^{p},0)$ is  $(p+1)$-$\mathcal{A}$-determined.
  \end{corollary}

  \begin{proof}
    Notice that Proposition \ref{prop:69-2} implies that if
    $j^{p+1}g(0)=j^{p+1}f(0),$ then $g$ is also infinitesimally
    stable.

    It is also clear that every such $g$ is $\mathcal{A}$-finitely determined,
    say $l$-$\mathcal{A}$-determined. Then, we can apply Proposition
    \ref{th:eps} to get the result.
  \end{proof}
  As the local algebra is a complete invariant for the
  classification of stable germs, we can ask:

  \noindent -- Can we provide a normal form of a stable germ whose
  local algebra is a given algebra $Q?$

 {
  The answer  was given by  Mather \cite{Mat69-2} and we review
  it here (see also section 1.2.5 of the Mond and Nu\~no-Ballesteros
  in this Handbook \cite{Han}).
}

{  We start with a rank zero $\mathcal{K}$-finitely determined
  $f:(\mathbb R^{n},0)\to (\mathbb R^{p},0)$, $f=(f_{1}, f_{2},\dots,
  f_{p}).$ Let
 $$
  Q(f)=\frac{\mathcal{E}_{n}}{f^{*}(\mathcal{M}_{p})\mathcal{E}_{n}}=\frac{\mathcal{E}_{n}}
  {\langle f_{1},\dots,f_{p}\rangle\mathcal{E}_{n}}.
  $$
  Since $f$ is $\mathcal{K}$-finitely determined, the quotient
}

{  \begin{equation}
    \label{eq:33}
    Nf =\frac{\Theta_{f}}{tf(\Theta_{n})+f^{*}(\mathcal{M}_{p})\Theta_{f}+{\omega}f(\Theta_{p})}
  \end{equation}
is a finite dimensional $\mathbb R$-vector space of dimension $r$ and
we can choose $\sigma_{i}\in \mathcal{E}^{p}_{n}, i=1, \dots,r$ such
that
\begin{equation}
  \label{eq:14}
  Nf = \mathbb R\{\sigma_{1}, \dots, \sigma_{r}\},
\end{equation}
}
{For practical purposes, note that the vector space $Nf$ admits the
following simpler characterization:
$$
Nf\simeq \frac{\mathcal{M}_{n}\Theta_{f}}{tf(\Theta_{n}) +f^{*}\mathcal{M}_{p}\Theta_{f}}
$$}

Let $F:(\mathbb R^{n}\times \mathbb R^{r},0)\to (\mathbb R^{p}\times
\mathbb R^{r},0)$ be the linear $r$-parameter unfolding of $f$ defined
by
\begin{equation}
  \label{eq:17}
  F(x,u)=(f(x)+\sum_{i=1}^{n}u_{i}\sigma_{i}(x), u).
\end{equation}

Then $F$ is infinitesimally stable. In, fact from \eqref{eq:14} we get
$$
  \Theta_{f}=tf(\Theta_{n})+
  {\omega}f(\Theta_{p})+f^{*}(\mathcal{M}_{p})\Theta_{f}+\mathbb
  R\{\sigma_{1}, \dots, \sigma_{r}\},
$$
which implies that
\begin{equation}
  \label{eq:15}
  \Psi_{F}= tF(\Psi_{n+r})+
  {\omega}F(\Psi_{p+r})+F^{*}(\mathcal{M}_{p+r})\Psi_{F}+\mathcal{E}_{r}\{\sigma_{1},
  \dots, \sigma_{r}\},
\end{equation}
where $\mathcal{E}_{r}\{\sigma_{1},  \dots, \sigma_{r}\}$ denotes  the
$\mathcal{E}_{r}$-module generated by $ \{\sigma_{1}, \dots,
\sigma_{r}\}.$ Notice that
$F^{*}(\mathcal{M}_{p})\mathcal{E}_{n+r}\supset \langle u_{1},\dots,
u_{r}\rangle \mathcal{E}_{n+r}.$ Then, it follows from that
$$
\Theta_{F}=tF(\Theta_{n+r})+
  {\omega}F(\Theta_{p+r})+F^{*}(\mathcal{M}_{p+r})\Theta_{F},
$$
and it follows from Proposition \ref{prop:69-2} that $F$ is
infinitesimally stable.

\begin{example}
  \noindent (\textbf{a}) $\mathcal{A}_{k}$ singularities

  Let  $f:(\mathbb R,0)\to (\mathbb R,0),$ $f(x)=x^{k+1}.$ Then 
$  Nf = \mathbb R\{1, x, \dots, x^{k-1}\}.$ From the above
construction, we obtain that

\begin{align*}
 F:\mathbb R\times \mathbb R^{k-1} &\to \mathbb R\times
  \mathbb R^{k-1}\\
(x,u) &\mapsto F(x,u)=(x^{k+1}+\sum_{i=1}^{k-1}u_{i} x^{i}, u),
\end{align*}
is infinitesimally stable.

\noindent (\textbf{b}) $\Sigma^{2,0}$ singularities $B^{\pm}_{2,2}=(x^{2}\pm
y^{2}, xy)$

(We use here du Plessis and Wall notation \cite{PleWal}. They are
denoted $\text{I}_{2,2}=(x^{2}+ y^{2}, xy)$ and $\text{II}_{2,2}= (x^{2}- y^{2}, xy)$
by Mather \cite{Mat69-2}.)

Normal forms for infinitesimally stable singularities whose local algebra are
$B^{\pm}_{2,2}$ are
\begin{align*}
 F:(\mathbb R^{2}\times \mathbb R^{2},0) &\to (\mathbb R^{2}\times
  \mathbb R^{2},0)\\
(x,y,u,v) &\mapsto F(x,y,u,v)=(x^{2}\pm y^{2}+ ux+vy, xy, u, v).
\end{align*}
\end{example}

As a consequence of the  results of this section we can state the
following addendum to Theorem \ref{th:3.10}.

\begin{theorem}[Mather \cite{Mat70}, Addendum to Theorem 4.1]\label{th:11}
Let $r\leq p+1$ and $k\geq p.$ Let  $f:N\to P$ be a proper
$C^{\infty}$ mapping. Then the following conditions are equivalent
\begin{itemize}
\item [(a)] $f$ is stable.
\item [(b)] $ _{r}j^{k}f$ is transversal to every contact class in $ _{r}j^{k}(N,P).$
\item [(c)] For every subset $S$ of $N$ having $r$ or fewer points,
  such that $f(S)$ is a single point $y\in P,$ we have
  $$ tf(\Theta_{(N,S)}) +{\omega}f(\Theta_{(P,y)} ) +\mathcal{M}_{S}^{k+1}\Theta_{f}=\Theta_{f}$$
\end{itemize}
\end{theorem}

\subsection{Maps of finite singularity type}
\label{sec:maps-finite-sing}
\index{Finite singularity type}\index{FST|see{ finite singularity type}}
Another fundamental notion introduced by Mather in \cite{Mat71-3} was the
notion of mappings of finite singularity type, denoted by
FST. Properties of such mappings are also discussed in \cite{GibWirPleLoo}.

A mapping $f:N\to P$ will be said of \emph{finite singularity type} if
$E=\frac{\Theta_{f}}{tf(\Theta_{N})}$ is a finite module over
$C^{\infty}(P)$ via $f.$

We can also define similarly the notion of FST for multigerms
$f:(\mathbb R^{n}, S)\to (\mathbb R^{p}, 0).$

Local properties of mappings of finite singularity type follow from
our previous discussion. The critical set of $f$ is the set $\Sigma(f)$ of
non-submersive points  of $f.$

Let $F:(\mathbb R^{n}\times \mathbb R^{r},0) \to (\mathbb R^{p}\times
\mathbb R^{r},0)$ with $F(x,u)=(\bar f(x,u),u)$ and $\bar
f(x,0)=f(x).$ If $F$ is a stable germ, we say that $F$ is a
\emph{parametrized stable unfolding} of $f.$

\begin{theorem}
  Let $f:(\mathbb R^{n}, S)\to (\mathbb R^{p}, 0).$ The following are equivalent.
  \begin{itemize}
  \item [(1)]\  $f$ is of FST.
  \item [(2)]\ $f$ is $\mathcal{K}$-finitely determined.
  \item [(3)]\  $f$ admits a stable parametrized unfolding.
  \end{itemize}    
Moreover, these conditions imply
\begin{itemize}
\item [(4)]\  for every sufficiently small representative $f:U\to V,$
  $f|_{\Sigma(f)}:\Sigma(f) \to V$ is proper and has finite fibers.  
  \end{itemize}
\end{theorem}

\begin{remark}
  We say that $f:X\to Y$ has finite fibers (or, is finite-to-one) if
  for every $y\in Y,$ $ f^{-1}(y)$ has a finite number of points.
\end{remark}

\begin{proof}
  The equivalence $(1)\Leftrightarrow (2)$ follows from the Preparation
  Theorem. In fact $E=\frac{\Theta_{f}}{tf(\Theta_{(n,S)})}$ is a finitely
  generated $f^{*}(\mathcal{E}_{p})$-module if and only if
  $\mathcal{K}_{e}$-$\cod f=\dim_{\mathbb R}\frac {\Theta_{f}}{tf(\Theta_{(n,S)})+f^{*}(\mathcal{M}_{p})\Theta_{f}}
  < \infty.$

  We saw in section \ref{sec:class-stable-sing} that a
  $\mathcal{K}$-finitely determined germ has a stable unfolding; so
  that $(2)\Rightarrow (3).$ We saw in Example \ref{ex:6.5} that
  $Q(f)=Q(f_{0}),$ so that $(3)\Rightarrow (2).$

  It is sufficient  to prove  $(4)$ for infinitesimally stable
  germs. In this case, the general position condition implies that  for any $y\in V,$
  $f^{-1}(y) \cap \Sigma(f)$ has at most $p$ points (see Remark \ref{rem:4.21}).
\end{proof}

We shall need some extra conditions to formulate the theory of
FST mappings $f:N\to P.$ The condition that $f$ has a parametrized
stable unfolding  is fairly easily computable, but it does
not always have a global version (see Mather \cite{Mat71-3} for
counter examples).

\begin{definition}\index{Unfolding}
  Let $f:N\to P$ be  smooth. We say that $\{F,N', P', i, j\}$ is an \emph{unfolding}
  of $f$ if we have a commutative diagram
  $$
  \xymatrix{
    N'  \ar[r]^{F}    & P'  \\
    N \ar[u]_{i} \ar[r]^{f}   & P\ar[u]_{j} }
  $$
where $N', P'$ are smooth manifolds, $F$ is a smooth mapping, $i, j$
are closed smooth embeddings, $i(N)= F^{-1}(j(P))$ and $F$ is
transverse to $j.$
\end{definition}

\begin{theorem}[Mather \cite{Mat76}, Proposition 7.2]

  Let $f:N \to P$ be smooth and $N$ compact. Then $f$ is of finite
  singularity   type  if and only if there exists an unfolding
  $\{F,N', P', i, j\}$ of $f$ such that $F$ is proper and
  infinitesimally stable.
\end{theorem}

\subsection{Notes}

All the results in this section  remain
true if we replace smooth germs by real analytic or complex analytic germs.
\index{Finite determinacy!of Mather's groups $\mathcal{G}$}
In particular, the notion of $\mathcal{G}$-finite determinacy for $\mathcal{G}=
\mathcal{R}, \mathcal{L}, \mathcal{A}, \mathcal{C}$ and $\mathcal{K}$
is independent of whether we consider $f$ as a real analytic,
$C^{\infty}$ or complex analytic map-germ. The Infinitesimal Criterion
of $\mathcal{G}$-finite determinacy holds with essentially the same
proof replacing Malgrange Preparation Theorem by Weirstrass
Preparation Theorem.
\index{Finite determinacy!infinitesimal criterion of}
We use the same notation $\mathcal{O}_{n}$ for the local rings of real
analytic or complex analytic map-germs at the origin. The maximal
ideal in both cases is also denoted by $\mathcal{M}_{n}.$ The set
$\mathcal{O}_{n}^{p}$ denotes the $\mathcal{O}
_{n}$-module of real or complex analytic map-germs from $(\mathbb K^{n},0)
\to (\mathbb K^{p},0),$ $\mathbb K=\mathbb R$ or $\mathbb C.$ The
following result explains the relation among finite determined germs
in these different modules.

\begin{proposition}
  Let $f:(\mathbb R^{n},0) \to (\mathbb R^{p},0) $ be a real analytic
  map-germ. The following are equivalent
  \begin{itemize}
  \item [(i)]\  $f$ is $k$-$\mathcal{G}$-determined in the space of real
    analytic map-germs $\mathcal{O}_{n}^{p}.$
  \item [(ii)]\  $f$ is $k$-$\mathcal{G}$-determined in $\mathcal{E}_{n}^{p}.$
  \item [(iii)] \hspace{1mm} The complexification of $f$,
    $f_{\mathbb C}:(\mathbb C^{n},0) \to
(\mathbb C^{p},0),$ is $k$-$\mathcal{G}$-determined in the space
$\mathcal{O}_{n}^{p}$ of holomorphic map-germs.
\end{itemize}
\end{proposition}

In the complex case there are  useful geometric characterization of
$\mathcal{G}$-finite determinacy. The main result characterizes
$\mathcal{G}$-finite determined germs as map-germs with isolated
instability. The case $\mathcal{G}=\mathcal{A}$ was stated by Mather
and proved by Gaffney. For a complete account we refer to Wall
\cite{Wal81} or Mond and Nu\~no-Ballesteros \cite{MonNun}. See also
Mond and Nu\~no-Ballesteros article in this Handbook \cite{Han}
\index{Finite determinacy!Geometric criterion of}
\begin{theorem}[Geometric criterion of finite determinacy]
  A holomorphic map-germ $f:(\mathbb C^{n},0) \to
(\mathbb C^{p},0),$ is $\mathcal{A}$-finite  if and only if
there is a neighborhood $U$ of $0$ in $\mathbb C^{n}$ such that for
every finite subset $S\subset U\setminus \{0\},$ the multigerm of $f$
at $S$ is $\mathcal{A}$-stable.
\end{theorem}

The geometric condition of this theorem (isolated instability)
holds for any real $\mathcal{A}$-finite map-germ. However, the converse
statement does not hold. For a simple example, let $f(x,y)= (x^{2}
+y^{2})^{2}.$ As $\Sigma (f)=\{0\},$ the origin is an isolated
instability, but $f$ is not $\mathcal{A}$-finitely determined.

\section{The nice dimensions}\index{Nice dimensions|(}
\label{sec:nice-dimensions}

We discuss in this section the main steps in the proof of theorem A.
Mather proved in \cite{Mat69-2} that for a pair of positive integers $(n,p),$ there exists a
smallest Zariski closed $\mathcal{K}^{k}$-invariant set
$\Pi^{k}(n,p)$ in the set $J^{k}(n,p)$
such that 
$J^{k}(n,p) \setminus \Pi^{k}(n,p)$  is the union of finitely many
$\mathcal{K}^{k}$-orbits.
The set $\Pi^{k}(n,p)$ is the ``bad set.'' It is in fact the set of
$k$-jets in  $J^{k}(n,p)$  of ``modality'' ( $\mathcal{K}$-modality)
greater than or equal to $1$ (see Section \ref{sec:thom-math-strat}
for the definition of modality).

We review Mather's construction of $\Pi^{k}(n,p).$ For each $r, k \in
\mathbb N$ we define $W^{k}_{r}(n,p)$ as the set of $z\in J^{k}(n,p)$
such that ${\mathcal K}^{k}$-$\cod{z} \geq r.$ This set is a closed
algebraic subset of $J^{k}(n,p).$  Let $W^{k}_{r}(n,p)^{*}$ denote the
union of all irreducible components of $W^{k}_{r}(n,p)$ whose
codimension is less than $r.$ We let
$\Pi^{k}(n,p)=\cup_{r\geq 0}W^{k}_{r}(n,p)^{*}.$ The following
properties hold:
\begin{itemize}
\item $\Pi^{k}(n,p)$ is a closed algebraic subset of $J^{k}(n,p).$
\item Let $\pi_{k}:J^{k+1}(n,p) \to
J^{k}(n,p)$ be the projection . It follows that
$\pi^{-1}_{k}(\Pi^{k}(n,p))\subset \Pi^{k+1}(n,p),$ hence $\cod
\Pi^{k+1}(n,p)\leq \cod \Pi^{k}(n,p).$ 
\item There exists a $k$ big enough for which the codimension
 of $\Pi^{k}(n,p)$ attains its minimum. For this $k$, $\cod
 \Pi^{k}(n,p)$ is denoted $\sigma(n,p).$
\end{itemize}
Mather calculated $\sigma(n,p)$ in \cite{Mat71} and the result is as follows:

 \noindent Case 1: $n\leq p$
 \begin{align*}
    \sigma(n,p) &=
 \begin{cases}
   6(p-n)+8 &\textrm{if } p-n \geq 4\, \textrm{ and }\, n\geq 4\\
   6(p-n)+9 &\textrm{if } 3\geq p-n \geq 0 \, \textrm{ and }\, n\geq 4\,
   \mathrm{ or\  if }\, n=3\\
   7(p-n) +10 &\textrm{if } n=2\\
   \infty &\textrm{if } n=1
 \end{cases}
\intertext{Case 2: $n>p$}
\sigma(n,p) &=
  \begin{cases}
    9 &\textrm{if } n=p+1\\
    8 &\textrm{if } n=p+2\\
    n-p+7 &\textrm{if } n\geq p+3\\
  \end{cases}
 \end{align*}

 \begin{figure}[h]
	\begin{center}
          \def\svgwidth{0.7\textwidth}
          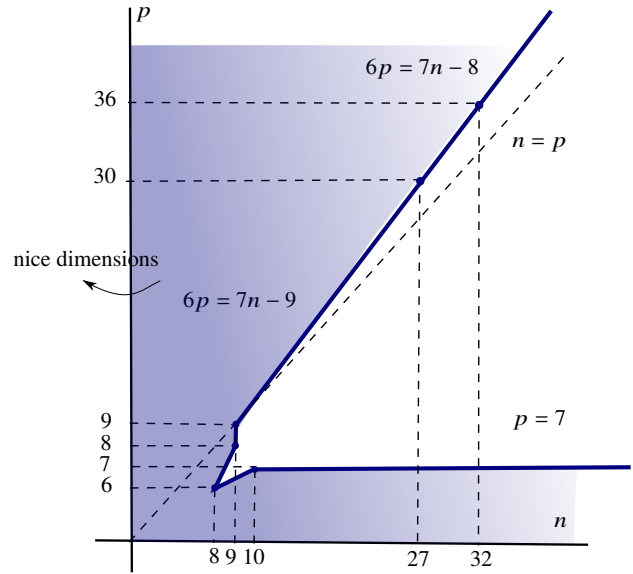
		\caption{Boundary of nice dimensions}
		\label{Sg}
	\end{center}
\end{figure}

 \begin{definition}\index{Nice dimensions}
   A pair $(n,p)$ is in the \emph{nice dimensions} if $n<\sigma(n,p).$
 \end{definition}

 Suppose $k$ has the property that  $\cod \Pi^{k}(n,p)=\sigma(n,p).$
If $(n,p)$ is in the nice dimensions, then there exists an
analytically trivial 
stratification $\mathcal{S}^{k}(n,p)$ of $J^{k}(n,p)\setminus \Pi^{k}(n,p)$ such that the
strata are a finite number of $\mathcal{K}$-orbits. To get a 
stratification of the whole jet space $J^{k}(N,P),$ we add to
$S^{k}(n,p)$  a Whitney regular stratification of $\Pi^{k}(n,p)$ (it
exists since $\Pi^{k}(n,p)$ is an algebraic closed set of
$J^{k}(n,p)$).

This stratification of $J^{k}(n,p)$ induces a
partition of $J^{k}(N,P)$ by $\mathcal{K}$-orbit bundles whose
restriction to $J^{k}(N,P)\setminus \Pi^{k}(N,P)$ is 
denoted by $S^{k}(N,P).$ 

As we saw in Theorem \ref{th:11}, stable mappings can be characterized
by transversality of the 
$k$-jet extension $j^{k}f:N\to J^{k}(N,P)$ to the $\mathcal
K^{k}$-orbits.

When $ \sigma(n,p)>n,$ transversality  to the strata
of the stratification $J^{k}(N,P),$ implies that 
 $j^{k}f(N)\cap\Pi^{k}(N,P)=\emptyset.$  Hence Theorem A  follows
from Thom's transversality theorem.
\index{Stable singularities!when $n=p\leq 8$}
\begin{example}[Stable singularities when $n=p\leq 8$]
  We refer to \cite{MonNun} for the list of stable singularities in
  the nice dimensions.

  When $n=p,$ $\sigma(n,p)=9$, then $(n,n)$ is a nice pair of
  dimensions if and only if $n\leq 8.$
  The set $\Pi^{k}(n,n)\subset J^{k}(n,n),$ $k\geq n+1,$ $n\leq 8$ is
  the closure of all ${\cal K}^{k}$-orbits of ${\cal
    K}^{k}$-codimension greater than or equal to $n+1.$ In particular,
  $\Sigma^{3}(n,n) \subset \Pi^{k}(n,n),$   where $n\leq 8$ since
  $\mathrm{cod}\, \Sigma^{3} =9.$ The strata of the stratification
  ${\cal S}^{k}(n,n),$ $k\geq n+1,$ $n\leq 8$ are presented  in
  Table \ref{tab:0}:

  \begin{table}[h]
  \centering
  \begin{tabular}[center]{|c|c|c|c|c|}
    \hline
    Type & Name & Normal form & Conditions &
                                             $\mathcal{K}$-$\textrm{cod}
                                             \leq n$\\
    \hline \hline
$\Sigma^{1}$ &$A_{j}$         &$(x^{j+1})$ & $ 1\leq j\leq n$ & $j$ \\\hline
$\Sigma^{2,0}$ &$B^{\pm}_{p,q}$ &$(xy, x^{p}\pm y^{q})$ & $2\leq p,q \leq n-2$ &
                                               $p+q $ \\\hline
$\Sigma^{2,0}$ & $B^{*}_{p,p}$ & $(x^{2}+y^{2}, x^{p})$  & $3 \leq p\leq 4$
                                           & $2p$ \\\hline
$\Sigma^{2,1}$ &$C_{2k-1}$     &$(x^{2}+y^{3},y^{3})$  &  &
                                                             $7$\\\hline
$\Sigma^{2,1}$ & $C_{2k}$      &$(x^{2}+y^{3},  xy^{2})$ &
                                           &$8$ \\
\hline    \hline
  \end{tabular}
  \caption{${\cal K}$-orbits of stable germs $n=p \leq 8$}
  \label{tab:0}
\end{table}

\end{example}

\begin{remark}\index{Stable singularities! in the nice dimensions}
  \emph{Classification of stable singularities in the nice dimensions.}
{Mather classified the stable germs in the nice dimensions as an 
application of results and arguments in \cite{Mat71}.
 He gave complete proofs of the classification of the local algebras
 of singularities of type $\Sigma^{1}$ and $\Sigma^{2,0}$ and outlined
the classification of $\Sigma^{2,1}$ and $\Sigma^{n-p+1}$
singularities. Further classification of simple and unimodular
algebras were performed by Arnold \cite{Arn75}, Wall \cite {Wal85},
Dimca and Gibson \cite{DimGib79, DimGib, DimGib85} and Damon
\cite{Dam75, Dam79-1,Dam79-3}.}

{A remarkable property  of stable map-germs in the nice dimensions is
that, with  respect  to suitable
coordinates, all singularities are weighted homogeneous. For many
years, this property was considered to be true but there was no
reference of a written proof.}

{This result was recently proved by Mond and Nu\~no-Ballesteros \cite{MonNun}
theorem 7.6. Their proof is based on Mather's classification of local
algebras of stable germs in the nice dimensions and on the direct
construction of the normal forms of their minimal stable
unfoldings. This property of the nice dimensions plays a crucial role
in the proof of Damon and  Mond \cite{DamMon} that the
$\mathcal{A}_{e}$-codimension is less than or equal to the rank of the
vanishing homology of the  discriminant ( the discriminant Milnor
number) for map germs $(\mathbb C^{n},0) \to (\mathbb C^{p},0)$ with
$n\geq p$ and $(n,p)$ nice dimensions.}
\end{remark}
\index{Nice dimensions|)}
\subsection{Notes}
\label{sec:notes-1}
\index{Map!non-proper stable}
\emph{Non proper stable mappings.} $C^{\infty}$ non-proper stable mappings were discussed by du Plessis
and Vosegaard \cite{PleVos} and more recently by Kenta Hayano
\cite{Hay}.

For proper maps $f:N\to P,$ Mather proves that stability,
strong stability, infinitesimal stability an local infinitesimal
stability are equivalent notions.
In \cite{PleVos}, du Plessis and Vosegaard prove that these notions
are equivalent  when $f$ is a quasi-proper map with closed discriminant.

The purpose of Hayano's paper, \cite{Hay}, is to give a sufficient
condition for strong stability of non-proper smooth functions $f:N\to
\mathbb R.$ He introduces the notion of \emph{end-triviality} of smooth
mappings, which controls the behavior of $f$ around the ends of the
source manifold $N.$ He shows that a Morse function is stable if it is
end-trivial at any point in its discriminant.
\index{Extra-nice dimensions}
\emph{The extra-nice dimensions.}  When the pair $(n,p)$ is in the
nice dimensions and 
the source $N$ is compact, an important problem  in the applications
of singularity theory  to topology of manifolds is the
characterization of generic singularities of $1$-parameter  paths
between two stable maps; they are also known as \emph{pseudo-isotopies.} A
$1$-parameter family $F:N\times[0,1] \to P$ connecting two
non equivalent stable maps always intersects the set of non stable maps
at a finite number of values of the parameter, the bifurcation
points. The classification of singularities of bifurcation points in
generic families of maps is an important  step in results on
elimination of singularities (see for instance \cite{Lev,BehHay}) and
on results about the topology of the space of smooth maps such as
\cite{Cer,Igu,Vas}.

We say that a family $F:N\times[0,1] \to P$ is a \emph{locally stable family}
if $F_{t}:N \to P$ is stable for all $t\in [0,1]$ except possibly a
finite number of values $\{t_{1}, \dots, t_{k}\}$ and the non stable
singularities of $F_{t}$ are a finite number of points $x_{j}$ at
which $\mathcal{A}_{e}$-cod$(F_{t_{i}})=1.$

In \cite{AtiRuaSin} Sinha, Ruas and
Atique obtain a result parallel to Mather's characterization of the
nice dimensions. They define the \emph{extra-nice dimensions} and (see
Figure \ref{fig:extra})
prove that the subset of stable $1$-parameter families in
$C^{\infty}(N\times [0,1],P)$ is dense if and only $(n,p)$ is in the
extra-nice dimensions.

\begin{figure}[h]
  \centering
 \scalebox{0.40}{\includegraphics{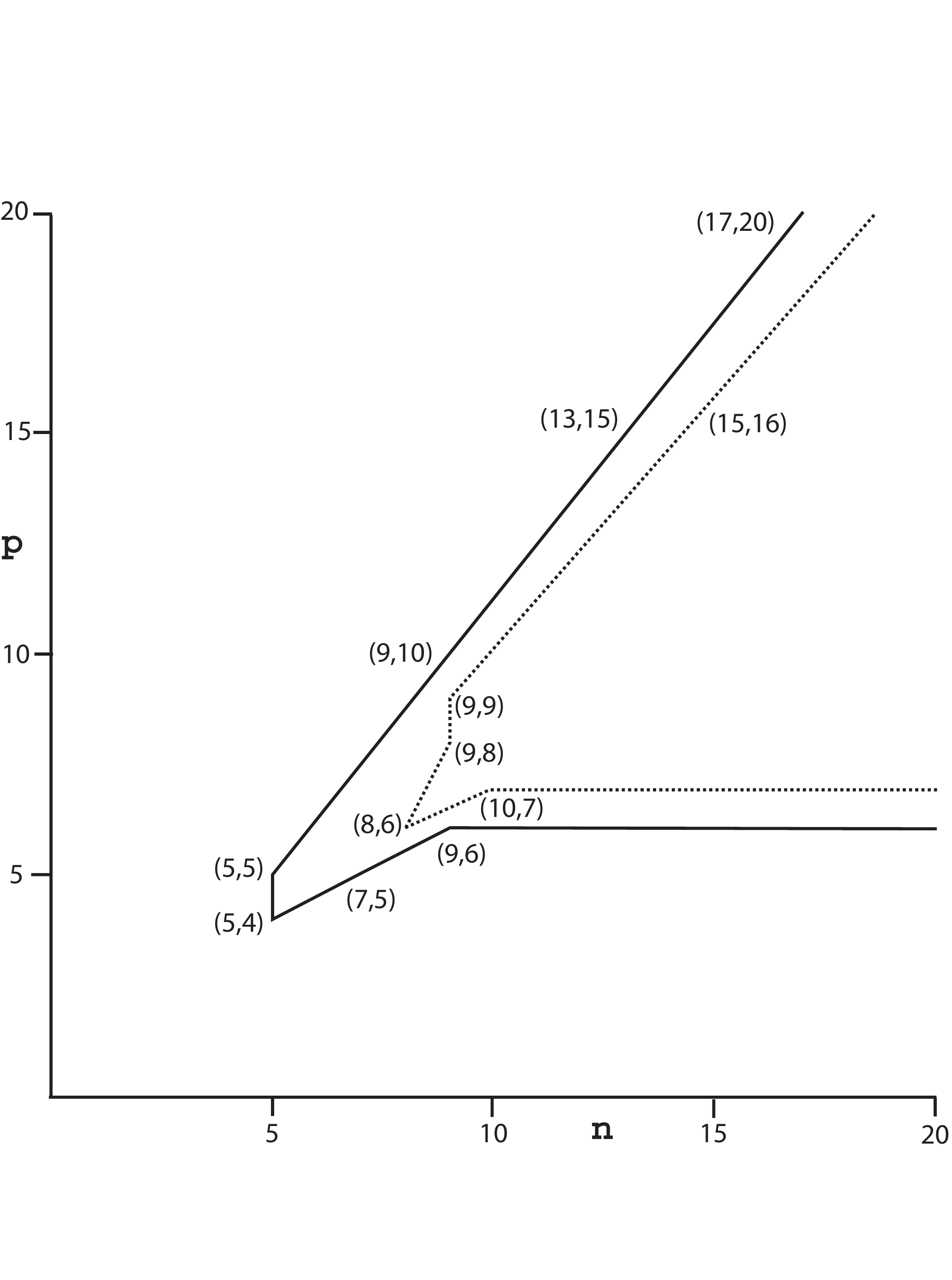}}
  \caption{Extra-nice dimensions}
  \label{fig:extra}
\end{figure}

In section \ref{sec:open-problems} we relate the condition that
$(n,p)$ is in the extra-nice dimensions to the geometry of sections of
the discriminant of stable maps in dimensions $(n+1,p+1).$

\section{Thom's example}
\label{sec:thoms-example}

If a pair of dimensions $(n,p)$ is not in the nice range of dimensions,
then there exists an open non void subset $U$ of $C^{\infty}(N,P),$
such that $U$ is the union of an uncountable number of $\mathcal{A}_{e}$-orbits.
This property was first proved by Ren\'e Thom when $n=p=9.$ We review
Thom's example (\cite{Tho68}) here.
\index{Nice dimensions!boundary of}
The pair $n=p=9$ is in the \emph{boundary of the nice dimensions},
which consists  of pairs $(n,p)$ such that $\sigma(n,p)=n.$

\label{ExemThom}
  The construction of Thom's example was based on the following
  \begin{enumerate}
  \item The set of mappings $F:N\to P,$ $\dim N= \dim P=n,$ such that $j^{k}F\pitchfork
    \Sigma^{r}(N,P),$ where $\Sigma^{r}(N,P)=\{\sigma \in J^{k}(N,P)|\, \mathrm{corank}\,
    \sigma =r\},$ $0\leq r\leq n$ is a residual set of $C^{\infty}(N,P).$
  \item $\cod_{J^{k}(N,P)}\Sigma^{r}(N,P)= r^{2}.$
  \item When $r=3,$ $n=9,$ there exists a $1$-parameter family of non
    $\mathcal{K}$-equivalent 
    mappings $F_{\lambda}:\mathbb R^{9} \to \mathbb R^{9},$ such that
    $j^{k}_{1}F:\mathbb R \times {\mathbb R}^{9} \to J^{k}(\mathbb R^{9},
    \mathbb R^{9})$ is transversal to $\Sigma^{3}(\mathbb
    R^{9},\mathbb R^{9}),$ where $j^{k}_{1}F$ denotes the $k-$jet with
    respect do the variable $x.$
  \end{enumerate}

  The sets $\Sigma^{r}$ are the first order Boardman symbols and it is
  an easy exercise to prove that they are codimension $r^{2}$
  submanifolds of $J^{k}(N,P)$ when $\mathrm{dim}(N)=\mathrm{dim}(P).$ Hence (1) follows from Thom's
  transversality theorem.

  It is sufficient to verify (3) for map-germs $F:(\mathbb R^{9},0)
  \to (\mathbb R^{9},0),$ such that $\mathrm{corank}\;F(0)=3.$ By changing
  coordinates in source and target, it follows that $F$ can be written
  in the form $F(x,u)=(f(x,u),u),$ $x=(x_{1}, x_{2}, x_{3}),$
  $u=(u_{1}, \dots, u_{6}), $ $ f_{0}(x)=f(x,0),$ where
  $f_{0}:(\mathbb R^{3},0) \to (\mathbb R^{3},0)$ has zero rank.

  The local algebras $Q(F)$ and $Q(f_{0})$ are isomorphic. As we saw
  in Example \ref{ex:6.5}, $F$ is $\mathcal{K}$-equivalent  to a suspension of $f_{0}.$
  The $2$-jet $j^{2}f_{0}$ is a quadratic polynomial mapping
  $q:\mathbb R^{3} \to \mathbb R^{3},$ which determines a net of real
  quadrics.
\index{Net of quadrics}
 Non degenerate nets of quadrics over the complex numbers were
 classified by C. T. C. Wall in \cite{Wal80}. Over the reals, the
 classification was given by Wall and Edwards in \cite{EdwWal}. The
 complete classification of real nets of quadrics can be found in \cite{PleWal}
 chapter 8, table 8.21.

 For our purpose here, it suffices to remark that the set
 $\Sigma^{3,3}$ has a Zariski open set, denoted by
 $W_{2},$   defined by the union of the
 $J^{2}\mathcal{K}$-orbits of the unimodular family:

 \begin{align}
   \label{eq:1}
   (f_{0})_{\lambda}:&(\mathbb R^{3},0) \to (\mathbb R^{3},0)\\
   &(x_{1},x_{2},x_{3})\to (x_{1}^{2}+\lambda
     x_{2}x_{3},x_{2}^{2}+\lambda x_{1}x_{3},x_{3}^{2}+\lambda x_{2}x_{3}) \nonumber
 \end{align}
 with $\lambda(\lambda^{3}+8)(\lambda^{3}-1)\neq 0.$

 For each $\lambda,$ $(f_{0})_{\lambda}$ is a homogeneous polynomial
 map of degree $2,$ hence the $J^{2}\mathcal{K}$-action in $W_{2}$
 coincides with the action of the linear group
 $\mathcal{G}=GL(3)\times GL(3)$ on $W_{2}.$ Notice that the dimension
 of the linear group $\mathcal{G}$ is $18,$ as well as the dimension
 of $W_{2}.$

 However $\mathcal{G}$ contains a one dimensional subgroup which acts
 trivially on $W_{2},$ namely $\{(cI_{\mathbb R^{3}}, \frac 1
 {c^{2}}I_{\mathbb R^{3}})\},$ $c$ a non zero number. Hence the orbits
 have codimension at least $1$ in $W_{2}.$

 We can prove that the family \eqref{eq:1} is $2$-determined with respect
to $\mathcal{K}$-equivalence. It follows that $W_{2}$ determines the
$\mathcal{K}$-invariant sets $W_{2}^{k}=(\pi_{2}^{k})^{-1}(W_{2}),$
where $\pi^{k}_{2}:J^{k}(9,9)\to J^{2}(9,9).$ Moreover,
$\cod_{J^{k}(9,9)} W_{2}=9,$ and $\mathcal{K}$-$\cod
(f_{0})_{\lambda}=10.$

In other words, $\sigma(9,9)= \cod W_{2}=9,$ so that the
unimodular stratum $W_{2}$ cannot be avoided by a generic set of
proper mappings $F:\mathbb R^{9} \to \mathbb R^{9}.$ As a consequence,
stable mappings are not dense when $n=p=9.$

For each $\lambda \notin \{0, -2, 1\},$ $(f_{0})_{\lambda}$ admits the topologically stable
unfolding 

 \begin{align}\label{nove}
   F_{\lambda}:&(\mathbb R^{9},0) \to (\mathbb R^{9},0)\\
   &(x,u) \mapsto (f_{\lambda}(x, u), u) \nonumber
 \end{align}
 where $f_{\lambda}(x,u)=( x_{1}^{2}+\lambda x_{2}x_{3} +u_{1}x_{2} +u_{2}x_{3},
 x_{2}^{2}+\lambda x_{1} x_{3} + u_{3}x_{1} +u_{4}x_{3}, x_{3}^{2}+\lambda x_{1}x_{2} + u_{5}x_{1} +u_{6}x_{2}).$

 We will discuss the topological stability of $F_{\lambda}$ in section \ref{sec:bound-nice-dimens}.

\section{Density of topologically stable mappings}
\label{sec:dens-topol-stable-1}

From the previous example, it becomes clear that outside the nice
dimensions, one has to loosen the formulation of Problem \ref{pro} to
obtain a solution. Mather considered in \cite{Mat71-2} two
possible ways.

One might hope that the space of mappings $f$ whose germ $f_{x}$ at
each point $x\in N$ is $\mathcal{A}$-finitely determined is an open
and dense subset in $C^{\infty}_{pr}(N,P).$ However, Mather gave in
\cite{Mat69-3} an example which shows that this set is not always
dense. { In \cite{Ple} du Plessis defined the \emph{semi-nice dimensions} as the
pairs $(n,p)$  for which finite determinacy holds in general (see
Definition \ref{nova}). } The complement of the semi-nice dimensions is
essentially made of pairs $(n,p)$ where singularities of
$\mathcal{K}$-modality greater than or equal to $2$ occur generically (see
\cite{Ple}, \cite {Wal85}).

The second way to try to solve the problem is based on ideas due to
Thom, and led to Theorem B on density of $C^{0}$stable mappings
in $C_{pr}^{\infty}(N,P).$

In his article \emph{Local topological properties of differentiable
  mappings} \cite{Tho64}, Thom describes the topological
structure of differentiable mappings, outlining the proof of the
topological stability theorem.

\begin{theorem}[Theorem 4, \cite{Tho64}]
  Let $z$ be any jet in $J^{r}(n,p).$ Then, there exists a positive
  integer $s$ depending only on $r, n$ and $p,$ and a proper algebraic
  variety $\Sigma$ in $\pi_{s}^{-1}(z) \subset J^{r+s}(n,p)$ such that
  any jet in $\pi_{s}^{-1}(z)$ outside $\Sigma$ is
  $C^{0}$-$\mathcal{A}$-finitely determined. Moreover, any two
  mappings realizing such jet are locally weakly stratified and
  isotopic. 
  \end{theorem} 

  A complete proof of this theorem  follows from the proof of the \emph{Main
    Theorem} in A. Varchenko's article with the same title, \emph{Local topological
    properties of differentiable mappings} \cite{Var74} ( see also
  \cite{Var73, Var75}). He also proves in \cite{Var74} a 
  \emph{stratification theorem}, although he states in the paper he
  does not know
  whether  Mather's density theorem follows from his
stratification theorem, or whether the stratification theorem can be
proved by Mather's methods.

Mather gave in 1970, an outline of a complete proof of Theorem
B. His proof was published in the Proceedings of the
Symposium of Dynamical Systems, held in Salvador, Bahia
\cite{Mat71-2}. As remarked by him, he expected to publish a book in
which the details of the proof would appear. In the Spring 1970, he
gave a series of lectures  and the notes appeared as a booklet
published  in the same year by the Harvard Printing Office.
The notes also discuss the Thom-Whitney theory of stratified sets
and stratified mappings. They were recently republished in the
Bulletin of the American Mathematical Society \cite{Mat12}.

Complete proofs of Theorem B  were given in 1976, independently, by Gibson,
Wirthm\"{u}ller, du Plessis and Looijenga in \cite{GibWirPleLoo} and by
Mather in \cite{Mat76}. Both proofs are based on Thom's ideas and Mather's outline
\cite{Mat71-2}. In what follows we refer to Theorem B as the Thom-Mather
theorem. 

The book \cite{GibWirPleLoo} comprises the notes of a seminar on
Topological Stability of Smooth Mappings held at the Department of
Pure Mathematics in the University of Liverpool, during the academic
year 1974-75. The main objective was to organize a complete proof of
the Topological Stability Theorem, for which no published complete
account existed.
The book has become a fundamental reference on the subject.

The proof in \cite{GibWirPleLoo} and \cite{Mat76} are similar and they
rely on the following ingredients:

\begin{enumerate}
\item [(1)] Properties of Whitney regular stratifications
\item [(2)] \L ojasiewicz theorem, giving the existence of Whitney
  regular stratification of semialgebraic sets.
\item [(3)] Properties of stable mappings and mappings of finite
  singularity type (FST). A fundamental property of mappings of FST is
  the existence of a stable unfolding.
\item [(4)] Thom's second isotopy theorem, applied to show that
  families of mappings transverse to the Thom-Mather stratification are
  topologically trivial.
\end{enumerate}

For a review of stratification theory and Thom's isotopy theorems
in the differentiable category, we also refer to the paper by David Trotman, in
Volume I of this Handbook. We only make a brief presentation of basic
concepts and results.

\index{Stratification!$C^{k}-$}
Let $V$ be a subset of a smooth manifold $N$ of class $C^{k}.$ A
$C^{k}$-\emph{stratification} of $V$ is a filtration by closed subsets
$$
V=V_{d}\supset V_{d-1}\supseteq \cdots \supseteq V_{1}\supseteq V_{0}
$$
such that each difference $V_{i}\setminus V_{i-1}$ is $C^{k}$-manifold
of dimension $i$, or is empty. Each connected component of
$V_{i}\setminus V_{i-1}$ is a \emph{stratum} of dimension $i.$
It follows that $V$ is disjoint union of
strata $\{X_{\alpha}\}_{\alpha\in A},$ and we say that $V$ is a
\emph{stratified set.}

For the purposes of these notes we assume that the stratified sets
$V=\cup_{\alpha \in A} X_{\alpha}$ are \emph{locally finite} and
satisfy the \emph{frontier condition} (see Gibson et al. book
\cite{GibWirPleLoo} or Trotman \cite{Tro} for the definition).

Let $V$ be a subset of $\mathbb R^{n}$ and $\{X_{\alpha}\}_{\alpha\in
  A}$ a stratification of $V.$ Whitney defined regularity conditions
(a) and (b), seeking for stratifications topologically trivial
along strata.
\index{Whitney!conditions (a) and (b)}
\begin{definition}[Whitney's conditions (a) and (b)]
  Let $X$ and $Y$ be strata of $\{X_{\alpha}\}_{\alpha\in  A},$ such
  that $Y\subset \overline X \setminus X.$
  \begin{itemize}
  \item [(a)]\ The pair $(X,Y)$ satisfies Whitney's condition (a) at
    $y\in Y$ if: for all sequences $(x_{m}) \in X$ with $x_m \to y,$
    such that $T_{x_{m}}X$ converges to a subspace $T \subset \mathbb
    R^{n}$ ( in Grassmannian of $\dim X$- planes in $\mathbb R^n$),
    then $T\supset T_{y}Y.$ 
  \item [(b)]\ The pair $(X,Y)$ satisfies Whitney's condition (b) at
    $y\in Y$ if: for all sequences $(x_{m}) \in X$ and $(y_{m}) \in Y,$ with $x_m \to y,$
    $y_m \to y,$ such that $\{T_{x_{m}}X\}$ converges to  $T$ and the
    lines $\overline{x_{m}y_{m}}$ converges to a line $\ell$ one has
    $\ell \in T.$
  \end{itemize}
 { It was pointed out by Mather in his notes on topological stability
  that Whitney (b) implies Whitney (a).} The reader may verify this as
  an exercise. We say that the stratification is \emph{Whitney regular} if every pair
of strata $(X_{\alpha},X_{\beta})$ satisfies  (b) ( hence also
satisfies (a)) at every point in $X_{\beta}.$
\end{definition}

These regularity conditions are local and can be easily  extended to
stratified sets of a manifold $N.$

Whitney \cite{Whi65-1, Whi65-2} proved in 1965  that any analytic
variety in $\mathbb R^{n}$ or $\mathbb C^{n}$ admits a regular
stratification whose strata are analytic. This result was extended to
semi-analytic sets by \L ojasiewicz \cite{Loj65}, also in 1965. For the
purposes of this section, the relevant result is the existence theorem
for semialgebraic sets. We refer to Thom \cite{Tho95} and Wall
\cite{Wal75} for accessible proofs.

\begin{definition}\label{def:str}
  Let $f:N\to P$ be a smooth mapping and $A\subseteq N, B\subseteq P$
sets with $f(A)\subset B.$ A \emph{stratification} of $f:A \to B$ is a
pair $(\mathcal{X},\mathcal{X}'),$ such that $\mathcal{X}$ is a
Whitney stratification of $A$, $\mathcal{X}'$ is a Whitney
stratification of $B,$ and the following conditions hold
\begin{itemize}
\item f maps strata to strata.
\item If $X\in \mathcal{X},$ $X' \in \mathcal{X}'$, $f(X) \subset X'$
  then $f:X \to X'$ is a submersion.
\end{itemize}
\end{definition}
\begin{definition}
  Let $f:N\to P$  and $\mathcal{X}$ and $\mathcal{X}'$ as in
  definition \ref{def:str}. Given $X_{\alpha}, X_{\beta}$ strata of
  $\mathcal{X},$ $x\in X_{\beta}$ we say that $X_{\alpha}$ is
  \emph{Thom regular} over $X_{\beta}$ at $x\in X_{\beta}$ relative
  to $f$ when the following holds: for every sequence $(x_{i}) \in
  X_{\alpha},$ $x_{i}\to x$ such that
  $\textrm{ker\;}(d_{x_{i}}(f|_{X_{\alpha}}))$ converges to $T$ in the
  appropriate Grasmannian, then $\textrm{ker\;}d_{x}(f|_{X_{\beta}})
  \subseteq T.$ We say that \emph{$X_{\alpha}$ is Thom regular over
    $X_{\beta}$ relative to $f$} when this condition hold for all
  $x\in X_{\beta}.$The pair
  $(\mathcal{X},\mathcal{X}')$ is a \emph{Thom stratification} for $f$
  when Thom's regularity condition holds for all pair of strata
  $(X_{\alpha}, X_{\beta})$ with $X_{\beta}\subset \overline
  X_{\alpha}.$ The triple $(f,\mathcal{X},\mathcal{X}')$ with $f$ a
  smooth mapping and $(\mathcal{X},\mathcal{X}')$ a Thom
  stratification for $f$ is called a \emph{Thom stratified mapping}.\index{Map!Thom stratified}
\end{definition}

\subsection{How to stratify mappings and jet spaces}
\label{sec:how-strat-mapp}

\index{Stratification!Thom-Mather}
We first discuss the Thom-Mather stratification in jet space and how
 to stratify stable mappings and mappings of finite singularity
 type. Then, we discuss why mappings transverse to  the Thom-Mather
 stratification are topologically stable.

 The idea of the proof is to  construct a stratification
 $\mathcal{A}^{l}(N,P),$ of  a big open subset of $J^{l}(N,P),$ with
 the following property: if $l$ is sufficiently large, then for any
 mapping $f:N\to P$ which is multitransverse to
 $\mathcal{A}^{l}(N,P),$ then the locally finite manifold partition
 $\mathcal{B}=((j^{l}f)^{-1} \mathcal{A}^{l}(N,P))$ is a Whitney
 stratification which extends to a Thom stratification
 $(\mathcal{B},\mathcal{B}')$ of $f.$
 
 Let $z \in J^{l}(n,p)$ and let $f:(\mathbb R^{n},0) \to (\mathbb
 R^{p},0)$ such that $j^{l}f(0)=z.$

 Following Gibson et al. \cite{GibWirPleLoo}, we let
 
 $$
 \chi_{z}=\dim_{\mathbb R}\frac{\Theta_{f}}{tf(\Theta_{n})+(f^{*}(\mathcal{M}_{p})+\mathcal{M}_{n}^{l})\Theta_{f}}
$$

We define $W^{l}(n,p)=\{z\in J^{l}(n,p)\,|\, \chi_{z}\geq l\}.$
$W^{l}(n,p)$ is the \emph{bad set}, and the following hold

\begin{itemize}
\item [(a)] If $z \in J^{l}(n,p) \setminus W^{l}(n,p),$ then any $f\in
  \mathcal{E}^{p}_{n}$ such that $j^{l}f(0)=z$ is
  $l$-$\mathcal{K}$-determined.
\item [(b)] $W^{l}(n,p)$ is  $\mathcal{K}$-invariant.
\item [(c)] $W^{l}(n,p)$ is a real algebraic variety in $J^{l}(n,p).$
\end{itemize}

To verify (a) notice that,  if $\chi_{z}\leq l-1,$ then
  \begin{equation}
    \label{eq:16}
 tf(\Theta_{n})+(f^{*}(\mathcal{M}_{p})+\mathcal{M}_{n}^{l})\Theta_{f}
 \supset \mathcal{M}_{n}^{l-1}\Theta_{f}.  
  \end{equation}
Then we can multiply \eqref{eq:16} by $\mathcal{M}_{n}$ and the result
follows from Theorem \ref{gaf-wal}.

It follows from (a) that map-germs $f\in \mathcal{E}_{n}^{p}$ such
that $z=j^{l}f(0)$ satisfy $\chi_{z}\leq l-1$ are of finite singularity
type. In the following proposition we prove that the property of
FST holds in general.

\begin{proposition}[Gibson et all \cite{GibWirPleLoo}, Theorem 7.2]\label{prop:gwpl}
  The following conditions hold:
  \begin{itemize}
  \item[(i)] $\cod W^{l+1}(n,p) \geq \cod W^{l}(n,p).$
  \item[(ii)]\  $\lim_{l\to\infty}\cod W^{l}(n,p)= \infty.$
  \item[(iii)]\ \  There is a subbundle $W^{l}(N,P)\subset J^{l}(N,P)$
    naturally associated to $W^{l}(n,p).$ Moreover, when $N$ is
    compact, mappings $f:N\to P$ such that $j^{l}f(N)\cap W^{l}(N,P)=
    \emptyset$ are of finite singularity type.
  \end{itemize}
\end{proposition}

\begin{definition}\label{nova} \index{Properties holding in general}
  
We say that a property $\mathcal{P}$ of map-germs \emph{holds in
  general} if the sets 
$ W_{\mathcal{P}}^{l}(n,p)= \{z\in  J^{l}(n,p)|\; z \textrm{ does not
  satisfy } \mathcal{P}\}, $ satisfy (i) and (ii) (see \cite{Wal81}).
\end{definition}
While condition (i) in Proposition \ref{prop:gwpl} can be easily
verified, we can prove (ii) as  follows.

Given $z \in W^{l}(n,p), $ find $z' \in W^{l+q}(n,p),$ $\pi_{l}(z')
=z,$ where $\pi_{l}:W^{l+q}(n,p)\to W^{l}(n,p)$ is the projection,
such that $z' \notin W^{l+q}(n,p)$ (see Bruce, Ruas and Saia
\cite{BruRuaSai}, for a simpler proof of this result).   

As $W^{l}(n,p)$ is a real algebraic variety, it follows from
\L ojasiewicz's result \cite{Loj65} that it has a Whitney stratification
with semialgebraic strata. Condition (iii) is
immediate. Notice that conditions (i) and
(ii) imply that we can choose sufficiently high $l$ for which $\cod
W^{l}(n,p)>n.$ Then, the mappings $f:N \to P$ which are multitransverse
to $\mathcal A^{l}(N,P)$ satisfy the condition $j^{l}f(N)\cap
W^{l}(N,P)=\emptyset.$

{Our problem now is to construct a stratification $\mathcal{A}^{l}(n,p)$ of $J^{l}(n,p)
\setminus W^{l}(n,p)$ whose members are ${\cal K}$-invariant sets  $S_{j}=\{z \in J^{l}(n,p)
\setminus W^{l}(n,p)|\; {\mathbf{cod}\,}z =j\}$, for $j=0,1,2,\dots.$ The
definition of ${\mathbf{cod}\,}z$ will be given in the sequel.}

We shall see that ${\cal K}^{l}$-equivalent jets $z$ and $z'$ have the
same codimension, i.e., ${\mathbf{cod}\,}z= {\mathbf{cod}\,}z'.$ This number does
not coincide with the ${\cal K}^{l}$-codimension.

Although  we know that contact classes are smooth submanifolds of the
jet spaces, it is not clear at this point that the collection $S_{j}$
defines a stratification of $J^{l}(n,p) \setminus W^{l}(n,p).$ To
define ${\mathbf{cod}\,}z$ and to understand the structure of the strata $S_{j}$ in  ${\cal
  A}^{l}(n,p),$ we first discuss shortly  how to stratify
infinitesimally stable mappings and mappings of FST. Recall that
for any smooth map $f:N\to P,$ the \emph{critical set} of $f$ is $\Sigma(f)=\{x \in N|\,
df_{x}:T_{x}N \to T_{f(x)}\mathrm{\ is\  not\  surjective} \}$ and the
\emph{discriminant} of $f$ is $\Delta(f)=f(\Delta(f)).$
\index{Critical set}
\index{Discriminant}

We saw in section \ref{sec:finite-determ-math} that if $f:N\to P$ is infinitesimally
stable, the restriction $f|_{\Sigma(f)}:\Sigma(f) \to P$ is proper and
uniformly finite-to-one. In fact for any $y \in P$, $\#(f^{-1}(y)\cap
\Sigma(f)) \leq p.$ Moreover, if $f^{-1}(y)\cap \Sigma(f)=\{x_{1},
x_{2}, \dots, x_{s}\}$ the  multigerm $f:(N,S) \to (P,y)$ has a 
representative equivalent to a polynomial mapping $f:U \subset \mathbb R^{n} \to V
\subset \mathbb R^{p},$ where $U$ and $V$ are open sets in $\mathbb
R^{n}$ and  $\mathbb R^{p}$ respectively. In other  words $f$ is a
semialgebraic map defined on semialgebraic subsets.
Then we can apply the basic theorems of Whitney  and
Lojasiewicz  to construct Whitney stratifications
$\mathcal{S}$ of $N$ and $\mathcal{S}'$ of $P$ with the following properties

\begin{enumerate}\label{cond}
\item For each stratum $X$ of $\mathcal{S},$ there is a stratum $Y$ of
  $\mathcal{S}'$ such that $f(X)\subset Y.$
\item For each stratum $Y$ of $\mathcal{S}',$ it follows that
  $f^{-1}(Y) \setminus \Sigma(f)$ is a stratum of $\mathcal{S}.$
\item For each stratum $X$ of $\mathcal{S},$ such that $X\subset
  \Sigma(f),$   we have that $\dim X= \dim Y$ and $f:X \to Y$ is an
  immersion, where $Y$ is the stratum of $\mathcal{S}'$ which contains $f(X).$
\end{enumerate}

Notice that from 2. it follows that $N\setminus \Sigma(f)$ is a
union of strata. Hence, $\Sigma(f)$ is also a union of strata.

Now, if $f:(N,x_{0})\to (P,y_{0})$ is a stable germ, for any  small
representative that we also denote by $f,$ the stratum $X\in
\mathcal{S}$ which contains $x_{0}$ is connected  and its codimension
is strictly greater than the codimension of any other stratum of
$\mathcal{S}.$ This number depends only of $f.$ We call it \emph{the
  codimension } of $f,$ and we write ${\mathbf{cod}\,}f.$ A germ $f$ has
codimension zero if and only if it is of maximal rank.

This notion generalizes to map-germs of finite singularity  type.

\begin{definition}
  Let $f:(\mathbb R^{n}, 0) \to (\mathbb R^{p}, 0)$ be a map of
  finite singularity type. We define \emph{${\mathbf{cod}\,}f$ at $x=0$} as the
  codimension of a stable unfolding of $f.$
\end{definition}

Notice that this number is well defined. In fact, if $F:(\mathbb
R^{n}\times\mathbb R^{s}, 0) \to (\mathbb R^{p}\times\mathbb R^{s},
0)$ and $F':(\mathbb
R^{n}\times\mathbb R^{r}, 0) \to (\mathbb R^{p}\times\mathbb R^{r},
0)$ are stable unfoldings of $f$ and if, say, $r=s+k,$ then it follows
that $F\times Id$ is equivalent to $F',$ where $Id$ is the identity
map in $\mathbb R^{k}.$ Then ${\mathbf {cod}\,}(F\times Id) = {\mathbf{cod}\,}F',$ and it easy
to see that ${\mathbf{cod}\,}F ={\mathbf {cod}\,}(F\times Id).$ Now the following result
follows easily.

\begin{proposition}
  If $f \widesim{\mathcal{K}}f'$ then ${\mathbf{cod}\,}f= {\mathbf{cod}\,} f'.$
\end{proposition}

The properties of the stratification $\mathcal{A}^{l}(N,P)$ can be
summarized in the following results.

\begin{proposition}
  Let $f:(N,x_{0})\to (P,y_{0})$ be a smooth map-germ with an
  unfolding $F:(N',x'_{0})\to (P',y'_{0})$, as in the diagram
$$
  \xymatrix{ 
(N',x'_{0})  \ar[r]^{F} &(P',y'_{0})   \\
(N,x_{0}) \ar[u]^{i}\ar[r]_f   &   (P,y_{0}). \ar[u]^{j} }
$$
Then the following conditions are equivalent
\begin{itemize}
\item[(i)] $j^{l}f \notin W^{l}(N,P)$ and $j^{l}f$ is transverse to
  $\mathcal{A}^{l}(N,P).$  
\item[(ii)]\  $j^{l}F \notin W^{l}(N',P')$ and $j^{l}F$ is transverse to
  $\mathcal{A}^{l}(N',P'),$ and in addition if $X\in
  (j^{l}F)^{-1}\mathcal{A}^{l}(N',P')$ contains $x'_{0},$ then $i$ is
  transverse to $N'.$  
\end{itemize}
\end{proposition}

\begin{proposition}[Gibson et all, \cite{GibWirPleLoo}, Proposition
  3.3, Chapter 4]\label{prop:gib}
Let $f:N\to P$ be a proper smooth mapping multi-transverse to
$\mathcal{A}^{l}(N,P)$
and such that $j^{l}f(N)\cap W^{l}(N,P)=
  \emptyset .$ Let $\mathcal{S}= (j^{l}f)^{-1}\mathcal{A}^{l}(N,P)$
  and $\mathcal{S}'= \{f(X)\,|\,X\in \mathcal{S} \}\cup \{P\setminus
  f(N) \}.$ Then $(\mathcal{S},\mathcal{S}')$ is a Thom stratification
  of $f.$   
\end{proposition}

\begin{remark}
  The pair $(\mathcal S, \mathcal S')$ in Proposition \ref{prop:gib}
  has a minimality property which uniquely characterizes it among all
  possible pairs. We refer to Gibson et al., \cite{GibWirPleLoo} or
  Mather \cite{Mat76} for details.
\end{remark}

\subsection{Proof that topologically stable mappings are dense (
  Mather, \cite{Mat76}, \S 8) }
\label{sec:proof-that-topol}

Initially, we state the Thom-Mather topological stability theorem, whose
proof we outline in this section. Theorem B will follow from this
result and Thom's transversality theorem.
\index{Thom!topological stability theorem}
\begin{theorem}
  If $f:N\to P$ is proper and for some (and hence for all) $k\geq p+1,$ $j^kf$ is
multitransverse to the Thom-Mather stratification of $J^{k}(N,P),$ then $f$
is strongly $C^{\infty}$-stable.
\end{theorem}

Given $f:N\to P,$ we will show that we can approximate it by a
topologically stable mapping. First, we approximate $f$ by a mapping
$f_{1}:N \to P$ of finite singularity type (Proposition
\ref{prop:gwpl}). Then, we can choose an unfolding  $(F,N',P', i, j)$ of $f_{1}$
such that $F$ is proper and infinitesimally stable. Let
$\mathcal{S}'_{N'}$ and $\mathcal{S}'_{P'}$ be stratifications  of $N'$
and  $P',$ respectively satisfying conditions (1)-(3) in Section \ref{sec:how-strat-mapp}.
\index{Map!Thom stratified}

By Thom's transversality theorem, we can approximate $j$ by
$j_{2}:P\to P'$ such that $j_{2}$ is transverse to the strata of
$\mathcal{S}'_{P'}.$ Moreover we may suppose $j_{2}=j$ outside a compact
neighborhood of $f(N).$

Since $F$ is transverse to $j,$
it follows  that $F$ is transverse  to $j_{2}$ for $j_{2}$
sufficiently close to $j.$

The set $N_{2}=F^{-1}(j_{2}(P))$ is a smooth manifold. One can show
that there is a diffeomorphism $i_{2}:N\to N_{2}$ close 
to $i:N\to N'.$

We let $f_{2}:j_{2}^{-1}\circ F\circ i_{2}:N\to P.$ It follows from construction that
$f_{2}$ is close to $f$ in the $C^{\infty}$ topology. We claim that
$f_{2}$ is topologically stable.

The proof is based in the
following facts from the construction we have made:
\begin{enumerate}
\item[(i)]\, $(F,N',P',i_{2},j_{2})$ is an unfolding of $f_{2};$
\item[(ii)]\, $F$ is proper and infinitesimally stable;
\item[(iii)]\, $j_{2}$ is transverse to the stratification $\mathcal{S}'_{P'}$ of
$P'.$
\end{enumerate}

Let $g$ be a small perturbation of $f_{2},$ so that we can suppose
$f_{2}$ and $g$ are connected by a small arc $g_{t}$ in
$C^{\infty}(N,P), t\in [0,1],$ $g_{0}=f_{2}, g_{1}=g.$ We can lift $g_{t}$ to an arc $G_{t}$ in 
$C^{\infty}(N',P')$ such that $G_{0}=F$ and $(G_{t},N',P',i,j)$ is an
unfolding of $g_{t}.$ Moreover, we may suppose that $G_{t}=F$ outside
of a sufficiently small compact neighborhood of $i(N).$

From Theorem \ref{star}, it follows that there exist one parameter
families of diffeomorphisms $(H_{t},K_{t})\in \mathcal{A}, $
$H_{0}=Id_{N'},$ $K_{0}=Id_{P'},$ such that $F=K_{t}\circ G_{t}\circ
H_{t}^{-1},$ for all $t\in [0,1].$

Now consider the commutative diagram

$$
  \xymatrix{ 
N \ar[d]_{g_{t}} \ar[r]^{i} &N' \ar[d]_{G_{t}}\ar[r]^{H_{t}} &N'\ar[d]_{F}   \\
P \ar[r]^{j}   &   P'  \ar[r]^{K_{t}} &P'  }
$$

Since $(G_{t},N',P',i,j)$ is an unfolding of $g_{t},$ it follows that
$(F,N',P',H_{t}\circ i,K_{t} \circ j)$ is also an
unfolding of $g_{t}.$ Let $G(x,t)=(g_{t}(x),t),$ $\tilde
H(x,t)=H_{t}(x)$ and $\tilde K(y,t)=K_{t}(y).$ Then we have the
following commutative diagram  

$$
  \xymatrix{ 
{N\times I}\ar[d]_{G}  \ar[r]^{\tilde H} &N'\ar[d]^{F}   \\
{P\times I} \ar[r]^{\tilde K}   &   P'  
}
$$

So,we have that the triple $(F,  \mathcal{S}'_{N'},
\mathcal{S}'_{P'})$ is a Thom stratified map, and $i$ and $j$ are
transverse  respectively to $\mathcal{S}'_{N'}$ and 
$\mathcal{S}'_{P'}.$ Then taking $g$ sufficiently close to $f_{2}$,
$H_{t}\circ i$ and $K_{t}\circ j$ are also transverse to  $\mathcal{S}'_{N'}$ and 
$\mathcal{S}'_{P'},$  respectively.

It follows that these stratifications pull back to the Whitney's
stratifications $\tilde H^{*}(\mathcal{S}_{N'})$ and $\tilde
K^{*}(\mathcal{S}'_{P'})$ in $N\times I$ and $P\times I,$ respectively. 

Moreover, each $N\times\{t\}, P\times\{t\}$ is transverse to $\tilde H^{*}(\mathcal{S}_{N'})$ and $\tilde
K^{*}(\mathcal{S}'_{P'}),$ and conditions (1)-(3) are satisfied.

\index{Thom!second isotopy theorem}
Then, we may apply the Thom's second isotopy  lemma (Gibson
et al., \cite{GibWirPleLoo}, theorem 5.8, Chapter II) and
conclude that $f_{2}=g_{0}$ is topologically equivalent to $g=g_{1}.$

\subsection{The geometry of topological stability}
\label{sec:ult}

Whether $C^{0}$-stability and $C^{\infty}$-stability are equivalent
notions in the nice dimensions is  a question not answered by
the Thom-Mather theory. The first steps towards such result appear in
Robert May's thesis \cite{May73, May74}. Mays's results were followed
by a series of papers by Damon \cite{Dam79-1, Dam79-2, Dam79-3}, who
proved in \cite{Dam79-2} that $C^{\infty}$-stability is equivalent to a stronger notion
of $C^{0}$-stability.

Some of the ideas introduced in these papers form part of the basis
for Andrew du Plessis and Terry Wall's book on topological stability.
The book, \emph{The geometry of
  topological stability,} \cite{PleWal} published in 1995, is a deep
contribution to the subject of topological stability of smooth
mappings. They are  motivated by the problems left unanswered in the
Thom-Mather theory. One such problem is that it is very
difficult to determine explicitly the Thom-Mather stratification
$\mathcal{A}^k(n,p)$ in the complement of the nice dimensions and its
boundary. Another problem is that the transversality to the Thom-Mather
stratification is not a \emph{necessary} condition for topological
stability. In fact, this follows from a combination of results of
Looijenga \cite{Loo77} and Bruce\cite{Bru80} as we see in examples
\ref{E8} and \ref{mdu} below. du Plessis and Wall  give partial answers to the following two
conjectures:

\noindent \emph{Conjecture (i)} (Conjecture 1.3 in \cite{PleWal}) The smooth map $f:N\to P$ is $W$-strongly
$C^{0}$-stable if and only if it is quasi-proper and locally $C^{0}$-stable. 

Following \cite{PleWal}, we say that a map $f$ is \emph{quasi-proper}\index{Map!quasi-proper}
if there is a neighborhood $V$ of the discriminant $\Delta (f)$ in $P$
such that the restriction of $f$ to $f^{-1}(V),$ $f:f^{-1}(V) \to V$, is
a proper map.

\noindent \emph{Conjecture (ii)} If $N$ is compact, $f:N \to P$ is
$C^{0}$-stable if and only if it is locally $C^{0}$-stable.

\noindent \emph{Conjecture (iii)} (Conjecture 1.4 in \cite{PleWal}) There
exist a $\mathcal{K}$-invariant semi-algebraic stratification
$\mathcal{B}^{k}(n,p)$ of $J^{k}(n,p)\setminus W^{k}(n,p)$ such that a
smooth map $f:N\to P$ is locally $C^{0}$-stable if and only if, for
$k$ such that $\mathrm{cod}\, W^{k}(n,p) > n,$ $j^k f$ avoids
$W^{k}(n,p)$ and is multitransverse to $\mathcal{B}^{k}(n,p).$

We summarize  now the main results of  \cite{PleWal}.

\begin{theorem}[Theorem 1.5, \cite{PleWal}]\label{plewal}
  \begin{enumerate}
  \item [(i)] If $f:N\to P$ is $W$-strongly $C^{0}$-stable, then it is
    quasi-proper and locally $C^{0}$-stable. 
  \item [(ii)] If $f:N\to P$ is quasi-proper, of a finite singularity
    type over a neighborhood of its discriminant, and locally tamely
    $P$-$C^{0}$-stable, then it is $W$-strongly $C^{0}$-stable. 
  \end{enumerate}
\end{theorem}

The local $P$-$C^{0}$-stability is a very strong form of local
$C^{0}$-stability. We refer to \cite[p. 113]{PleWal}, for the
definition of \emph{tame $P$-$C^{0}$-stability}.

\begin{theorem}[Theorem 1.6, \cite{PleWal}]
  There exist $\mathcal{K}$-invariant algebraic subsets $Y^{k}(n,k)$
  in   $J^{k}(n,k)$ with  $W^{k}(n,k) \subseteq Y^{k}(n,k),$ and a
  $\mathcal{K}$-invariant stratification $\mathcal{B}^{k}(n,p)$ of
  $J^{k}(n,k)\setminus Y^{k}(n,k)$ with the following properties:
  \begin{enumerate}
  \item [(a)] If $f:N\to P$ is locally $C^{0}$-stable, or if $N$ is
    compact and $f$ is $C^{0}$-stable, then $j^{k}f$ is
    multitransverse to $\mathcal{B}^{k}(N,P);$ moreover, if
    $\mathrm{codim\;} Y^{k}(n,p)\geq n,$ then $j^{k}f$ avoids $Y^{k}(N,P).$ 
  \item [(b)] If $f:N\to P$ is such that $j^{k}f$ avoids $Y^{k}(N,P)$
    and is multitransverse to $\mathcal{B}^{k}(N,P),$ then $f$ is
    locally tamely $C^{0}$-stable.
\end{enumerate}
\end{theorem}

As remarked by the authors,  in the range of dimensions $n <
\mathrm{codim\;} Y^{k}(n,p),$ the results imply that Conjectures 1.3 and
1.4, with $W^{k}$ replaced by $Y^{k},$ hold. 

We finish this section with two examples illustrating two rather
delicate questions in the theory of $C^{0}$-stability.

\index{Simple elliptic singularity}
\begin{example}[The simple elliptic singularity $\tilde{E}_8$]\label{E8}
  The simple elliptic singularities $\tilde{E}_8$ in $\mathbb{K}^{3}$,
  $\mathbb{K}=\mathbb{R}$ or $\mathbb{C},$ is the
  $\mathcal{K}$-unimodular family of hypersurfaces with isolated
  singularities defined by
  $$
  \tilde E_{8}:\ \ \ f_{\lambda}(x_{0},x_{1},x_{2})=x_{0}^{2}+x_{1}^{3}+x_{2}^{6}+\lambda
  x_{0}x_{1}x_{2}.
  $$

The family $f_{\lambda}$ is weighted homogeneous of type $(3,2,1;6),$
then the Milnor number $\mu(f_{\lambda})$ is constant and equal to
$10.$ When $\mathbb K=\mathbb C,$ it was shown by Looijenga
\cite{Loo77} that  the stable unfolding of
$f_{\lambda}$ is topologically trivial along the moduli parameter
$\lambda.$

From section \ref{sec:class-stable-sing}, \eqref{eq:17}, it follows
that the stable unfolding of $f_{\lambda}$ can be given as

\begin{align*}
  F:(\mathbb C^{3} \times \mathbb C^{8}\times \mathbb C,0) &\to
(\mathbb C \times\mathbb C^{8}\times\mathbb C,0)\\
(x,u,\lambda) &\mapsto (\tilde f(x,u,\lambda), u,\lambda)
\end{align*}
with $x=(x_{0},x_{1},x_{2})$, $u=(u_{1}, \dots, u_{8}),$
$\tilde f_{\lambda}(x,u)=\tilde f(x,u,\lambda),$ $\tilde
f_{\lambda}(x,0)=f_{\lambda}(x),$ and
\begin{multline*}
  \tilde f(x,u,\lambda)=x_{0}^{2}+x_{1}^{3}+x_{2}^{6}+\lambda
  x_{0}x_{1}x_{2}+u_{1}x_{1}+u_{2}x_{2}+
  u_{3}x_{1}x_{2}\\ +u_{4}x_{2}^{2} 
  +u_{5}x_{1}x_{2}^{2}+u_{6}x_{2}^{3}+u_{7}x_{1}x_{2}^{3}+u_{8}x_{2}^{4}.
\end{multline*}

For all $\lambda$ sufficiently small, including
$\lambda=0,$ $F_{\lambda}:(\mathbb C^{11},0) \to (\mathbb C^{9},0)$ is
topologically stable. See Looijenga \cite{Loo77} and Bruce \cite{Bru80}.

On the other hand, the construction of the  Thom-Mather
stratification $\mathcal{A}^{k}(n,p)$ in $J^{k}(n,p)\setminus
W^{k}(n,p)$ as discussed in sections \ref{sec:proof-that-topol} and
\ref{sec:ult} reduces to the problem of finding a minimal
Whitney stratification of jets of finite singularity
type. However, Bruce proved that at $\lambda=0$ the Whitney condition
$(b)$ fails (see \cite{Bru80}, Proposition 2 and Example 3(a)). The
failure of condition (b) can be geometrically detected as follows: the
number of cusps (${\cal A}_2$-singularities) of the intersection of
the discriminant $\Delta(F)$ with a family of $2$-planes transversal
to $\Delta(F)$ jumps from $12$ to $13$ at $\lambda=0.$ This number is
an invariant of the Thom-Mather stratification (\cite{Bru80},Proposition
2).

If follows that the germ $F_{0}:(\mathbb C^{11},0) \to (\mathbb
C^{9},0)$ is topologically stable, but $j^k F_{0}$ is not transverse to
the Thom-Mather stratification.

\end{example}

\begin{example}[May \cite{May73} and  du Plessis and
Wall \cite{PleWal}, Section 4.1 ]\label{mdu}

Lef $f:\mathbb R \to \mathbb R$ be the proper map whose graph is
illustrated in Figure \ref{fig:co}. Its singular set $\Sigma(f)$ is
$\mathbb Z \subset \mathbb R,$ and the critical values are $F(0)=0,$
$f(n)=n+1,$ for $n>0$ and odd, and $f(n)=n-1$ for $n>0$ and even;
while $f(-x)=-f(x).$ For example, we may define, as in du Plessis and
Wall \cite{PleWal},
$$
f(x)=
\begin{cases}
   x^{3}  & x \in [\frac{-1}{4}, \frac{1}{4}], \\
   n+1 -(x-n)^{2}  &x\in[n-\frac{1}{4},n+ \frac{1}{4}], n\in \mathbb
  N, n\ \mathrm{odd} \\
   n-1 +(x-n)^{2}  &x\in[n-\frac{1}{4},n+ \frac{1}{4}], n\in \mathbb
  N, n\ \mathrm{even} 
\end{cases}
$$
with $f$ defined on the remaining intervals so that it is monotone
(with $f'\neq 0$) on each interval and $C^{\infty}$ everywhere.

\begin{figure}[h]
  \centering
\scalebox{0.60}{\includegraphics{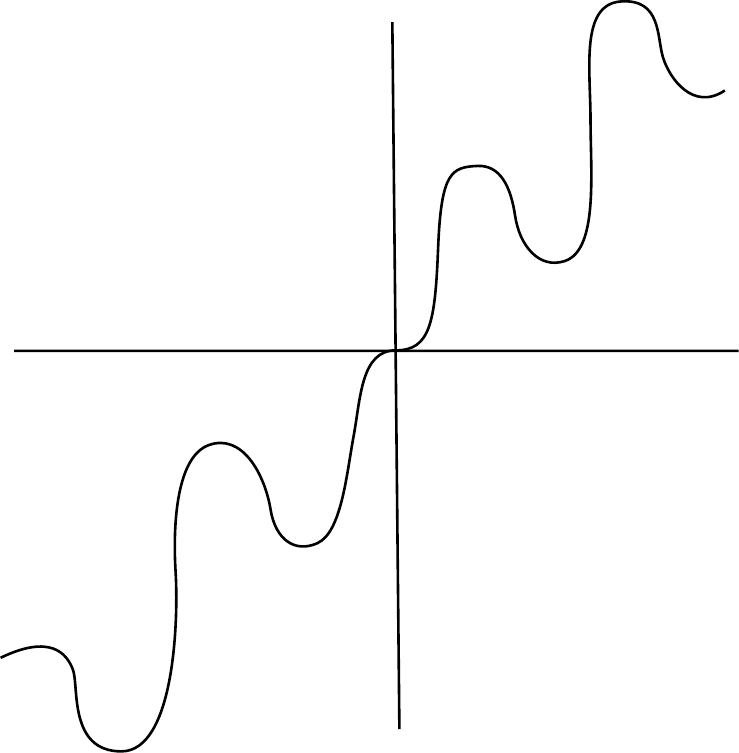}}
  \caption{$C^{0}$-stable non transversal map}
  \label{fig:co}
\end{figure}

One can see that $f$ is $C^{0}$-stable. However it is not transverse to
the Boardman manifold $\Sigma^{1}$ at the origin. In fact, $f$ cannot
be transverse to any invariant stratification of jet space. Thus
$C^{0}$-stability  of proper maps $f:\mathbb R \to \mathbb R$ cannot be
characterized by multitransversality to any stratification.

Notice that $f$ is not locally $C^{0}$-stable, then it follows from
Theorem \ref{plewal}(i) that $f$ is not strongly stable.

\end{example}

\subsection{Notes}
\label{sec:notes-3}
\index{Whitney!fibering conjecture}
In the recent paper \emph{On the smooth Whitney fibering
  conjecture} \cite{MurTroPle} Murolo, du Plessis and Trotman give a
remarkable improvement of the first Thom-Mather isotopy theorem for Whitney
stratified sets. The result follows from their proof, in the same
paper, of the smooth version of the Whitney fibering conjecture for
Bekka (c)-regular stratifications. The original conjecture made by
Whitney in \cite{Whi65-2} in the real and complex, local analytic
and global algebraic cases, was proved by Parusinski and
Paunescu \cite{ParPau} in 2014.

As an application of the results, in section 9 of
the paper,  the authors give a sufficient condition for a smooth
map between two smooth manifolds to be strongly topologically stable
(\cite[Theorem 13]{MurTroPle}).

This result in turn, implies the long-awaited improvements of Mather's
topological stability theorem, which we state below. 

\begin{corollary}[Corollary 11, \cite{MurTroPle}]\label{mtp1}
Lef $f:N \to P$ be a quasi-proper smooth map of finite singularity
type whose $l$-jet   avoids $W^{l}(N,P)$ and is multi-transverse to
$\mathcal{A}^{l}(N,P).$ Then $f$ is strongly topologically stable.
\end{corollary}

Corolllary \ref{mtp1}  has the following immediate consequence.
\begin{corollary}[Corollary 12, \cite{MurTroPle}]\label{con}
  The space of strong topologically stable maps is dense in the space
  of quasi-proper maps between two smooth manifolds.
\end{corollary} 

\section{The boundary of the nice dimensions }
\label{sec:bound-nice-dimens}
\index{BND|see{ nice dimensions, boundary of}}
In this section we give a systematic presentation of  the Thom-Mather
singularities in the boundary of the nice dimensions (BND). Much of the
material presented here is well known to experts. However, it seems that
the organized presentation of the construction of the Thom-Mather
stratification of $J^{k}(n,p)$ when $(n,p)$ is a pair in BND  combined
with the discussion of the properties of topologically stable mappings
in these dimensions do not appear in the literature. The results come
from Mather \cite{Mat69-2,Mat71}, Damon \cite{Dam80,Dam82}, du Plessis and
Wall \cite{PleWal} and  Ruas \cite{Rua} and recent results by Ruas and
Trivedi \cite{RuaTri}.

We only give an outline of most of the proofs but we present the full
details in the case $n=p=9.$ 

We also review du Plessis and Wall main result in \cite{PleWal89} that
$C^{1}$-stable mappings are dense if and only if $(n,p)$ is in the
nice dimensions.

\subsection{A candidate for the Thom-Mather stratification in BND}
\label{sec:thom-math-strat}

The main reference for this section is Ruas and Trivedi \cite{RuaTri}.
We saw that a pair $(n,p)$ is in \emph{the boundary
of nice dimensions} if $\sigma(n,p)=n,$ where $\sigma(n,p)=\cod
\pi^{k}(n,p),$ $k\geq p+1,$ and $\pi^{k}(n,p)$ is the smallest Zariski
closed $\mathcal{K}^{k}$-invariant set in $J^{k}(n,p)$ such that its
complement in $J^{k}(n,p)$ is the union of finitely many
$\mathcal{K}^{k}$-orbits.

In the nice dimensions $\sigma(n,p)>n,$ so it follows that the strata of the stratification of
$J^{k}(n,p)\setminus \pi^{k}(n,p)$ are the simple
$\mathcal{K}^{k}$-orbits of $\mathcal{K}$-codimension $\leq n.$
However, at the BND, there are strata of codimension $n$ in
$\pi^{k}(n,p);$ these strata cannot be avoided by transversal
maps.We shall see that for all pairs $(n,p)$ in BND with the exception
of the pair $(10,7)$ these strata are unimodular
strata\index{Unimodular strata|(} consisting of
the union of a one-parameter family of ${\cal K}$-orbits. When
$(n,p)=(10,7),$ surprisingly, the Thom-Mather stratification also has a
bimodal strata which is the union of a two parameter  family of ${\cal
  K}$-orbits. We call the pair $(10,7)$ \emph{the exceptional pair in BND.}

We recall here the notion of modality (or modularity). This notion can be defined
for any geometric subgroup of $\mathcal{K},$ but here we refer to
modularity for group $\mathcal{K}.$

Let $z \in J^{k}(n,p)$ and denote by $K^{*}(z)$ the union of all
$\mathcal{K}^{k}$-orbits of codimension equal to the codimension of
$\mathcal{K}^{k}(z)$ in $J^{k}(n,p).$ Suppose $K_{*}(z)$ is the connected component of
$K^{*}(z)$ in which $z$ lies. Then we say that $z \in J^{k}(n,p)$ is
$r$-modular if
$$
\cod K_{*}(z)=\cod K^{k}\cdot z-r  .
$$
We say that $1$-modular jets are \emph{unimodular}, $2$-modular jets
are \emph{bimodular} and so on.
Also, if the union of unimodular jets is a submanifold of
$J^{k}(n,p),$ as it happens in our case, we call this union a
\emph{unimodular stratum.} 

The bad set $\tilde \Pi^{k}(n,p)$ in this case is a proper Zariski
closed subset of $\Pi^{k}(n,p)$ such that $\cod \tilde \Pi^{k}(n,p)
\geq n+1$ and $\Pi^{k}(n,p) \setminus \tilde  \Pi^{k}(n,p)$ is the
union of the connected components of a unique unimodular family, while
for the pair $(10,7)$ this set is the union of the unimodular and
the bimodular families.

We stratify $J^{k}(n,p) \setminus \tilde  \Pi^{k}(n,p)$ by taking
as strata the $\mathcal{K}$-orbits of the stable maps and the
modular strata. We call this stratification $\Sigma_{bnd}^{k}(n,p)$
(see \cite{RuaTri}).

In
the global setting we have the following situation. Let $N,P$ and
$J^{k}(N,P)$ as before. Denote by $\tilde \Pi (N,P)$ the subbundle of
$J^{k}(N,P)$ with fibers $\tilde \Pi(n,p).$ Then the
codimension of $J^{k}(N,P) \setminus \tilde \Pi(N,P)$ is equal to the
codimension of $\tilde \Pi(n,p)$ in $J^k(n.p).$ Moreover, the
stratification $\Sigma^{k}_{bnd}(n,p)$ induces a stratification on  $J^{k}(N,P)
\setminus \tilde \Pi(N,P),$ denoted by $\Sigma^{k}_{bnd}(N,P).$

 The
following result appears in  \cite{RuaTri}.

\begin{theorem}[Ruas and Trivedi, \cite{RuaTri}, Theorem 3.1]
  The set of maps $f:N \to P$ such that $j^{k}f(N)\cap \tilde \Pi
  (N,P) =\emptyset$ and $j^{k}f$ is transverse to the strata of
  $\Sigma_{bnd}^{k}(N,P)$ is open in $C^{\infty}(N,P)$ with the Whitney topology.
\end{theorem}

The $(a)$ regularity of $\Sigma^{k}_{bnd}(N,P)$ follows from the above
result and the Main Theorem in Trotman \cite{Tro}.

\begin{corollary}
  The stratification $\Sigma^{k}_{bnd}(n,p)$ is  $(a)$-regular.
\end{corollary}

We prove in Theorem \ref{th:8.4} that maps transverse to
$\Sigma^{k}_{bnd}(N,P)$ are Thom-Mather  maps  for any pair $(n,p)$ in BND.

\subsection{The unimodular strata in BND}
\label{sec:unim-strata-b.n.d}

The results in this section are local and hold for map-germs
$f:(\mathbb K^{n},0) \to (\mathbb K^{p},0)$ for $\mathbb K=\mathbb R$
or $\mathbb C,$ $f\in \mathcal{E}^p_n$ or $f\in \mathcal{O}^{p}_{n}.$
From  Mather's calculations  in \cite{Mat71}, it follows that the
following pairs lie in the boundary of the nice dimensions:

\noindent $\mathbf{(i)\ n\leq p:}$
\begin{enumerate}
\item [(1)] The case $\sigma(n,p)= 6(p-n) +9 $ for $3\geq p-n\geq 0$
  and $n\geq 4$ or $n=3,$ gives $(n,p)\in \{(9,9), (15,16), (21,23),(27,30)\}.$
\item [(2)] The case $\sigma(n,p)= 6(p-n) +8 $ for $p-n\geq 4$ and
  $n\geq 4,$ gives $(n,p) \in \{(6t+2, 7t+1); t\geq 5\}.$
\end{enumerate}

\noindent $\mathbf{(ii)\  n> p:}$ 
\begin{enumerate}
\item [(1)] The case $\sigma(n,p)= 9 $ for $n=p+1,$ gives $(n,p)=(9,8).$
\item [(2)] The case $\sigma(n,p)= 8 $ for $n=p+2,$ gives $(n,p)=(8,6).$
\item [(3)] The case $\sigma(n,p)= n-p +7 $ for $n\geq p +3$ gives
  $(n,p) \in \{(10+k,7):\, k\geq 0\}.$
\end{enumerate}

The strategy to find the strata of $\Sigma_{bnd}(n,p)$ has the following
steps:

\noindent (1) inspecting the classification of the local algebras
$Q(z), z\in J^{k}(n,p),$ such that  $\mathcal{K}$-$\cod(z) \leq n.$ By
Mather's results these algebras are simple and for each such algebra $Q(z)$
there exists a stable germ $f:(\mathbb K^n ,0)\to (\mathbb K^p ,0),$
such that $Q(f)\simeq Q(z);$

\noindent (2) listing the unimodular algebras of $\mathcal{K}$-codimension
$n+1,$ whose union makes the unimodular strata of the stratification;

\noindent (3) Excluding the existence of bimodular strata of
codimension $n$ for pairs $(n,p)$ in BND except $(10,7).$ For
$(n,p)=(10,7)$ we include the classification of the bimodular strata.

A detailed discussion of simple and unimodular algebras appears in
Chapter 8 of the book of du Plessis and Wall \cite{PleWal}. For the
convenience of the reader we give the precise references of the classifications.
First a word about the notation. we use mainly  Thom's notation, and the relevant here are the first and second
order the Thom-Boardman symbols $ \Sigma^{r}$ and $\Sigma^{r,s},$
respectively, $r=1,2,3,4.$ Mather's adaptation $\Sigma^{r(s)}$ also
appears, as it is useful for $2$-jet classification. A germ $f$ in
$\Sigma^{r}$ may be regarded as an unfolding of a germ $f_{0}$ with rank zero
and source dimension $r.$  {When we look at the second degree terms, the
notation $s$ in $\Sigma^{r(s)}$ indicates 
how many independent components the $2$-jet of $f_0$ has.}




 We first describe the unimodular
strata in the boundary of the nice dimensions, based on the
presentation in Ruas and Trivedi \cite{RuaTri},

\subsubsection{Case 1: $n\leq p$}
\label{sec:case1:-nleq-p}

\noindent $ \mathbf{(1)\ (n,p)=(9,9)} $

The first unimodular family in this case is the one parameter family
of type $\Sigma^{3,0}$ ( $\Sigma^{3(3)}$ in Mather's notation) introduced in
section \ref{sec:thoms-example}:
\begin{align}
  \label{eq:18}
  f_{\lambda}:(\mathbb K^{3},0)&\to (\mathbb K^{3},0)\\\nonumber
  (x,y,z) &\mapsto (x^{2}+\lambda y z, y^{2}+\lambda xz, z^{2}+\lambda xy)
\end{align}
with $\lambda\neq 0,-2,1.$

Calculating the $\mathcal{K}$-tangent space of $f_{\lambda}$ we find
that $\mathcal{K}$-$\cod (f_{\lambda})=10,$ for $\lambda\neq 0,-2,1.$
The sets $(-\infty, 0), (0,-2), (-2,1), (1,\infty)$ parametrize
orbits in the connected components of the unimodular strata of
codimension $9.$

\noindent $\mathbf{(2)\ (n,p)=(15,16)}$

The unimodular stratum in these dimensions is related to the moduli
stratum in dimensions $(9,9)$  in the 
following way. From a result of Serre and Berger (see Eisenbud
(\cite{Eis}, Proposition 2) it follows that for  
analytic map-germs $f:(\mathbb{K}^{n},0) \to (\mathbb{K}^{n},0)$ the class of
the Jacobian $J(f)$ is a non-zero element in the local algebra $Q(f).$
Moreover, the ideal generated by $J(f)$ in this algebra is the unique
minimal non-zero  ideal in $Q(f).$ It also follows that the
residue class of $J^{2}(f)$ in $Q(f)$ is zero.

The unimodular family here is
\begin{equation}
  \label{eq:29}
  f_{1\lambda}:(\mathbb
K^{3},0)\to(\mathbb
K^{4},0),\,f_{1\lambda}(x,y,z)=(f_{\lambda}(x,y,z),J(f_{\lambda})(x,y,z)),
\end{equation}
where $f_{\lambda}$ is the map given in
\eqref{eq:18} and $J(f_{\lambda})(x,y,z)=xyz.$  The following holds
$$
\mathcal{K}\textrm{-}\cod(f_{1\lambda})=\mathcal{K}\textrm{-}\cod(f_{\lambda})+(\delta(f_{\lambda}
)-2) =16
$$
where $\delta(f_{\lambda})=\mathrm{dim}_{\mathbb K}Q(f_{\lambda})=8.$ The unimodular stratum in $J^{k}(15,16),$
$k\geq 3$ is the union of all corank $3$ $k$-jets $z \in
J^{k}(15,16),$ $\mathcal{K}$-equivalent to a suspension of $f_{1\lambda}.$

\noindent $\mathbf{(3)\  (n,p)=(21,23)}$

In this case the unimodular family is
\begin{equation}
  \label{eq:30}
  f_{2\lambda}:(\mathbb K^{3},0)\to (\mathbb K^{5},0),\
  f_{2\lambda}(x,y,z)=(f_{1\lambda}(x,y,z),0).
\end{equation}

\noindent $\mathbf{(4)\  (n,p)=(27,30)}$

The unimodular family here is 
\begin{equation}
  \label{eq:31}
    f_{3\lambda}:(\mathbb K^{3},0)\to
    (\mathbb K^{6},0),\  f_{3\lambda}(x,y,z)= (f_{2\lambda}(x,y,z),0).
  \end{equation}
  
\begin{remark}
The following formula holds (du Plessis and Wall \cite{PleWal},
Chapter 8)
  $$
\mathcal{K}\textrm{-}\cod(f_{i\lambda})=\mathcal{K}\textrm{-}\cod(f_{\lambda})+(p-n)(\mathrm{dim}_{\mathbb
    R}Q(f_{\lambda})-2),
  $$
for $i=1,2,3,p=n+i.$ 
\end{remark}

\noindent $\mathbf{(5)\ (n,p)=(6t+2, 7t+1)\  \mathbf{for}\ t\geq 5}$

When $t=5$ the unimodular stratum is defined by
$$
  f_{\lambda}:(\mathbb K^{4},0)\to(\mathbb K^{8},0),  f_{\lambda}(x,y,z,w)=(u_{1},u_{2}, \dots, u_{8})
$$
where
\begin{align*}
  u_{1}&= x^{2}+y^{2}+z^{2}   &u_{2}&=y^{2}+\lambda z^{2}+w^{2} &u_{3}&=xy
  &u_{4}&=xz \\
  u_{5}&=xw  &u_{6}&=yz  &u_{7}&=yw  &u_{8}&=zw
\end{align*}

\subsubsection{Case 2: $n>p$}
\label{sec:case-2:-np}

\noindent $\mathbf{(6)\  (n,p)=(8,6)}$

The smallest pair $(n,p)$ with $n>p$ in the boundary of the nice dimensions
is $(8,6).$ The unimodular stratum is given by the following
one-parameter family of maps
\begin{align*}
  &f_{\lambda}:(\mathbb K^{4},0)\to(\mathbb K^{2},0),\\
   &f_{\lambda}(x,y,z,w)=(x^{2}+y^{2}+z^{2},
     y^{2}+\lambda z^{2}+w^{2}), \ \lambda \neq 0,1.
\end{align*}

\noindent $\mathbf{(7)\  (n,p)=(9,8)}$

The unimodular family here is
\begin{align*}
  &f_{\lambda}:(\mathbb K^{2},0)\to(\mathbb K,0),\\
&f_{\lambda}(x,y)= x^{4}+y^{4}+\lambda x^{2}y^{2},
                                                      \lambda\neq \pm 2.
\end{align*}

\noindent $\mathbf{(8)\  (n,p)=(10+k, 7)\  for\ k\geq 0}$

In this case, the unimodular family is
\begin{align*}
  &f_{\lambda}:(\mathbb K^{4+k},0)\to (\mathbb K,0),\\
  &f_{\lambda}(x,y,z,w_{0}, \dots, w_{k})=
    x^{3}+y^{3}+z^{3} +\lambda xyz +\sum_{i=0}^{k}\delta_{i}w_{i}^{2},
\end{align*}
for $\delta_{í}=\pm 1,\, i=0,\dots, k,\, \lambda^{3}\neq -1.$

The pair $(n,p)=(10,7)$ is the exceptional pair in BND. It follows from
Wall \cite{Wal85} that the following two parameter moduli family of
$\Sigma^{5}$ singularities has codimension $n=10,$ providing for this
pair of dimensions a new relevant strata.
\begin{equation}
  \label{eq:24}
  \begin{aligned}
    &f_{\lambda}:(\mathbb K^{5},0)\to(\mathbb K^{2},0),\\
&f_{\lambda}(x)= (\sum_{i=1}^{5}a_{i}x_{i}^{2},
\sum_{i=1}^{5}b_{i}x_{i}^{2}),\ \ 
                                                           a_{i}b_{j}-a_{j}b_{i}\neq
                                                           0, \, i\neq j.
                                                         \end{aligned}
                                                       \end{equation}

\begin{theorem}\label{th:8.4}
  For each pair $(n,p)$ in the boundary of the nice dimensions the
  following hold:
  \begin{enumerate}
  \item [(a)]\ If $(n,p)\neq (10,7)$ the strata of 
    $\Sigma^{k}_{bnd}(n,p)$ are are the $\mathcal{K}^{k}$-orbits of the stable germs of
  $\mathcal{K}$-codimension $\leq n$ and the  unimodular strata of
  codimension $n$ defined  by the connected 
  components of the unimodular families described in
  \ref{sec:case1:-nleq-p} and \ref{sec:case-2:-np}.   If
  $(n,p)=(10,7),$ besides the unimodular  strata defined in
  \ref{sec:case-2:-np}\textbf{(8)}, there is an exceptional bimodular
  strata as  defined in \eqref{eq:24}.

  \item [(b)]\, Maps $f:N\to P$ such that $j^{k}f$ is transverse to
    the strata of $\Sigma^{k}_{bnd}(n,p)$ are Thom-Mather maps for
    any pair $(n,p)$ in BND.
  \end{enumerate}
  \end{theorem}

\begin{proof}
  The proof consists on a careful inspection of the tables of simple
  and unimodular singularities in order to list the relevant strata
  and to verify that the codimension of the set
  $\tilde \Pi^{k}(n,p), \, k\geq p+1$ is greater than or
  equal to $n+1.$ We give an outline of the proof.

  \noindent $\mathbf{I.\ n\leq p}$

  For $(n,p)\in \{(9,9),(15,16),(21,23),(27,30)\}$ the relevant
  Boardman types are $\Sigma^{1},\Sigma^{2,0}, \Sigma^{2,1}$ and
  $\Sigma^{3}.$  We first analyze the pair $(9,9).$

  \noindent $\mathbf{Case\,(1)\  (n,p)= (9,9)}$

All singularities of type $\Sigma^{1}$  and $\Sigma^{2,0}$ are simple. A
complete list of strata of type $\Sigma^{2,1}$ has been given by Dimca
and Gibson \cite{DimGib}.  See also Table 8.4 in du
Plessis and Wall \cite{PleWal}.

The first unimodular family of type $\Sigma^{2,1}$ is
\begin{equation} \label{eq:20}
I_{2,3}:
(x^{2}-\eta y^{4}, xy^{3}+c y^{5}), \, c^{2} \neq 0, \eta.
\end{equation}
It follows
that the $\mathcal{K}$-codimension of each orbit is $12,$ the unimodular stratum has
codimension $11,$ so that this family does not appear generically when
$n=p=9.$ As a consequence, the relevant $\Sigma^{2,1}$ strata in this
case are simple $\mathcal{K}$-orbits. Notice that
$\mathrm{cod}\,\Sigma^{2,2}(9,9) \geq 10$ and then the $\Sigma^{2,2}$
singularities 
do not appear generically in $J^{k}(9,9).$

The next Boardman symbol is $\Sigma^{3},$ and as we saw in
\ref{sec:case1:-nleq-p}, the relevant strata are the connected
components of the unimodular family $(1).$

We list all the strata in Table \ref{tab:4}.

The set $\tilde \Pi^{k}(9,9)$ is the finite union of the following
Zariski closed sets of codimension $\geq 10$ in $J^{k}(9,9), k\geq 10:$

$$
\tilde \Pi^{k}(9,9)=\tilde \Pi_{1}^{k} \cup \tilde \Pi_{2}^{k}\cup
\tilde \Pi_{j\geq 3}^{k}
$$
where
\begin{align*}
\tilde \Pi_1^{k}&= \{\sigma \in J^{k}(9,9), \sigma \in
  \Sigma^{1},\,\mathcal{K}^{k}\textrm{-cod}(\sigma) \geq 10 \}\\
\tilde \Pi_2^{k}&=\{\sigma \in J^{k}(9,9), \sigma \in
  \Sigma^{2},\,\mathcal{K}^{k}\textrm{-cod}(\sigma) \geq 10 \}\\
\tilde \Pi^k_{j\geq 3}&=\{\sigma \in J^{k}(9,9),\, \sigma \in \Sigma^{j}, j\geq 3,
\mathcal{K}^{k}\textrm{-cod}(\sigma) \geq 11 \}
\end{align*}

\end{proof}

\begin{table}[h]
  \centering
  \begin{tabular}[center]{|c|c|c|c|c|}
    \hline
    Type & Name & Normal form & Conditions &
                                             $\mathcal{K}$-$\textrm{cod}
                                             \leq 9$\\
    \hline \hline
$\Sigma^{1}$ &$A_{j}$         &$(x^{j+1})$ & $1\leq j\leq 9$ & $j\leq 9$ \\
$\Sigma^{2,0}$ &$B^{\pm}_{p,q}$ &$(xy, x^{p}\pm y^{q})$ &$2\leq p\leq q \leq 8$ &
                                               $4\leq p+q \leq 9$ \\
$\Sigma^{2,0}$ & $B^{*}_{p,p}$ & $(x^{2}+y^{2}, x^{p})$  & $p=3,4$
                                           & $5\leq 2p\leq 9$ \\
$\Sigma^{2,1}$ &$C_{2k-1}$     &$(x^{2}+y^{3},y^{k+2})$  &$k=1,2$ &
                                                             $2k+5\leq 9$\\
$\Sigma^{2,1}$ & $C_{2k}$      &$(x^{2}+y^{3},  xy^{k+1})$ &$k=1$
                                           &$2k+6\leq 9$ \\
$\Sigma^{3,0}$ &    &$(x^{2}+\lambda yz,y^{2}+\lambda xz,z^{2}+\lambda
                      xy)$ & $\lambda\neq -2, 0, 1$ & 10\\
\hline    \hline
  \end{tabular}
  \caption{$(n,p)=(9,9)$}
  \label{tab:4}
\end{table}

\noindent $\mathbf{Cases\  (2)\, (15,16);\  (3)\, (21,23);\   (4)\, (27,30)}$

  The singularities of type $\Sigma^{1}$ and $\Sigma^{2,0}$ are
simple. The classification of the singularities of type $\Sigma^{2,1}$
and their invariants in these cases can be found in Tables 8.7, 8.8 and
8.9 of \cite{PleWal}. The first unimodular family of type
$\Sigma^{2,1},$ when $n<p,$ is ${\overline
  D}_{3,5}$( also denoted by ${\overline J}_{2,3,5,5}$ in \cite{PleWal}).

The normal forms are
\begin{align*}
  f_{1\lambda}(x,y)&=(x^{2}\pm y^{4}, xy^{3}+cy^{5}, y^{6})  \\
  f_{2\lambda}(x,y)&=(x^{2}\pm y^{4}, xy^{3}+cy^{5}, y^{6},0)  \\
  f_{3\lambda}(x,y)&=(x^{2}\pm y^{4}, xy^{3}+cy^{5}, y^{6}, 0, 0) 
\end{align*}

From \eqref{eq:20}, we get
$$
\mathcal{K}\textrm{-cod}(f_{i\lambda})=\mathcal{K}\textrm{-cod}(f_{\lambda})+
+ i(\textrm{dim}_{\mathbb R}Q(f_{\lambda})-2),
$$
for $i=1,2,3$ where
\begin{equation}
  \label{eq:21}
f_{\lambda}(x,y)=(x^{2}\pm y^{4}, xy^{3}+cy^{5}).
 \end{equation}

 Then $\mathcal{K}\textrm{-cod}(f_{i\lambda})= 12 +i(10-2), i=1,2,3$
 and these singularities do not appear generically in BND. As in Case
 (1), for $n=9+6i,\, i=1,2,3$ with the help of Tables 8.7, 8.9, 8.9  and 8.11 in
 \cite{PleWal} we can verify that the strata of type
 $\Sigma^{1}, \Sigma^{2,0}, \Sigma^{2,1}$ and $\Sigma^{2,2}$ are
 $\mathcal{K}$-orbits of $\mathcal{K}$-codimension $\leq 9+6i,\,
 i=1,2,3$ and the unimodular strata defined in \eqref{eq:29},
 \eqref{eq:30} and \eqref{eq:31}. Moreover, $\mathrm{cod}\,\tilde
 \Pi^{k}(n,p)\geq n+1.$ 

 \noindent $\mathbf{Cases\  (5)\, (6t+2,7t+1), t\geq 5}$

 The relevant Boardman types here are
 $\Sigma^{1},\Sigma^{2,0},\Sigma^{2,1}, \Sigma^{2,2},\Sigma^{3}$ and
 $\Sigma^{4}.$ As before $\Sigma^{1},\Sigma^{2,0}$ are simple, and the
 moduli strata of type $\Sigma^{2,1}$ has normal form
 $f_{\lambda}:(\mathbb K^{2},0) \to (\mathbb K^{t+1},0), t\geq 5,$
 $$f_{3\lambda}(x,y)= (x^{2}\pm y^{4},
 xy^{3}+cy^{5},y^{6},\underbrace{0,\dots, 0}_{\text{t-1}}),$$
 where $f_{\lambda}(x,y)=(x^{2}\pm y^{4}, xy^{3}+cy^{5}).$ Since
 ${\cal K}$-cod$(f_{\lambda})= 12,$ then  ${\cal
   K}$-cod$(f_{3\lambda})\geq 12 + (t-1)(10 -2)=4 +8t> 6t +2,$ and it
 follows that this family is not generic when $(n,p)=(6t+2,7t+1), t\geq 5.$

The $\Sigma^{2,2}$ germs of order $3$ appear in du Plessis and Wall
\cite{PleWal}, Section 8.5, Tables 8.10 and 8.11.  The type
$\Sigma^{2,2}$ is subdivided ( see \cite{PleWal}) into types
$\Sigma^{2,2(j)},$ where $j$ is the rank of the kernel of the third
intrinsic derivative. It follows that $\mathrm{codim}\,
\Sigma^{2,2(j)}= 6e +10 +j(e+j-2),$ where $e=p-n.$ With a simple
calculation we get that the relevant  are $j=0,1.$ Based on Table 8.10
of \cite{PleWal} we can verify that $\tilde\Pi(6t+2,7t+1)$ contains the
closure of the ${\cal K}$-orbit $(x^{3}\pm xy^{2},
x^{2}y,y^{3},0,0,0)$ \ (type $E$-$Q^{I}_{4}$).

Germs of type $\Sigma^{n},$ $n=3,4$ are classified in 
\cite{PleWal}, Section 8.6. 

For $n=3,$ the more delicate analysis is that of singularities of type
$\Sigma^{3(2)}.$ Based on  Tables
8.15, 8.17 and 8.20 in \cite{PleWal}, it follows that the moduli does not occur in
strata of codimension $\leq 6t-2,t\geq 5.$ It follows then that
$\tilde\Pi(6t+2,7t+1)\cap \overline{\Sigma^{3(2)}}$ is the closure of
${\cal K}$-orbits of codimensions $> 6t +2.$

For the singularities  of type $\Sigma^{3(3)},$ the
best algebra of this type is the unimodular family whose normal form
is $f_{4\lambda}=(f_{3\lambda},0),$ where $f_{3\lambda}$ is as in
\ref{sec:case1:-nleq-p} \textbf{(4)}.

We know that ${\cal K}$-cod$(f_{3\lambda})=28$ and
$\delta(f_{3\lambda})=7,$ so that  ${\cal K}$-cod$(f_{4\lambda})=28 +6=34>32.$ As the family is $1$-modal it
follows that the codimension of the stratum is $33,$ then this
singularity does not occur generically in $(32,26).$ It is easy to
extend this argument to all pairs $(6t+2,7t+1),\, t> 5.$

The first singularity of type $\Sigma^{4}$ in $(32,36)$ is the unimodular
family \ref{sec:case1:-nleq-p} \textbf{(5)}. The ${\cal
  K}$-cod$\,(f_{\lambda})$ is $33$ and the codimension of the stratum is $32.$

It follows from our description that $\mathrm{cod}\,\tilde\Pi(6t+2,7t+1)>6t+2.$ 

 \noindent $\mathbf{Cases\  (6)\, (8,6);\  (7)\, (10+k, 7)\, k>0}$ 

 These cases are simpler, since the deformations of the algebras have to
 be a simple function singularity, i.e., a singularity from Arnold's
 list of simple singularities of functions\cite{Arn}. We can obtain
 the complete list from the adjacencies of simple and unimodular
 singularities from Arnold's \cite{Arn76}.

 The exceptional pair $(10,7)$ has two modular strata
 \begin{enumerate}
 \item [(i)]  The unimodular family $f_{\lambda}(x,y,z,w)=x^{3}+y^{3}+z^{3}+\lambda
 xyz +w^{2}$ with ${\cal K}$-cod$(f_{\lambda})=11$ and codimension
 of the stratum equal to $10.$
\item [(ii)] The bimodular family $f_{\lambda}(x)= (\sum_{i=1}^{5}a_{i}x_{i}^{2},
\sum_{j=1}^{5}b_{j}x_{j}^{2}),\ \ a_{i}b_{j}-a_{j}b_{i}\neq 0, \,
1\leq i,j\leq 5,\, i\neq j.$
\end{enumerate}

\subsection{Topological triviality of unimodular families}
\label{sec:topol-triv-unim}

Results on $C^{0}$-$\mathcal{A}$-triviality of the unimodular families
of mappings appeared few
years after Mather's theorem,
due mainly to Eduard Looijenga
\cite{Loo77,Loo78} and Jim Damon \cite{Dam80, Dam82}.

In the 1977 paper Looijenga  obtained explicit
examples of topologically stable map-germs which are not analytically
stable. He studied the simple elliptic singularities:\index{Simple elliptic singularity}
\begin{align*}
  &\tilde E_{6}:\ \ f(z_{0}, \dots, z_{n}) = z_{1}(z_{1}-z_{0})(z_{1}-\lambda
z_{0}) + z_{0}z_{2}^{2}+Q(z_{3}, \dots, z_{n}),\  (n\geq 2);\\
&\tilde E_{7}:\ \ f(z_{0}, \dots, z_{n}) = z_{1}z_{0}(z_{1}-z_{0})(z_{1}-\lambda
z_{0}) + Q(z_{2}, \dots, z_{n}),\  (n\geq 1);\\
&\tilde E_{8}:\ \ f(z_{0}, \dots, z_{n}) = z_{1}(z_{1}-z_{0}^{2})(z_{1}-\lambda
z_{0}^{2}) + Q(z_{2}, \dots, z_{n}),\  (n\geq 1).
\end{align*}
where $Q$ is any nondegenerate quadratic form. He proved that two simple-elliptic singularities  
in the same family have topologically equivalent semi-universal
deformations. As a consequence he obtained the
$C^{0}$-$\mathcal{A}$-triviality of the stable unfolding of these
singularities along the moduli parameter.

\begin{remark}
  The family $\tilde E_{6}$ is analytically  equivalent to the family
  \ref{sec:case-2:-np} \textbf{(8)} and $\tilde E_{7}$ is analytically 
equivalent to the family   \ref{sec:case-2:-np}\,\textbf{(7)}. The family $\tilde
E_{8}$ does not occur generically in BND.
\end{remark}

Looijenga's approach to this problem is based on the weighted homogeneity of
the germs together with algebraic calculations to solve a localized
form of equation for infinitesimal $C^{\infty}$ or analytic
triviality.

Wirthm\"uller \cite{Wir} extended Looijenga's results proving the topological
triviality of the versal unfolding of non-simple hypersurfaces  germs
along the Hessian deformation parameter. These results were further
extended by J.Damon \cite{Dam80,Dam82} for unfoldings $F$ of ``non-negative weight'' of a
weighted homogeneous polynomial germ $f:(\mathbb K^{n},0)\to(\mathbb
K^{p},0).$ His main result applies to a large class of unimodular
families, which includes all unimodular families in the boundary
of the nice dimensions.

\begin{theorem}[Damon, \cite{Dam80}]\label{th:dam80}
  If $f$ is a weighted homogeneous $\mathcal{A}$-finitely determined
  germ, then any polynomial unfolding of $f$ of non-negative weight is
  topologically trivial
 \end{theorem}

 Damon's result apply to weighted homogeneous ${\cal A}$-finitely
 determined germs $f$ of type $(w_{1}, \dots,
 w_{n}; d_{1}, d_{2}, \dots, d_{p})$ and their unfoldings of weighted
 degree equal to or higher than the weighted degree of $f.$

 The unimodular families in the boundary of the nice dimensions
 satisfy an even stronger condition: up to the addition of a quadratic
 form, the $\mathcal{K}$-orbits $\mathcal{K}(f_{\lambda})$ in
 \ref{sec:case1:-nleq-p} and \ref{sec:case-2:-np} have a
 homogeneous normal form; in other words we can take weights
 $w_{1}=w_{2}=\dots =w_{n} = 1,$ and if we write $f_{\lambda}:(\mathbb
 R^{s},0)\to(\mathbb R^{t},0), $
 $f_{\lambda}=(f_{1\lambda},f_{2\lambda}, \dots, f_{t\lambda}),$ then
 $f_{i\lambda}$ is homogeneous of degree $d_{i},$ $i=1,\dots, t.$  As
 in section \ref{sec:class-stable-sing} let
 $$N(f_{\lambda})\simeq \frac
 {\Theta(f_{\lambda})}{TK_{e}(f_{\lambda})+
   {\omega}f_{\lambda}(\Theta_{t})}.$$
 Notice that since $f_{\lambda}$ has rank
 $0$, it follows that $N(f_{\lambda})\simeq \frac
 {\mathcal{M}_{s}\Theta(f_{\lambda})}{TK_{e}(f_{\lambda})}.$

 Let $J(f_\lambda)$ be the ideal generated by the $t\times t$ minors
 of $f_{\lambda}$ and let $I(f_{\lambda})= J(f_{\lambda}) +
 f^{*}_{\lambda}(\mathcal{M}_{p}).$ Notice that when $s<t,$
 $I(f_{\lambda})=  f^{*}_{\lambda}(\mathcal{M}_{p}).$

 \begin{lemma}\label{lema8.8}
   \begin{enumerate}
   \item [(a)] If
     \begin{align*}
       I^{1}_{\lambda}&=\langle x^{2}+\lambda yz,
                        y^{2}+\lambda xz, z^{2}+ \lambda xy, xyz
                        \rangle,\ \  \lambda\neq -2,0,1\\
       \intertext{and}
     I^{2}_{\lambda}=&\langle x^{2}+y^{2}+z^{2}, y^{2}+\lambda z^{2}+
                       w^{2}, xy, x z, xw, yz, yw \rangle,\ \  \lambda\neq 0,1
     \end{align*}
then $I^{i}_{\lambda}
     \supseteq \mathcal{M}^{3},$ $i =1,2.$
   \item [(b)] For each normal form $(1)$ to $(5)$ in
     \ref{sec:case1:-nleq-p} and $(6)$ in \ref{sec:case-2:-np},
     $TK_{e}(f_{\lambda}) \supseteq \mathcal{M}^3 \Theta(f_{\lambda}).$
      \item [(c)] For the normal form $(8)$ in \ref{sec:case-2:-np},
     $J(f_{\lambda}) \supseteq \mathcal{M}^{4}.$  
   \item [(d)] For the normal form $(7)$ in \ref{sec:case-2:-np},
     $J(f_{\lambda}) \supseteq \mathcal{M}^{5}.$
   \end{enumerate}
 \end{lemma}

 \begin{proof}
   $(a), (c)$ and $(d)$ follows from easy calculations, using the
   corresponding normal forms.

   To prove $(b)$ notice that if $I(f_{\lambda}) = J(f_{\lambda}) +
   f_{\lambda}^{*}(\mathcal{M}_{t}),$ it follows that
   $I(f_{\lambda})\Theta(f_{\lambda})\subset TK_{e}(f_{\lambda})$, and
   the result follows from $(a).$  
 \end{proof}

With the help of the above Lemma it is an easy task to find, for each
normal form, $(1)$ to $(5)$ in \ref{sec:case1:-nleq-p} and $(6)$ to
$(8)$ in \ref{sec:case-2:-np}, a monomial basis for the normal space
$N(f_{\lambda}),$ so that we can write
$$
N(f_{\lambda})\cong \mathbb K\{\sigma_{1}, \sigma_{2}, \dots
\sigma_{r}, \sigma_{m}\}
$$ where the
$r$ generators $\sigma_{j}=(\sigma_{1j}, \sigma_{2j}, \dots,
\sigma_{tj})\in \theta(f_{\lambda}),$ $j=1,\dots, r$ have the
following property: each coordinate $\sigma_{ij},$ $i=1, \dots, t$ 
of $\sigma_{j}$ satisfies the following condition
$$
\mathrm{degree}\,
\sigma_{ij} < \mathrm{degree} f_{i\lambda}\,\ \ \   i=1,
\dots, t,\,  \ \  j= 1, \dots, r.
$$

The generator $\sigma_{m}=(\sigma_{1m}, \sigma_{2m}, \dots,
\sigma_{tm})$ is the direction of the modulus and the
$\mathrm{degree}\, \sigma_{im} = \mathrm{degree} f_{i\lambda}$ for $i=1, \dots, t.$

For each $\lambda=\lambda_{0},$ the stable unfolding of
$f_{\lambda_{0}}$ is the map-germ
\begin{align}
  F:(\mathbb K^{s}\times \mathbb K^{r} \times \mathbb K, 0) &\to (\mathbb
  K^{t} \times \mathbb K^{r} \times \mathbb K, 0) \label{eq:25} \\
  (x,u,\lambda) &\mapsto (\tilde f(x,u,\lambda), u,\lambda), \nonumber
\end{align}
$x=(x_{1}, \dots, x_{s}),$, $u=(u_{1}, \dots, u_{r}),$ and

$$
  \tilde f(x,u,\lambda)= f(x,\lambda_{0}) +
  \sum_{j=1}^{r}u_{j}\sigma_{j}(x) +\lambda \sigma_{m}(x).
$$

For each $\lambda_{0},$ with the exception of a finite
number of exceptional values, we obtain the normal form of the
unimodular topologically stable   singularity:

$$
  F_{\lambda_{0}}:(\mathbb K^{n}, 0)\to (\mathbb K^{p}, 0),
  $$
  with
  \begin{equation}
    F_{\lambda_{0}}(x,u)=(\tilde f_{\lambda_{0}}(x,u),u), \label{eq:19}
  \end{equation}
where  
\begin{equation}
  \label{eq:26}
  \tilde f_{\lambda_{0}}(x,u)=f(x,\lambda_{0})+ \sum_{j=1}^{r}u_{j}\sigma_{j}(x).
\end{equation}
and $n=s+r,\, p=t+r.$
\begin{remark}
  Notice that $F_{\lambda_{0}}$ is unfolding of $f_{\lambda_{0}}(x)$
  by terms $\sigma_{j}$ of smaller degree. Damon's in \cite{Dam80}
  refers to $F_{\lambda_{0}}$ as unfolding of negative weight of
  $f_{\lambda_{0}}$ ( see section 2 in Damon \cite {Dam82}).
\end{remark}

A similar construction can be made for the exceptional pair
$(n,p)=(10,7).$ The bimodal family $f_{\lambda}=(\mathbb{K}^{5},0)\to
(\mathbb{K}^{2},0),$ $\lambda=\lambda_{1},\lambda_{2}$ has a normal
space
$$ N(f_{\lambda})\simeq \mathcal{R}\{\sigma_{1},\dots,\sigma_{r},\sigma^{1}_{m},\sigma^{2}_{m} \},$$
where $\{\sigma^{1}_{m},\sigma^{2}_{m} \}$ generates the bimodal plane
and $\mathrm{degree}\,\sigma_{m}^{i}=\mathrm{degree}\,f_{\lambda}=2,\,
i=1,2.$ The normal form of the topologically stable singularity is
given by \eqref{eq:19}. 

We display these normal forms in tables below. To simplify notation we denote the canonical basis
in $(\mathbb R^{t},0)$ by $\{e_{i}, i=1, \dots , t\},$ so that an
element  $ g\in \mathcal{E}^{t}_{s}$ can be written as
$g(x)=\sum_{i=1}^{r}g_{i}(x)e_{i}.$
\begin{table}[h]
  \centering
  \begin{adjustbox}{width=\textwidth,center}
    \begin{tabular}{|c|c|c|c|c|}
    \hline
$(n,p)$    & $f=(f_{1},\dots, f_{t})$ & Unfolding monomials $< m$ & $r$ & $\sigma_{m}$\\
\hline\hline
\multirow{2}{*}{$(9,9)$}    & $f_{\lambda}=(x^{2}+\lambda yz, y^{2}+\lambda xz,
                              z^{2}+\lambda xy)$
                                      &$\{y,z\}e_1,\,\{x,z\}e_2,$
                                                            &$6$ &$yze_{1}+xze_{2}+xye_{3}$  \\
           & $\lambda \neq -2, 0, 1 $ & $\{x,y\}e_3$   &   & \\
    \hline
\multirow{3}{*}{$(15,16)$}  & $f_{1\lambda}=(f_{\lambda},Jf_{\lambda}),\,Jf_{\lambda}=xyz$ &
                                                                  $\{y,z\}e_{1},\,\{x,z\}e_{2},$
                                                            & $12$  & $yze_{1}+xze_{2}+xye_{3}$  \\
                            &                &$\{x,y\}e_3,\,\{x,y,z\}e_4,$   &  &  \\
                            &                &$\{yz, xz, xy\}e_4,$   &  &  \\
    \hline
\multirow{4}{*}{$(21,23)$}  & $f_{2\lambda}=(f_{1\lambda},0)$ &$\{y,z\}e_1,\,\{x,z\}e_2,$  & $18$  &  $yze_{1}+xze_{2}+xye_{3}$  \\
                            &                             &$\{x,y\}e_3,\,\{x,y,z\}e_{4} $      &  &  \\
                            &                             & $\{yz,xz,xy\}e_{4} $      &  &  \\
                            &                             & $\{x,y,z\}e_{5},\,\{yz,xz,xy\}e_{5}$  &  &  \\
    \hline
    \multirow{5}{*}{$(27,30)$}  & $f_{3\lambda}=(f_{2\lambda},0)$& $\{y,z\}e_1,\,\{x,z\}e_2,$  & $24$  &  $yze_{1}+xze_{2}+xye_{3}$ \\
                               &
                                      &$\{x,y\}e_3,\,\{x,y,z\}e_{4}$ &
                                                                  & \\
                       &                               & $\{yz,xz,xy\}e_{4}$       &  &  \\
                    &       &$\{x,y,z\}e_{5},\,\{yz,xz,xy\}e_{5}$ 
                                                                  &  & \\
                            &                               & $\{x,y,z\}e_{6},\,\{yz,xz,xy\}e_{6}$  &  &  \\
    \hline
  \end{tabular}
\end{adjustbox}
\caption{$6(p-n) +9= n, \ 3\leq p-n\leq 0$ }
  \label{tab:1}
\end{table}

\begin{table}[h]
  \centering
  \begin{tabular}{|c|c|c|c|c|}
    \hline
$(n,p)$    & $f=(f_{1},\dots, f_{t})$ & Unfolding monomials $< m$ & $r$ & $\sigma_{m}$\\
\hline\hline    
\multirow{4}{*}{$(6s+2, 7s+1)$} &$f_{\lambda}:=( x^{2}+y^{2}+z^{2},
                                  y^{2}+\lambda z^{2}+w^{2}, $ &$\{x,y\}e_1,\,\{z,x\}e_{2}$ & $6s-2$ & $z^{2}e_{2}$\\
                        &$xy,xz,xw,yz, yw, zw, \underbrace{0, \dots, 0})$ &$\{x,y,z,w\}e_{3+i}$ &$s \geq 5 $  &\\
                        &\phantom{$xy,xz,xw,yz, yw, zw,$ }$ s- 5$ & & &\\
              $s\geq 5$ &$t=s+3,\, s\geq 5$ &$0\leq i\leq s,\, s\geq 5$ & &\\
\hline
  \end{tabular}
  \caption{$6(p-n)+8,\, p-n\geq 4,\, n\geq 4 $ }
  \label{tab:2}
\end{table}

\begin{table}[h]
  \centering
  \begin{adjustbox}{width=\textwidth,center}
  \begin{tabular}{|c|c|c|c|c|}
    \hline
$(n,p)$    & $f=(f_{1},\dots, f_{t})$ & Unfolding monomials $< m$ & $r$ & $\sigma_{m}$\\
\hline\hline    
 $(8,6)$ & $(x^{2} +y^{2} +z^{2}, y^{2}+\lambda z^{2}+w^{2}),\,
           \lambda\neq 1$ & $\{y,w\}e_{1},\{x,z\}e_{2}$ & $4$ &
                                                                $z^{2}e_{2}$\\
\hline
\multirow{2}{*}{$(10+k,7),\,k\geq 0$} & $x^{3} +y^{3} +z^{3}+\lambda xyz +\sum_{i=1}^{k}\delta_{i}w_{i}^{2},$
           &$\{x,y,z, yz,xz,xy\}e_{1}$ & $6$& $xyze_{1}$\\
&$\delta_{i}=\pm 1,\, \lambda^{3}\neq -1$ &&&\\
    \hline
$(9,8)$&$x^{4}+y^{4}+\lambda x^{2}y^{2},\,\lambda \neq \pm 2$
                                      &$\{x, y, x^{2},xy, y^{2},
                                        x^{2}y, xy^{2}\}e_{1}$ &$7$&$x^{2}y^{2}e_{1}$ \\
\hline
  \end{tabular}
\end{adjustbox}
\caption{$n>p$}
    \label{tab:3}
\end{table}

\begin{table}[h]
  \centering
  \begin{adjustbox}{width=\textwidth,center}
  \begin{tabular}{|c|c|c|c|c|}
    \hline
exceptional pair  & complex normal form  & Unfolding monomials $< m,$
                                           $m=2$ & $r$ & $\sigma^{1}_{m}, \sigma^{2}_{m}$\\
\hline\hline    
 \multirow{4}{*}{$(10,7)$} & $f_{\lambda_{1}\lambda_{2}}=(p(x),q(x))$     &$\{x_{2},x_{3},x_{4},x_{5}\}e_{1}$
                                          &5 & $\{x_{3}^{2},x_{4}^{2}\}e_{2}$ \\
                            &$p(x)=\sum_{i=1}^{4}x_{i}^{2}$     &  $\{x_{1}\}e_{2}$  & & \\
                            & $q(x)=x_{2}^{2}+\lambda_{1}x_{3}^{2}+\lambda_{2}x_{4}^{2}+x_{5}^{2}$    &  & & \\
                             & $ \lambda_{i}\neq 0,1 \, i=1,2$   &  & & \\
\hline
  \end{tabular}
\end{adjustbox}
\caption{Bimodular strata}
  \label{tab:6}
\end{table}

We remark that, with convenient choices of weights for the variables
$u_{1},\dots, u_{r},$ each normal form $F_{\lambda_{0}}$ is a weighted
homogeneous germ. To apply Damon's result (Theorem \ref{th:dam80}) we need to
show that $F_{\lambda_{0}}$ is ${\cal A}$-finitely determined. The
relevant property of $F_{\lambda_{0}}$ is that the ${\cal A}$-orbit is
open in the ${\cal K} $-orbit, as we now explain.

\begin{definition}
  Let $f:(\mathbb R^{n},0)\to (\mathbb R^{p},0)$ be a
  $\mathcal{A}$-finitely determined map-germ. The
  \emph{$\mathcal{A}$-orbit of $f$ is open in the $\mathcal{K}$-orbit
    of $f$} if $T\mathcal{A}(f)=T\mathcal{K}(f).$ 
\end{definition}
\index{$\mathcal{A}$-orbit open in the $\mathcal{K}$-orbit}
  Given a pair  $(n,p)$ and a $\mathcal{K}^{k}$-orbit in $J^{k}(n,p),$ if
  this $\mathcal{K}^{k}$-orbit does not contain an infinitesimally stable map-germ  $f:(\mathbb
  K^{n},0)\to (\mathbb K^{p},0),$ $j^{k}f(0)\in \mathcal{K}^{k},$ we
  can ask whether there exist $f$ such that $\mathcal{A}^{k}(f)$ is
  open in $\mathcal{K}^{k}(f).$ This was introduced by Ruas \cite{Rua}
  as an approach to the $\mathcal{A}$-classification problem. The non
  existence of $f$ with such property implies that all map-germs $f\in
  \mathcal{K}^{k}$ are non-simple. The following necessary and
  sufficient condition for the existence of an open orbit in
  $\mathcal{K}(f)$ was given in \cite{Rua} (see also Rieger and Ruas \cite{RieRua}).

\begin{proposition}[Ruas, \cite{Rua},Theorem 5.1, Rieger and Ruas,
    \cite{RieRua}, Prop.4.6]\label{cidinha}
  Let  $f:(\mathbb K^{n},0)\to (\mathbb K^{p},0)$ be a
  $\mathcal{K}$-finitely determined germ and denote by
  $\{v_{1},v_{2},\dots, v_{r}\}$ a basis for
  $N=\frac{\theta_{f}}{T\mathcal{A}_{e}f+f^{*}\mathcal{M}_{p}\theta_{f}}.$
  The $\mathcal{A}$-orbit of $f$ is open in the $\mathcal{K}$-orbit
    of $f$ if $f_{i}v_{j}\in T\mathcal{A}f,\, \textrm{mod} (
    f^{*}\mathcal{M}^{2}_{p}\theta_{f})$ for $i=1,\dots, p, \,
    j=1,\dots, r.$
\end{proposition}

To apply proposition \ref{cidinha} to the unimodular singularities at
BND we introduce the following notation, where $F_{\lambda}$ is as
in equation \eqref{eq:19}.

Let $$T_{F_{\lambda}}=F^{*}_{\lambda}(\mathcal{M}_{p})\{\sigma_{1},\sigma_{2},\dots,
\sigma_{r}\}+
tF_{\lambda}(\mathcal{M}_{s+n}\Psi_{s+r})+{\omega}F_{\lambda}(\mathcal{M}_{t+r}\Psi_{t+r}).$$ 
This is a $F_{\lambda}^{*}(\mathcal{E}_{t+r})$-submodule of
$\Psi_{F_{\lambda}}$ consisting of elements of
$T\mathcal{A}(F_{\lambda})$ with zero components in the $\mathbb
R^{r}$ direction (see section \ref{subsec:fi-nite-determined}).

\begin{corollary}\label{cor:8.13}
  Let $F_{\lambda}$ as in \eqref{eq:19}. Then
  $\mathcal{A}(F_{\lambda})$ is open in $\mathcal{K}(F_{\lambda})$ is
  and only if
  \begin{enumerate}
  \item [(i)] $(\tilde f_{\lambda})_{i}\cdot \sigma_{m}\in
    T_{F_{\lambda}}+F^{*}(\mathcal{M}^{2}_{p})\Psi_{F_{\lambda}}, \,
    i=1, \dots, t.$
  \item [(ii)] $u_{j}\cdot \sigma_{m}\in
    T_{F_{\lambda}}+F^{*}(\mathcal{M}^{2}_{p})\Psi_{F_{\lambda}}, \,
    i=1, \dots, r.$
  \end{enumerate}
\end{corollary}

\begin{remark}
  Taking the quotient
  $\frac{T_{F_{\lambda}}}{\mathcal{M}_{u}T_{F_{\lambda}}}$ in 
  condition $(i)$ of Corollary \ref{cor:8.13}, we get

    \begin{equation}
      \label{eq:27}
    (i_{0})\ \   (f_{\lambda})_{i}\cdot \sigma_{m} \in
    \frac{T_{F_{\lambda}}}{\mathcal{M}_{u}T_{F_{\lambda}}}\simeq
    f^{*}(m_{t})\{\sigma_{1},\dots
    \sigma_{r}\}+tf_{\lambda}(m_{s}\Theta_{s})
    +{\omega}f_{\lambda}(M_{t}\Theta_{t}.
  \end{equation}

  The $f^{*}(\theta_{t})$-module
  $\frac{T_{F_{\lambda}}}{\mathcal{M}_{u}T_{F_{\lambda}}}$ is
  $im(z_{0})$ in Damon's notation ( see definition of $z_{0}$ in section 1
  of Damon \cite{Dam82}). 

  Condition $(i_{0})$ is a necessary condition for the property
  $T\mathcal{A}(F_{\lambda})=T\mathcal{K}(F_{\lambda})$ to hold.

\end{remark}

We collect in the following proposition the relevant properties of
$F_{\lambda_{0}}.$

\begin{theorem}\label{th:8.14}
  Let $(n,p)$ be a pair in BND and $F_{\lambda_{0}}:(\mathbb
  K^{n}, 0) \to (\mathbb K^{p}, 0)$  the unimodular map-germ as in \eqref{eq:19}. Then for all
  $\lambda_{0}\in \mathbb K,$ except a finite number of exceptional
  values  the following hold:
  \begin{enumerate}
  \item [(a)] $F_{\lambda_{0}}$ is $\mathcal{A}$-finitely determined.
  \item [(b)] $\mathcal{A}_{e}$-$\cod F_{\lambda_{0}}=1.$
  \item [(c)] The $\mathcal{A}$-orbit of $F_{\lambda_{0}}$ is open in $\mathcal{K}(F_{\lambda_{0}}).$
  \end{enumerate}
\end{theorem}

\begin{proof}
  First notice that $(c) \Leftrightarrow (b) \Rightarrow (a).$ In fact if
  $(c)$ holds,  $T\mathcal{A}(F_{\lambda_{0}})=
  T\mathcal{K}(F_{\lambda_{0}}).$ We saw that $\mathcal{K}$-$\cod
  (F_{\lambda_{0}})= n+1.$ Now, for any $\mathcal{A}$-finitely
  determined $f:(\mathbb  R^{n}, S) \to (\mathbb R^{p}, 0),$
  $S=\{ x_{1},\dots, x_{s}\},$ the following formula due to L. Wilson
  \cite{Wil} holds (see Rieger \cite{Rie} for a proof):

  $$\mathcal{A}_{e}\textrm{-cod} (f)= \mathcal{A}\textrm{-cod} (f)+ s(p-n) -p.$$

  Applying this formula with $s=1,$ it follows that
  $\mathcal{A}_e$-$\cod (F_{\lambda_{0}})=1 \Leftrightarrow
  \mathcal{A}$-$\cod (f)= n+1$  and the equivalence   $(c)\Leftrightarrow
  (b)$ follows from this. It is also clear that $(b)\Rightarrow (a)$.

  We now want to verify $(c)$ ( or
  equivalently $(b)).$ for
  each normal form $F_{\lambda}:(\mathbb K^{n}, 0) \to (\mathbb
  K^{p},0),$
  with $F_{\lambda}(x,u)=(\tilde f_{\lambda}(x,u),u),$ $\tilde
  f_{\lambda}(x,u)= f_{\lambda}(x)+\sum_{j=1}^{r}u_{j}\cdot
  \sigma_{j}(x),$
  $\mathrm{degree}(\sigma_{j})<\mathrm{degree}(f_{\lambda}),$ $j=1,
  \dots, r.$

  To verify $(c),$ we verify condition $(i)$ and $(ii)$ in Corollary
  \ref{cor:8.13} to $F_{\lambda}.$ We do it case by case, collecting
  calculations that appeared previously in the literature.

  \noindent \textbf{(1)} Cases $(n,p)= \{(9,9), (15,16), (21,23), (27,30)\}.$

  These were solved by Damon  in Example 2 and Proposition 8.2,
  \S 8 in \cite{Dam82}.

  Notice that Damon uses Wall's normal form for the $\Sigma^{3,0}$
    unimodular family

    $$f_{\lambda}=(2xz +y^{2}, 2yz
    , x^{2}+3g y^{2}-c z^{2}), c\neq 0,\ \  c+9g^{2}\neq 0.
    $$
Here $c$ is fixed and $g$ is the modulus.
  \end{proof}

\noindent \textbf{(2)} Cases $(n,p)=(8,6)$ and  $(n,p)=(32,36).$

We first consider  $(n,p)=(8,6).$

$F_{\lambda}:(\mathbb K^{8},0)\to (\mathbb K^{6},0),$
$F_{\lambda}=(\tilde f_{\lambda},u),$ where
$$ \tilde f_{\lambda}(x,y,z,w,u)=(x^{2}+y^{2}+z^{2}+ u_{1}y+u_{2}w,
y^{2}+\lambda z^{2}+w^{2}+u_{3}x+u_{4}z).$$

It follows from Lemma \ref{lema8.8} that $F_{\lambda}$ is $2$-determined
with respect to the group $\mathcal{K},$ if $\lambda\neq 0,1.$ The
following follow from simple calculations
\begin{enumerate}
\item [(i)]\, $J(f_{\lambda})+f^{*}_{\lambda}(\mathcal{M}_{2})$
  contains the mixed monomials $xy, xz, xw, yz,yw,zw.$ 
\item [(ii)]\, If $\alpha= x^{4},y^{4},z^{4},w^{4},$ then $\alpha
  e_{i}\in T\mathcal{A}f_{\lambda}\, i=1,2 (\mathrm{mod}\, J(f_{\lambda})\Theta(f_{\lambda}))$.
\end{enumerate}

Using (i) and (ii) it follows that the conditions of Corollary
\ref{cor:8.13} hold, and $\mathcal{A}(F_{\lambda})$ is open in
$\mathcal{K}(F_{\lambda}).$

We leave the calculations of the pair $(n,p)=(32,36)$ as an exercise
for the reader.

\noindent \textbf{(3)} Cases $(n,p)=(9,8)$ and  $(n,p)=(10+k,7), \,
k\geq 0.$

These cases follows from Looijenga \cite{Loo78}, Lemma 2.2.

\begin{remark}
  A similar result holds for the bimodular strata in the pair $(10,7)$
  replacing $\mathcal{A}_{e}$-cod$(F_{\lambda})=1$ by  $\mathcal{A}_{e}$-cod$(F_{\lambda})=2.$
\end{remark}

We summarize the discussion of this section stating the following
results.

\begin{corollary}
  Let $(n,p)$ be a pair in BND and $F_{\lambda_{0}}:(\mathbb
  K^{n},0)\to (\mathbb K^{p},0)$  the unimodular
  map-germ as in \eqref{eq:19}. Then for all $\lambda_{0} \in \mathbb
  K,$  except for a finite number of exceptional values, the one
  parameter unfolding $F:(\mathbb  K^{n}\times\mathbb K,0)\to (\mathbb
  K^{p}\times \mathbb K,0)$ of $F_{\lambda_{0}},$ as in
  \eqref{eq:25}, is ${\cal A}$-topologically trivial.
\end{corollary}

\begin{proof}
  The proof follows from Theorem \ref{th:8.14} and Damon's result
  (Theorem \ref{th:dam80}).
\end{proof}

\begin{corollary}\label{azul}
  Let $(n,p)$ be a pair in BND. Then a Thom-Mather map $f:N^{n}\to
  P^{p}$ has at most a finite set of points $S=\{x_{1}, \dots, x_{r}\}$ such that for all
  $x_{i}\in S,$ $j^kf(x_{i}) \in \mathcal{A}_{M},$ $j^{k}f \pitchfork
  \mathcal{A}_{M},$ where $\mathcal{A}_{M}$ is any of the modal
  stratum of ${\cal A}^{k}(N,P).$ Moreover, if 
  $f(x_{i})=y_{i},\,i=1,\dots r$ then
  $f^{-1}(y_{i}) \cap \Sigma(f)=\{x_{i}\}, \ i=1, \dots, r.$ The
  restriction of $f$ to  $N\setminus S$ is a infinitesimally stable map.

\end{corollary}
\index{Unimodular strata|)}

 \subsection{Notes}
 \label{sec:notes-2}
\emph{Density of $C^{1}$ stable mappings.} In \cite{PleWal89}, du
Plessis and Wall determine the precise range of dimensions where
$C^{1}$-stable maps are dense. This property holds if and only if the
pair $(n,p)$ is in the nice dimensions.

A parallel result is also obtained when $C^{1}$-stability is replaced
by $\infty$-$C^{1}$-determinacy. We say that a map-germ
$f\in\mathcal{E}_{n}^{p}$ is $\infty$-determined with respect to
$C^{1}$-$\mathcal{A}$-equivalence if the $C^{1}$-$\mathcal{A}$-orbit
of $f$ contains all $g\in \mathcal{E}_{n}^{p}$ such that
$j^{\infty}g(0)=j^{\infty}f(0).$ we can also denote the group
$C^{1}$-$\mathcal{A}$ by $\mathcal{A}^{(1)}.$

The paper \cite{PleWal89} appeared in 1989. In contrast with the $C^{0}$
and $C^{\infty}$ cases much less was known in the $C^{1}$ case. Wall \cite{Wal80-2}
sketched in 1980 the proof that $C^{1}$-stable maps are not dense when
$n=8$ and $p=6$ and Mather \cite{Mat69-3} proved that finite
$\mathcal{A}^{(1)}$determinacy does not hold in general for map-germs
$(\mathbb R^{n},0) \to (\mathbb R^{n+1},0),$ with $n\geq 15.$

The main result of \cite{PleWal89} is the following theorem:(A) if
$(n,p)$ is in complement of the nice dimensions, then for any smooth
manifolds $N,P$ there is a nonempty open subset $U\subset C^{\infty}$
containing no $C^{1}$-stable mapping. (B) If $(n,p)$ is in the
complement of semi-nice dimensions (see \cite{Ple, Wal85} for details) with
the exception of the pairs $(14.14),$ $(15,15),$ $(16,16),$ $(12,11),$
$(14,12)$ and $(15,13),$ then for any pair of smooth manifolds $N,P$
there is a nonempty open subset $U\subset C^{\infty}$ containing no
map all of whose point-germs are $\infty$-$\mathcal{A}^{1}$-determined.

The proof of this theorem follows the line of the proof of the
corresponding $C^{\infty}$ result. It is shown that $C^{1}$stability
implies transversality and $\infty$-$\mathcal{A}^{(1)}$-determinacy
implies transversality off the base-point to the fibres of a
$\mathcal{K}$-invariant fibred submanifold of $J^{r}(n,p)$ in the
complement of the set $W^{r}(n,p)$ of $r$-jets with
$\mathcal{K}^{r}$-modality $\geq 1.$ This follows from the property
that stability an determinacy  conditions imply a weak
form of transversality (the preimage is a $C^{1}$-submanifold).
 To strengthen this to actual transversality the use of unfolding theory
 and a perturbation lemma of R.D. May \cite{May73} were the important tools.

Several notions of $C^{1}$-invariance of submanifolds of jet space are
discussed in \cite{PleWal}. In particular, the $C^{1}$-invariance of the Thom-Boardman
varieties and, in some cases, of $\mathcal{K}^{r}$-orbits within them
are obtained.

\section{Density of  Lipschitz stable mappings}
\label{sec:non-trivial-example}

\index{Map!stable!Lipschitz}
We discuss here the problem of density of Lipschitz stable mappings,
which is still widely open.

In \cite{NguRuaTri21} Nguyen, Ruas and\index{Lipschitz nice dimensions}
Trivedi introduced the \emph{Lipschitz nice dimensions (LND)} as the pairs
$(n,p)$ for which the set $\mathcal{S}^{Lip}(N,P)$ of 
Lipschitz stable mappings is \emph{dense}  in $C_{pr}^{\infty}(N^{n}, P^{p}).$

When $N$ is compact, it is clear that the LND contains Mather's nice dimensions, since every
$C^{\infty}$ stable mapping is Lipschitz stable. The main purpose in
Nguyen, Ruas and Trivedi \cite{NguRuaTri21} is to give an answer for
the following conjectures.

\begin{conjecture}\label{lnd}
  The Lipschitz nice dimensions contains  Mather's nice dimensions and
  its boundary.
\end{conjecture}

\begin{conjecture}
 The result in Conjecture \ref{lnd} is sharp, that is, if $(n,p)$ is in the complement of the nice dimensions or its boundary then
$\mathrm{S}^{Lip}(N,P)$ is not a dense set in $C^{\infty}(N,P).$
 \end{conjecture}

The following result is proved by Ruas and Trivedi \cite{RuaTri}.
\begin{theorem}[Section 6, \cite{RuaTri}]
  The unimodular strata
in the boundary of the nice dimensions are bi-Lipschitz
$\mathcal{K}$-trivial.
\end{theorem}

\begin{remark}
  The exceptional unimodular strata when $(n,p)=(10,7)$ also safisfies
  bi-Lipschitz $\mathcal{K}$-triviality condition.
\end{remark}

We first review the notions of $\mathcal{K}$-equivalence and
$\mathcal{K}$-triviality of $r$-parameter deformations.

\index{Bi-Lipschitz!$\mathcal{K}$-equivalence of $r$-parameter
  deformation}
\begin{definition}
  A bi-Lipschitz  $\mathcal{K}$-equivalence of $r$-parameter
  deformations is a pair $(H,K)$ of bi-Lipschitz germs
  $H:(\mathbb R^{r}\times \mathbb R^{n}, 0) \to (\mathbb R^{r}\times
  \mathbb R^{n}, 0)$ and $K:(\mathbb R^{r}\times \mathbb R^{n}\times
  \mathbb R^{p}, 0)\to (\mathbb R^{r}\times \mathbb R^{n}\times \mathbb
  R^{p}, 0)$ with $H$ an $r$-parameter unfolding at $0$ of the germ of the
  identity map of $\mathbb R^{n},$ and $K$ an $r$-parameter unfolding
  at $0$ of the germ of the identity in $\mathbb R^{n}\times \mathbb
  R^{p}$ such that the following diagram commutes

  $$
  \xymatrix{ 
(\mathbb R^{r}\times \mathbb R^{n}, 0) \ar[r]^{i} &(\mathbb
R^{r}\times \mathbb R^{n}\times \mathbb R^{p}, 0) \ar[r]^{\pi} &(\mathbb R^{r}\times \mathbb R^{n}, 0) \\
(\mathbb R^{r}\times \mathbb R^{n}, 0) \ar[u]_{H}\ar[r]^{j}   &  (\mathbb
R^{r}\times \mathbb R^{n}\times\mathbb R^{p}, 0)
\ar[r]^{\pi}\ar[u]_{K} & (\mathbb R^{r}\times \mathbb R^{n},
0)\ar[u]_{H}  }
$$

Here $i$ is the canonical inclusion and $\pi$ is the canonical
projection. Two $r$-parameter deformations $\Phi$ and $\Psi$ of $f$ are
bi-Lipschitz $\mathcal{K}$-equivalent if there exist a bi-Lipschitz
$\mathcal{K}$-equivalence $(H,K)$ as above such that
$$K\circ (id,\phi)=(id,\Psi)\circ H.$$

If $(H,K)$ has the special property that $H$ is the germ of the
identity on $\mathbb R^{n},$ then $(H,K)$ is said to be a
$\mathcal{C}$-equivalence and $\phi$ and $\Psi$ are said to be
$\mathcal{C}$-equivalent deformations.
\end{definition}
\index{Bi-Lipschitz!$\mathcal{K}$-trivial}
\index{Bi-Lipschitz!$\mathcal{C}$-trivial}
\begin{definition}
  An $r$-parameter deformation $\Phi$ of a germ $f:(\mathbb R^{n},0)\to
(\mathbb R^{p},0)$ is \emph{bi-Lipschitz $\mathcal{K}$-trivial}
(resp. \emph{bi-Lipschitz $\mathcal{C}$-trivial}) if it is
bi-Lipschitz $\mathcal{K}$-equivalent (resp. bi-Lipschitz
$\mathcal{C}$-equivalent) to the deformation $\Psi: (\mathbb
R^{r}\times \mathbb R^{n}, 0) \to (\mathbb R^{p}, 0),$ given by $\Psi (u,x)=f(x).$
\end{definition}

A sufficient condition for bi-Lipschitz $\mathcal{K}$-triviality is
the following Thom-Levine type lemma.

\begin{lemma}
Let $F: (\mathbb R^{n}\times \mathbb R, 0) \to (\mathbb R^{p}, 0)  $
be a one parameter deformation of $f: (\mathbb R^{r}, 0) \to (\mathbb
R^{p}, 0).$ If there exist a $p\times p$ matrix $(a_{ij})$ (not
necessarily invertible) with entries germs of Lipschitz functions $ (\mathbb
R^{n}\times \mathbb R, 0)$ and a germ of a Lipschitz vector field $X$
of the form
$$
X=\frac{\partial}{\partial t} +\sum_{i=1}^{n}X_{i}(x,t)\frac{\partial}{\partial x_{i}}
$$
with $X_{i}(0,t)=0$ such that
\begin{equation}
  \label{eq:28}
  X \cdot
\begin{bmatrix}
  F_{1}\\ \vdots \\ F_{p}
\end{bmatrix}
=
\begin{bmatrix}
  a_{11} &\dots & a_{1p}\\
  \vdots & \dots &\vdots\\
  a_{p1} &\dots &a_{pp}
\end{bmatrix}
\begin{bmatrix}
  F_{1}\\ \vdots \\ F_{p}
\end{bmatrix}
\end{equation}

Then, $F$ is a bi-Lipschitz $\mathcal{K}$-trivial deformation.
\end{lemma}

The proof follows from the fact the integration of a Lipschitz vector
field gives a bi-Lipschitz flow. In fact, the bi-Lipschitz
trivialization in source is given by integrating the vector field
$X$ and that in the product is given by integration of the vector
field $W,$ where 

$$
W(x,y,t)=\frac{\partial}{\partial t} +\sum_{i=1}^{p}W_{i}(x,y,t)\frac{\partial}{\partial y_{i}}
$$
where $W_{i}(x,y,t)=\sum_{j=1}^{p}a_{ij}y_{j}$

The converse of the above lemma is not known and so we say that  a one
parameter deformation is \emph{strongly bi-Lipschitz
  $\mathcal{K}$-trivial} if the conditions of the above lemma hold.

If $X_{i}(x,t)\equiv 0,$ $i=1,\dots, n,$  condition \eqref{eq:28}
implies that $F$ is $\mathcal{C}$-trivial.

 A case by case proof of the
bi-Lipschitz $\mathcal{K}$-triviality of the unimodular strata
\ref{sec:case1:-nleq-p} and \ref{sec:case-2:-np} is given in Ruas and
Trivedi \cite{RuaTri}. 
The cases $n\leq p$ and $n>p$ are treated
separately.

When $n\leq p,$ the modal families are families of finite maps. For
them, $\mathcal{K}$-determinacy holds if and only if
$\mathcal{C}$-determinacy holds (see Wall \cite{Wal81}, Prop. 2.4). In this case,
we can apply the Lipschitz Thom-Levine lemma to prove the bi-Lipschitz
$\mathcal{C}$-triviality of these families.

We discuss here the case $n=p=9.$

\begin{lemma}[Ruas and Trivedi, \cite{RuaTri}, Lemma 6.1]\label{ruatri}
  The unimodular family \ref{sec:case1:-nleq-p} (1)
  $$
  F(x,y,z,\lambda)= (x^{2}+\lambda yz,y^{2}+\lambda xz, z^{2}+\lambda xy),
  $$
$\lambda\neq -2,0,1,$ is strongly bi-Lipschitz $\mathcal{C}$-trivial.
\end{lemma}

\begin{proof}
  Let $\mathcal{I}$ be the $\mathcal{E}_{4}$-ideal generated by the components
  of $F,$ i.e.,
  $$
  \mathcal{I}=\langle x^{2}+\lambda yz,y^{2}+\lambda xz, z^{2}+\lambda xy \rangle.
  $$

  We can prove that $\mathcal{I}\supset
  \mathcal{M}^{4}_{3}\mathcal{E}_{4},$ where $\mathcal{M}_{3}$ is the
  ideal generated by $x,y,z.$ More precisely
  \begin{equation}
    \label{eq:22}
    \mathcal{I}\cdot\mathcal{M}^{2}_{3}\mathcal{E}_{4}=\mathcal{M}^{4}_{3}\mathcal{E}_{4}
  \end{equation}

  Consider the following control function
  $\rho(x,y,z,\lambda)=\sqrt{F_{1}^{2}+F_{2}^{2}+F_{3}^{2}}.$ Since
  $F_{\lambda}$ is $\mathcal{C}$-finitely determined and homogeneous
  of degree $2$ for all $\lambda\neq -2,0,1,$ there exist constants
  $c$ and $c',$ (see Ruas \cite{Rua86}), such that
  $$
c' ||(x,y,z)||^{2}\leq \rho(x,y,z,\lambda) \leq c ||(x,y,z)||^{2}
$$

From \eqref{eq:22} it follows that there exists a $3\times 3$ matrix
$(a_{ij})$ with entries in $\mathcal{M}^{4}_{3}\mathcal{E}_{4}$ such
that
$$ \rho^{2}(x,y,z,\lambda)
\begin{bmatrix}
  \frac{\partial F_{1}}{\partial\lambda}\\
  \frac{\partial F_{2}}{\partial\lambda}\\
  \frac{\partial F_{3}}{\partial\lambda}
\end{bmatrix}
=
\begin{bmatrix}
  a_{11} &a_{12} &a_{13}\\
  a_{21} &a_{22} &a_{23}\\
  a_{31} &a_{32} &a_{33}\\
\end{bmatrix}
\begin{bmatrix}
  F_{1}\\F_{2}\\ F_{3}
\end{bmatrix}
$$
\end{proof}

Now consider the germ  of the vector field $V$ on $(\mathbb R^{3}\times
\mathbb R^{3}\times \mathbb R, 0)$ defined by
$$
V= \frac{\partial}{\partial \lambda}+
\frac{1}{\rho^{2}} \left\{ \sum_{j=1}^{3}a_{1j}Y_{j}\frac{\partial}{\partial
      Y_{1}}+
  \sum_{j=1}^{3}a_{2j}Y_{j}\frac{\partial}{\partial Y_{2}}+
  \sum_{j=1}^{3}a_{3j}Y_{j}\frac{\partial}{\partial Y_{3}}
\right\}
$$
where $(Y_{1}, Y_{2}, Y_{3})=Y$ are the target coordinates. Notice
that $\frac{a_{ij}Y_{j}}{\rho^{2}}$ are continuous in a 
neighborhood of the origin in $(\mathbb R^{3}\times
\mathbb R^{3}\times \mathbb R, 0),$ but the derivative with respect to
$x,y,z$ are not bounded, so that $V$ is not Lipschitz. However
we can modify $V$ to get a Lipschitz vector field $V'=pV$ where
$p:(\mathbb R^{3}\times \mathbb R^{3} \times \mathbb R,0) \to (\mathbb
R,0)$ is defined  as follows.

Let $D_{1}=\{||Y||\leq c_{1}||(x,y,z,\lambda)|| \}$ and  $D_{2}=\{||Y||\geq
c_{2}||(x,y,z,\lambda)|| \}$ be cones in $(\mathbb R^{3}\times
\mathbb R^{3}\times \mathbb R)$ with $c_{1}< c_{2}$ and let $p$ be
defined by
$$p(x,y,z,\lambda,Y)=
\begin{cases}
  1 & \mathrm{if}\  (x,y,z,\lambda,Y) \in D_{1}\\ 
  0 & \mathrm{if}\  (x,y,z,\lambda,Y) \in D_{2}
\end{cases}
$$
and $0<p(x,y,z,\lambda,Y)<1$ if $c_{1}||(x,y,z,\lambda)||<
|Y|<c_{2}||(x,y,z,\lambda)||,$ such that the derivative of
$p(x,y,z,\lambda,Y)$ with respect to any coordinate is bounded by a
real number $K$ (see Ruas \cite{Rua86} for details).

The integration of $V'$ will give a bi-Lipschitz
$\mathcal{C}$-trivialization of $F$ by the Thom-Levine criterion. This
completes the proof.

\begin{remark}
  For any fixed $\lambda=\lambda_{0}\neq -2,0,1,$ the deformation
  $F(x,y,z,\lambda)$ in Lemma \ref{ruatri} is semialgebraic and
  satisfies the condition
  $\frac{|F_{\lambda}(x,y,z)|}{|F_{\lambda_{0}}(x,y,z)|}$ is bounded
  for any $(x,y,z,\lambda)$ in $(\mathbb R^{3}\times \mathbb R,0).$
  Then we can also apply  Theorem 3.1 of Ruas and Valette \cite
  {RuaVal} to prove that $F_{\lambda}$ is semialgebraically
  bi-Lipschitz $\mathcal{K}$-trivial. Notice however that the
  conclusion in Lemma \ref{ruatri} is stronger, as we prove that the family
  $F_{\lambda}$ is strongly bi-Lipschitz $\mathcal{K}$-trivial.
\end{remark}

 {The bi-Lipschitz $\mathcal{K}$-triviality of the Thom-Mather
stratification along the unimodular strata  in the boundary of the
nice dimensions suggest that mappings transverse to this
stratification are bi-Lipschitz stable. }

{A natural approach to prove Conjecture \ref{lnd} is to follow the proof of Theorem \ref{th:dam80}, taking into account that the
pair $(n,p)$ is in the boundary of the nice dimensions.}

 {    We saw in Corollary \ref{azul}
  that a Thom-Mather map $f:N^{n}\to P^{p},$  $(n,p)$ in the boundary
  of the nice dimensions  has at most a finite set
  of points $S=\{x_{1}, \dots, x_{\ell}\}$ such that for all
  $x_{i}\in S,$ $j^kf(x_{i}) \in \mathcal{A}_{M},$ $j^{k}f \pitchfork
  \mathcal{A}_{M},$ where $\mathcal{A}_{M}$ is the modal
  stratum. Moreover by multi-transversality, if 
  $f(x_{i})=y_{i},\,i=1,\dots \ell$ then
  $$f^{-1}(y_{i}) \cap \Sigma(f)=\{x_{i}\}, \ i=1, \dots, \ell.
  $$
}

  Clearly, $f$ is  an infinitesimally stable mapping in the
  complement of $S.$

  To prove   that $f$ is Lipschitz stable it would be  sufficient to prove
  that each unimodular family $F_{\lambda}$ (see Section \ref{sec:topol-triv-unim}), and also the bimodular family
  when $(n,p)=(10,7),$  is bi-Lipschitz
  $\mathcal{A}$-trivial.

  Let $$F(x,u,\lambda)=
  (\bar f(x,u,\lambda),u,\lambda)$$ be the (weighted homogeneous) normal form of a unimodular
  family in BND  as in  \eqref{eq:25}, where $x=(x_{1}, \dots, x_{s}),$
  $u=(u_{1}, \dots, u_{r}),$ $s+r =n$ and $\bar f = (\bar f_{1},
  \dots, \bar f_{t}).$

Following the proof of Theorem \ref{th:dam80}, we can find weighted homogeneous vector fields
$V$ and $W$ in  source and  target respectively, given by:

 $$V(x,u,\lambda) =
 \sum_{j=1}^{s}v_{j}(x,u,\lambda)\frac{\partial }{\partial x_{j}}+\sum_{i=1}^{r}\overline
   v_{i}(\tilde f, u,\lambda)\frac{\partial}{\partial u_{i}} + \frac{\partial}{\partial\lambda}
$$
where $x=(x_{1}, \dots , x_{s}),$ $u=(u_{1},u_{2}, \dots, u_{r})$ and
$\tilde v_{i}(0,0,\lambda)=v_{j}(0,0,\lambda)=0,$

{ $$W(X,U,\lambda) =
 \sum_{j=1}^{t}w_{j}(X,U,\lambda)\frac{\partial }{\partial X_{j}}+\sum_{i=1}^{r}\overline
   w_{i}(X, U,\lambda)\frac{\partial}{\partial U_{i}}+ \frac{\partial}{\partial\lambda}
$$
where $X=(X_{1},\dots , X_{t}),$ $U=(U_{1}, \dots, U_{r}),$ and
$\overline w_i(0,0,\lambda)= w_{j}(0,0,\lambda)=0,$ (capital letters
denote the coordinates in the target),and a weighted homogeneous control function
$\rho(X,U,\lambda)$ such that}

\begin{equation}
  \label{eq:23}
(\rho\circ F)(x,u,\lambda)\frac{\partial \tilde f}{\partial \lambda} =
-\sum_{j=1}^{3}\frac{\partial \tilde f}{\partial x_{j}}v_{j}(x,u,\lambda)
-\sum_{i=1}^{6}\frac{\partial \tilde f}{\partial u_{i}}\overline v_{i}(\tilde f, u,\lambda) +\tilde W(\tilde f, u,\lambda)
\end{equation}
where $\tilde W=(w_{1}, \dots, w_{t}).$

{It follows from \eqref{eq:23} that the vector fields  $\mathcal{X}(x,u,\lambda)=\frac 1 {(\rho\circ F)(x,u,\lambda)} {V}(x,u,\lambda)$ and
$\mathcal{Y}(X,U,\lambda) =\frac 1{\rho(X,U,\lambda)}{W}(X,U,\lambda)$
 satisfy the equation
$DF(\mathcal{X})=\mathcal{Y}\circ F.$ Moreover, they are continuous and can be integrated to give the topological $\mathcal A$-triviality
of $F$ along the moduli space.}

{If we can prove
that $\mathcal{X}$ and $\mathcal{Y}$ are Lipschitz vector fields in the 
source and target, respectively, the  bi-Lipschitz
$\mathcal{A}$-triviality of $F$ would follow from the Lipschitz version
of the Thom-Levine lemma.}

\section{Sections of discriminant of stable germs. Open Problems}
\label{sec:open-problems}

An important consequence of Theorems \textbf{A} and \textbf{B} is that
we can approximate any map $f:U\subset\mathbb K^{n}\to \mathbb K^{p},$
$\mathbb K=\mathbb R$ or  $\mathbb C,$ by a stable mapping if $(n,p)$ is
in the nice dimensions or else by a topologically stable map if
$(n,p)$  is not in the nice dimensions.

For a map-germ of finite singularity type $f:(\mathbb K^{n},0)\to
(\mathbb K^{p},0),$ a stable perturbation can be realized as the
generic member of a $1$-parameter unfolding $\bar
f(x,t)=(f_{t}(x),t)$ of $f.$ More precisely, $\bar f$ is a \index{Stabilization}
\emph{stabilization} of $f$ if there exists a representative $\bar f:U
\to V\times T$ such that $f_{t}:U\cap (\mathbb K^{n}\times \{t\}) \to
V$ is stable for all $t\neq 0.$

When $\mathbb K=\mathbb C,$ the stable perturbation of $f$ is uniquely
determined up to $\mathcal{A}$-equivalence  when $(n,p)$ is in the nice
dimensions and up to $C^{0}$-$\mathcal{A}$-equivalence  otherwise.
When $\mathbb K=\mathbb R,$ there may exists a finite number of
nonequivalent stabilizations  of $f.$  On the reals, in general $t>0$ and $t<0$  
give non-equivalent perturbations of $f$ (see Mond and
Nu\~no-Ballesteros in this Handbook or \cite{MonNun} for details).

The geometry of the stable perturbations $\bar f$ are associated to
invariants of the germ $f$.

We discuss here this important  tool in singularity theory.

Let $f:(\mathbb K^{n},0)\to (\mathbb K^{p},0),$ $\mathbb K=\mathbb R$
or $\mathbb C$ be a germ of a finite singularity type and $F$ its
stable unfolding:

\begin{equation}
  \label{eq:32}
  \xymatrix{
(\mathbb K^{n'},0) \ar[r]^{F}       & (\mathbb K^{p'},0)  \\
(\mathbb K^{n},0) \ar[u]\ar[r]_{f}  & (\mathbb K^{p},0)\ar[u]^{g}}
\end{equation}
where $g$ is the germ of an immersion transverse to $F.$

Let $V=\Delta(F)$ be the discriminant of $F$ (recall that when $n<p$
the discriminant is the image $F(\mathbb K^{n})$.) Damon in \cite{Dam91}
described a relation between $\mathcal{A}$-equivalence and properties
of the discriminant $V.$ This relation is valid for all pairs $(n,p)$
and directly relates $\mathcal{A}_{e}$-codimension of $f$ with
a codimension of the germ at $0$ of $g(\mathbb K^{p})$ as a section of 
the discriminant. The idea of using sections of the discriminant to
determine $\mathcal{A}$-determinacy properties of $f$ appears in
\cite{Mar76-1, Mar76} (see also du Plessis \cite{Ple}). However, the precise
relation between an equivalence relation for germs of immersions $g$
and the $\mathcal{A}$-equivalence of $f$ was derived in \cite{Dam91}.

Given the germ of a variety $(V,0) \subset (\mathbb K^{p'},0)$ Damon
defined the group $\mathcal{K}_{V}$ of \emph{contact equivalences
  preserving $V$} which acts on the set of germs $g:(\mathbb
K^{p},0)\to (\mathbb K^{p'},0)$  (the map-germs $g$ are in
$\mathcal{E}^{p'}_{p}$ when $\mathbb K =\mathbb R$ or in
$\mathcal{O}^{p'}_{p}$ when $\mathbb K =\mathbb  C.$)

\index{Contact!group $\mathcal{K}_{V}$}
The \emph{contact group $\mathcal{K}_{V}$} is defined as follows:
$$
\mathcal{K}_{V}=\{(h,H)\in \mathcal{K}|\, H(\mathbb K^{p}\times V)\subseteq\mathbb K^{p}\times V\}
$$
(see definition \ref{def:contact}).

The action of $\mathcal{K}_{V}$ on $\mathcal{E}^{p'}_{p}$ or
$\mathcal{O}^{p'}_{p}$ is defined as in definition
\ref{def:contact}. We can also define the similar notions for
unfoldings.The group $\mathcal{K}_{V}$ is a geometric subgroup of the
contact group, so that the machinery of singularity theory applies to
$\mathcal{K}_{V}$-equivalence. In particular the infinitesimal and the
geometric criteria for $\mathcal{K}_{V}$-determinacy.

We can define
\begin{align*}
  T\mathcal{K}_{V}\cdot g &= tg(\mathcal{M}_{p}\Theta_{p}) +
                           \epsilon_{p}\{\eta_{i}\circ g,\, i=1,\dots,m\}\\
  T\mathcal{K}_{V_{e}}\cdot g &=tg(\Theta_{p}) +  \epsilon_{p}\{\eta_{i}\circ g,\, i=1,\dots,m\}\\
\end{align*}
{where $\eta_{i},\, i=1,\dots,m$ are the generators of $\Theta_{V},$
the $\epsilon_{p'}$-module of vector fields in $\mathbb K^{p'}$ tangent to the variety $V$
at its smooth points. Equivalently, $\Theta_{V} $ is the
$\epsilon_{p'}$ module of derivations of $\Theta_{p'}$ which preserve
the ideal defining $V.$ The notation $\mathrm{Der(-log V)}$ proposed by
Saito for the module of these vector fields as well the notation $g^{*}(\mathrm{Der(-log V)})$ for
the $\epsilon_{p}$-module $\epsilon_{p}\{\eta_{i}\circ g,\,
i=1,\dots,m\}$ are also widely used. See section 2.9 of the article of
Nu\~no-Ballesteros and David Mond in this Handbook \cite{Han}.}

$T\mathcal{K}_{V}g$ and $  T\mathcal{K}_{V_{e}}g$ are
$\mathcal{O}_{p}$-modules when $\mathbb K=\mathbb C$ and $\mathcal{E}_{p}$-modules when $\mathbb K=\mathbb R.$ 

With the notations as in \eqref{eq:32} we can state the main results
in \cite{Dam91} as follows.

\begin{enumerate}
\item $g$ has finite $\mathcal{K}_{V}$-codimension if and only if $f$
  has finite $\mathcal{A}$-codimension.
\item If $N\mathcal{A}_{e}f$ and $N\mathcal{K}_{V_{e}}g$ denote the
  normal spaces to  $\mathcal{A}_{e}f$ and $\mathcal{K}_{V_{e}}g,$
  respectively, then $$N\mathcal{A}_{e}f \cong N\mathcal{K}_{V_{e}}g.$$
\item $\mathcal{A}_{e}$-codimension$(f)=\mathcal{K}_{V_{e}}$-codimension$(g).$ 
\item Conditions 1. to 3. hold for multigerms $f:(\mathbb
K^{n},S)\to (\mathbb K^{p},0).$
\end{enumerate}

The geometric characterization of $\mathcal{K}_{V}$-equivalence holds
only for holomorphic map-germs $f\in \mathcal{O}_{n}^{p},$ namely: $f:(\mathbb
C^{n},0)\to (\mathbb C^{p},0)$ is $\mathcal{A}$- finitely determined
if and only if $g$ is transverse to the strata of $V$ away from the
origin. For real germs, the geometric condition is a necessary
condition for $\mathcal{K}_{V}$finite determinacy, but the converse
does not hold.

Damon's theory builds a solid bridge between  singularity theory of
mappings and topology of singular varieties. This connection has
been used successfully  for the past three decades. We follow this 
approach to formulate some open problems in singularity theory,
related to the subject discussed in this paper.

\subsection{Geometry of sections of discriminant of stable mappings in
  the nice dimensions}
\label{sec:sect-discr-stable}

Let $(n+1,p+1)$ be a nice  pair of dimensions and $F:(\mathbb
K^{n+1},0)\to (\mathbb K^{p+1},0)$ a minimal stable map-germ. Minimal
here means that  $\{0\}\in \mathbb K^{n+1}$ is a stratum on the
stratification of $F$ by stable types.  A hyperplane section
$H=g(\mathbb K^{p})$ transversal to the discriminant $V=\Delta(F)
\subset \mathbb K^{p+1}$ away from the origin pulls back by $F$ to an
$\mathcal{A}$-finite map-germ  $f:(\mathbb K^{n},0) \to (\mathbb
K^{p},0).$

From Damon's result 3. above, it follows that if $(n,p)$ is in the
semi-nice dimensions (see section \ref{sec:how-strat-mapp}) there exists an
open and dense set $\mathcal{I}$ of immersions $g:(\mathbb K^{p},0)\to
(\mathbb K^{p+1},0)$ such that the pull back of $g$ by $F$ is an
$\mathcal{A}$-finite map germ $f:(\mathbb K^{n},0)\to (\mathbb
K^{p},0)$ whose $\mathcal{A}_{e}$-codimension is minimal, that is,
$$
\mathcal{A}_{e}\mathrm{-cod}\, f \leq \mathcal{A}_{e}\mathrm{-cod}\, f',
\ \ \mathrm{for\  all }\ f'\widesim{\mathcal{K}} f.
$$

As $F$ is a minimal stable unfolding of $f$ we may ask: is there a
map-germ $f:(\mathbb K^{n},0)\to (\mathbb K^{p},0),$ $Q(f)\cong Q(F)$
such that $\mathcal{A}_{e}\mathrm{-cod}\, f=1,$ which in this case implies
that $\mathcal{A}$-orbit of $f$ is open in  its
$\mathcal{K}$-orbit?

It follows form Proposition \ref{cidinha} that this condition holds if and
only if it holds for a general linear hyperplane section (see
\cite{Mir} for the case $(n,n+1)).$
 Notice however that sections of $\Delta(F)$ minimizing
 $\mathcal{A}_{e}$-codimension are not necessarily linear (see section
 3.1 in \cite{AtiRuaSin}). The complete answer to the question above
 appears in \cite{AtiRuaSin}.

 \begin{theorem}[\cite{AtiRuaSin}, Theorem 4.6]
   If the pair $(n,p)$ is in the extra-nice dimensions, then every
   stable germ  $F:(\mathbb K^{n+1},0)\to (\mathbb K^{p+1},0)$ admits
   a section of $\mathcal{A}_{e}$-codimension $1$  $f:(\mathbb
K^{n},0)\to (\mathbb K^{p},0).$ The converse is true if $(n+1,p+1)$ is
in the nice dimensions.
 \end{theorem}

 \begin{corollary}
   If $\mathbb K=\mathbb C$ and $(n,p)$ is in the extra-nice
   dimensions any two generic hyperplane sections $g$ and $g'$ of the
   discriminant $\Delta(F)$ of a stable germ $F:(\mathbb K^{n+1},0)\to
   (\mathbb K^{p+1},0)$  pull back by $F$ to $\mathcal{A}$-equivalent
   germs $f,\,f':(\mathbb K^{n},0)\to (\mathbb K^{p},0).$ Moreover
   $\mathcal{A}_{e}\mathrm{-cod}\, f = \mathcal{A}_{e}\mathrm{-cod}\, f'.$ 
 \end{corollary}

 \begin{remark}
   When $\mathbb K=\mathbb C,$ $p\leq n+1$ and $(n,p)$ is in the nice
   dimensions, the topology of the stabilization of holomorphic
   $\mathcal{A}_{e}$-codimension 1, $\mathrm{corank}\, 1$ germs and
   multigerms is well understood. See \cite{AtiCooMon} where
   T. Cooper, D. Mond and Wik-Atique classify these singularities and
   study the topology of their stabilizations.
 \end{remark}

 \noindent\underline{Problem 1}. To study the geometry of generic
 hyperplane sections of the discriminant of stable mappings in 
 $(n+1,p+1)$ when $(n,p)$ is in extra-nice dimensions and its boundary.

 \noindent\underline{Problem 2}. To study equisingularity of families of
 generic hyperplane sections $g_{t}(\mathbb C^{p})$ of the
   discriminant $\Delta(F)$ of stable map-germs $F:(\mathbb
     C^{n+1},0)\to (\mathbb C^{p+1},0)$ where $g_{t}:(\mathbb
     C^{p},0)\to (\mathbb C^{p+1},0)$ are germs of immersions, when
     $(n,p)$ is in the boundary of extra-nice dimensions. These pair of
     extra-nice dimensions have  been calculated in \cite{AtiRuaSin}.
     \begin{enumerate}
     \item [(i)]\ $n\leq p,\ \ 4p=5n-5,\ p\geq 5.$
     \item [(ii)] $n>p,\ \ (n,p)=\{(5,4), (7,5), (9+k,6),\, k\geq 0\}. $ 
     \end{enumerate}

  Observe that these families are always topologically
  trivial. However the Whitney equisingularity and the bi-Lipschitz
  triviality of these families are open questions. 

  \begin{conjecture}
  At the boundary of the extra-nice dimensions any two generic
  immersions $g, g':(\mathbb C^{p},0)\to (\mathbb C^{p+1},0)$ are
  bi-Lipschitz $\mathcal{K}_{V}$-equivalent and they define bi-Lipschitz
  $\mathcal{A}$-equivalent germs $f,\,f':(\mathbb C^{n},0)\to (\mathbb
  C^{p},0).$
 \end{conjecture}

\noindent\underline{Problem 3}. Apply the geometric approach discussed
in this section to study the bi-Lipschitz $\mathcal{G}$-classification
of analytic map-germs $f\in \mathcal{O}_{n}^{p}$ where
$\mathcal{G}=\mathcal{R}, \mathcal{C}, \mathcal{K}, \mathcal{L},
\mathcal{A}$ or more generally, any geometric subgroup of
$\mathcal{K}.$ The Lipschitz theory of singularity is an almost
completely open problem. See \cite{Rua20} for an account on
bi-Lipschitz $\mathcal{G}$-classification of function germs
$\mathcal{G}=\mathcal{R}, \mathcal{C}, \mathcal{K}$ and references
therein \cite{BirCosFerRua,RuaVal,HenPar,NguRuaTri,
  BirFerGraGaf,BirFerGraGab, FerRua04, KoiPar10, KoiPar13,FukSatKuo}.

\section*{Acknowledgments}
\label{sec:acknowledgment}

This research was partially supported by CNPq grant \# 433257/2016/4
and FAPESP, grant \# 2019/18226-1.

I am grateful to Jose  Seade for the invitation to publish this work as a chapter of
the “Handbook of Geometry and Topology of Singularities.”

Special thanks are due to David Trotman, David Mond and Raul Oset-Sinha for
their valuable comments on preliminary versions of this paper, and to
D\' ebora Lopes for the help with the figures.

It is a pleasure to thank Nhan Nguyen ans Saurabh Trivedi for helpful discussions
on Lipschitz stability of smooth mappings. I amn also grateful to the referee for his
careful reading and detailed remarks that improved very much the presentation of
these notes.

\bibliography{sing}
\bibliographystyle{spmpsci}
\printindex
\end{document}

%% file: Fig1.pdf_tex
\begingroup%
  \makeatletter%
  \providecommand\color[2][]{%
    \errmessage{(Inkscape) Color is used for the text in Inkscape, but the package 'color.sty' is not loaded}%
    \renewcommand\color[2][]{}%
  }%
  \providecommand\transparent[1]{%
    \errmessage{(Inkscape) Transparency is used (non-zero) for the text in Inkscape, but the package 'transparent.sty' is not loaded}%
    \renewcommand\transparent[1]{}%
  }%
  \providecommand\rotatebox[2]{#2}%
  \newcommand*\fsize{\dimexpr\f@size pt\relax}%
  \newcommand*\lineheight[1]{\fontsize{\fsize}{#1\fsize}\selectfont}%
  \ifx\svgwidth\undefined%
    \setlength{\unitlength}{576.15964097bp}%
    \ifx\svgscale\undefined%
      \relax%
    \else%
      \setlength{\unitlength}{\unitlength * \real{\svgscale}}%
    \fi%
  \else%
    \setlength{\unitlength}{\svgwidth}%
  \fi%
  \global\let\svgwidth\undefined%
  \global\let\svgscale\undefined%
  \makeatother%
  \begin{picture}(1,0.91792955)%
    \lineheight{1}%
    \setlength\tabcolsep{0pt}%
    \put(0,0){\includegraphics[width=\unitlength,page=1]{Fig1.pdf}}%
    \put(0.57014892,0.81010212){\color[rgb]{0,0,0}\makebox(0,0)[lt]{\lineheight{1.25}\smash{\begin{tabular}[t]{l}$6p=7n-8$\end{tabular}}}}%
    \put(0.8072167,0.69087496){\color[rgb]{0,0,0}\makebox(0,0)[lt]{\lineheight{1.25}\smash{\begin{tabular}[t]{l}$n=p$\end{tabular}}}}%
    \put(0.19665278,0.90190211){\color[rgb]{0,0,0}\makebox(0,0)[lt]{\lineheight{1.25}\smash{\begin{tabular}[t]{l}$p$\end{tabular}}}}%
    \put(0.12892816,0.75672144){\color[rgb]{0,0,0}\makebox(0,0)[lt]{\lineheight{1.25}\smash{\begin{tabular}[t]{l}$36$\end{tabular}}}}%
    \put(0.128111,0.63295233){\color[rgb]{0,0,0}\makebox(0,0)[lt]{\lineheight{1.25}\smash{\begin{tabular}[t]{l}$30$\end{tabular}}}}%
    \put(0.27362003,0.43140242){\color[rgb]{0,0,0}\makebox(0,0)[lt]{\lineheight{1.25}\smash{\begin{tabular}[t]{l}$6p=7n-9$\end{tabular}}}}%
    \put(0.13979814,0.23424239){\color[rgb]{0,0,0}\makebox(0,0)[lt]{\lineheight{1.25}\smash{\begin{tabular}[t]{l}$9$\end{tabular}}}}%
    \put(0.13934771,0.19784893){\color[rgb]{0,0,0}\makebox(0,0)[lt]{\lineheight{1.25}\smash{\begin{tabular}[t]{l}$8$\end{tabular}}}}%
    \put(0.1378087,0.16466958){\color[rgb]{0,0,0}\makebox(0,0)[lt]{\lineheight{1.25}\smash{\begin{tabular}[t]{l}$7$\end{tabular}}}}%
    \put(0.14071927,0.13010411){\color[rgb]{0,0,0}\makebox(0,0)[lt]{\lineheight{1.25}\smash{\begin{tabular}[t]{l}$6$\end{tabular}}}}%
    \put(0.34257448,0.01397907){\color[rgb]{0,0,0}\makebox(0,0)[lt]{\lineheight{1.25}\smash{\begin{tabular}[t]{l}$9$\end{tabular}}}}%
    \put(0.31542005,0.01511993){\color[rgb]{0,0,0}\makebox(0,0)[lt]{\lineheight{1.25}\smash{\begin{tabular}[t]{l}$8$\end{tabular}}}}%
    \put(0.6374432,0.01077103){\color[rgb]{0,0,0}\makebox(0,0)[lt]{\lineheight{1.25}\smash{\begin{tabular}[t]{l}$27$\end{tabular}}}}%
    \put(0.73924855,0.01028539){\color[rgb]{0,0,0}\makebox(0,0)[lt]{\lineheight{1.25}\smash{\begin{tabular}[t]{l}$32$\end{tabular}}}}%
    \put(0.80884703,0.23852869){\color[rgb]{0,0,0}\makebox(0,0)[lt]{\lineheight{1.25}\smash{\begin{tabular}[t]{l}$p=7$\end{tabular}}}}%
    \put(0.87582589,0.07331946){\color[rgb]{0,0,0}\makebox(0,0)[lt]{\lineheight{1.25}\smash{\begin{tabular}[t]{l}$n$\end{tabular}}}}%
    \put(0,0){\includegraphics[width=\unitlength,page=2]{Fig1.pdf}}%
    \put(-0.00188139,0.49829847){\color[rgb]{0,0,0}\makebox(0,0)[lt]{\lineheight{1.25}\smash{\begin{tabular}[t]{l}nice dimensions\end{tabular}}}}%
    \put(0.36967325,0.01397554){\color[rgb]{0,0,0}\makebox(0,0)[lt]{\lineheight{1.25}\smash{\begin{tabular}[t]{l}$10$\end{tabular}}}}%
    \put(0,0){\includegraphics[width=\unitlength,page=3]{Fig1.pdf}}%
  \end{picture}%
\endgroup%